\documentclass[13pt,reqno]{amsart}
\UseRawInputEncoding

\usepackage{amsmath,amssymb,amsthm}
\usepackage{bm}
\usepackage{graphicx}
\usepackage{color}
\usepackage{float}
\usepackage{url}
\usepackage{stmaryrd}
\usepackage{mathrsfs}
\usepackage{caption}
\usepackage{url}
\newcommand{\bel}[1]{\begin{equation*}\label{#1}}
	
	\newcommand{\be}{\begin{equation}}

		\newcommand{\ba}{\begin{eqnarray}}
			\newcommand{\ea}{\end{eqnarray}}

		\newcommand{\qe}{\end{equation}}
	\newcommand{\R}{{\mathbb R}}
	
	\newcommand{\Z}{{\mathbb Z}}
	\newcommand{\C}{{\mathbb C}}
	\newcommand{\T}{{\mathbb T}}
	\newcommand{\supp}{{\mathrm{supp}}}

	\newcommand{\eg}{\begin{example}}
		\newcommand{\egd}{\end{example}}
	\newcommand{\tm}{\begin{thm}}
		\newcommand{\tmd}{\end{thm}}
	\newcommand{\co}{\begin{coro}}
		\newcommand{\cod}{\end{coro}}
	\newcommand{\enu}{\begin{enumerate}}
		\newcommand{\enud}{\end{enumerate}}
	\newcommand{\rmk}{\begin{rem}}
		\newcommand{\rmkd}{\end{rem}}
	
	\theoremstyle{plain}
	\newtheorem{thm}{Theorem}[section]
	\newtheorem{prop}[thm]{Proposition}
	\theoremstyle{plain}
	\newtheorem{example}[thm]{Example}
	\newtheorem{coro}[thm]{Corollary}
	\theoremstyle{plain}
	\newtheorem{lemma}[thm]{Lemma}
	\theoremstyle{definition}
	
	\theoremstyle{plain}
	
	\theoremstyle{remark}
	\newtheorem{rem}[thm]{Remark}
	\theoremstyle{remark}
	
	\theoremstyle{plain}
	\newtheorem{conj}[thm]{Conjecture}
	
	\UseRawInputEncoding

\begin{document}
		
		\title[Shell-Type Decoupling and Periodic Strichartz Estimates]{Sharp Shell-Type Bilinear Decoupling and Improved Strichartz Estimates on Tori}
		\author{Jiajun Wang}
		\address{Jiajun Wang: Courant Institute of Mathematical Sciences,
			New York University, New York, NY}
		\email{jw9409@nyu.edu}
		
		\numberwithin{equation}{section}
		
		\begin{abstract}
		We establish sharp bilinear decoupling under shell-type Fourier support
		restrictions and show that its optimal coefficient need not be optimal
		for the corresponding periodic bilinear Strichartz estimate.
		For $N_1\ge N_2$, the sharp decoupling coefficient is
		$N_2^{(2d-7)/4}(1+N_2^2/N_1)^{1/4}$,
		up to an $N_2^\varepsilon$ loss.
		This yields a shell-type bilinear Strichartz estimate with the same
		coefficient on every fixed rectangular torus, including irrational tori.
		On the standard square torus, we improve this Strichartz estimate
		for $d\ge6$, obtaining the optimal coefficient
		$(N_2^{d-4}+N_2^{d-2}/N_1)^{1/2}$,
		up to an $N_2^\varepsilon$ loss.
		The proof combines localization in spatial frequency and energy,
		a circle-method estimate for the localized high-frequency factor,
		and a linear estimate on arbitrarily oriented strips.
			
			As applications, we obtain estimates on toral eigenfunctions, mixed
			additive energy, and separated nonlinear smoothing for the periodic Zakharov system.
		\end{abstract}
		\maketitle
		\vspace{-0.7cm}

		\section{introduction}
		Decoupling inequalities have become a fundamental tool in modern harmonic
		analysis, with applications to Fourier restriction, dispersive equations,
		exponential sums, and geometric measure theory. The subject has its origins
		in the work of Wolff \cite{Wolff2000}, who introduced a plate decomposition
		inequality for the cone in connection with the local smoothing problem for
		the wave equation. This circle of ideas was subsequently developed further in
		\cite{LabaWolff2002}. For related developments during this early stage of
		decoupling theory, we refer to
		\cite{GarrigosSeeger2009,GarrigosSeeger2010}.
		
		A major breakthrough was achieved by Bourgain and Demeter
		\cite{1}, who proved the following sharp $\ell^2$ decoupling theorem,
		up to an $N^\epsilon$ loss, for compact hypersurfaces with positive definite
		second fundamental form.
		
		\begin{itemize}
			\item 
			
			Let $R\gg 1$ and $d\ge 1$, and suppose that $f:\R^{d}\to \C$ is supported on 
			\begin{equation*}
				\supp\; f\subseteq \{\xi\in \R^{d}: |\xi|\le 1\}.
			\end{equation*}
			Then, for $2\le p\le \infty$ and any $\epsilon>0$, the following sharp estimate holds:
			\begin{equation}\label{BD}
				\|E_{\mathbb{P}^d}f\|_{L^p(B_R^{d+1})}\lesssim_{\epsilon}R^{\epsilon} M_d(R) \left( \sum_{|\theta|=R^{-1/2}}\|E_{\mathbb{P}^d}f_{\theta}\|_{L^p(B_R^{d+1})}^2\right)^{1/2},
			\end{equation}
			\begin{equation*}
				(E_{\mathbb{P}^d}f)(x,t):=\int_{\R^d} e^{2\pi i(x\cdot\xi+t|\xi|^2)}f(\xi)d\xi,\qquad \mathbb{P}^d:=\lbrace(\xi,|\xi|^2): |\xi|\le 1 \rbrace,
			\end{equation*}
			where $f_{\theta}:=f\chi_{\theta}$ and $M_d(R)$ is given by 
			\begin{equation*}
				M_d(R):=
				\begin{cases}
					R^{\left( \frac{d}{4}-\frac{d+2}{2p}\right)}, & \dfrac{2(d+2)}{d}\le p\le \infty,\\
					1, & 2\le p<\dfrac{2(d+2)}{d}.
				\end{cases}
			\end{equation*}
		\end{itemize}

		Their theorem established
		the canonical framework of $\ell^2$ decoupling theory and led to a wide range of
		applications, including discrete restriction and periodic Strichartz
		estimates. 
		
		One of the most striking applications of decoupling theory is its connection
		with exponential sums and Diophantine equations. Bourgain, Demeter, and Guth
		\cite{BourgainDemeterGuth2016} established the sharp decoupling theorem for
		the moment curve and used it to prove the main conjecture in Vinogradov's
		mean value theorem. This work initiated the development of decoupling for
		higher-dimensional moment manifolds and systems of Diophantine equations.
		For further developments in this direction, we refer to
		\cite{BourgainDemeterGuo2017,GuoZorinKranich2020}.

		An important refinement of the classical theory is to incorporate
		additional information on the spatial distribution of the associated wave
		packets. Refined decoupling estimates, developed in
		particular in the work of Guth-Iosevich-Ou-Wang
		\cite{guth2020falconer} and Du-Zhang \cite{DuZhang2019}, yield improved
		estimates when the wave packets satisfy suitable concentration or
		multiplicity conditions.
		
		These developments illustrate a general principle that has become
		increasingly important in decoupling theory: additional geometric information,
		either on the frequency support or on the spatial distribution of wave
		packets, can lead to estimates that are substantially stronger than the
		corresponding canonical decoupling inequality. In the present paper, we
		pursue this principle in a bilinear setting, where the Fourier supports are
		subject to additional shell-type restrictions.

		We first recall the linear shell-type decoupling estimate obtained in \cite{kinoshita2023decoupling}. 
		\begin{thm}
			Let $R\gg 1$, $d\ge 2$, and suppose that $f:\R^{d}\to \C$ is supported on 
			\begin{equation*}
				\supp\; f\subseteq \{\xi\in \R^{d}: 1-\frac{1}{R^{1/2}}\le |\xi|\le 1+\frac{1}{R^{1/2}}\},
			\end{equation*}
			then for $2\le p\le \infty$, we have the following sharp estimate:
			\begin{equation}\label{linear}
				\|E_{\mathbb{P}^d}f\|_{L^p(B_R^{d+1})}\lessapprox C_d(R) \left( \sum_{|\theta|=R^{-1/2}}\|E_{\mathbb{P}^d}f_{\theta}\|_{L^p(B_R^{d+1})}^2\right)^{1/2},
			\end{equation}
			where $f_{\theta}:=f\chi_{\theta}$ and $C_d(R)$ is given by 
			\begin{equation*}
				C_d(R):=
				\begin{cases}
					R^{\left( \frac{d-1}{4}-\frac{d+1}{2p}\right)}, & \dfrac{2(d+1)}{d-1}\le p\le \infty,\\
					1, & 2\le p<\dfrac{2(d+1)}{d-1}.
				\end{cases}
			\end{equation*}
		\end{thm}
		
		This linear shell-type decoupling estimate can be applied to obtain improved multilinear estimates for the Zakharov system. We refer to \cite{kinoshita2023decoupling} for further details.
		
		Inspired by the above shell-type structure, we aim to establish an improved bilinear estimate, formulated as follows (for convenience, we abbreviate $E_{\mathbb{P}^d}f$ as $Ef$).

		\begin{conj}\label{c}
			Given \(N_1 \geq N_2 \geq 1\), let \(f_1:\R^d\to \C\) be supported on  
			\begin{equation}\label{region1}
				\lbrace \xi\in\R^d: 1-\frac{1}{N_1}\le |\xi|\le 1+\frac{1}{N_1}\rbrace,
			\end{equation}
			and let \(f_2:\R^d\to \C\) be supported on 
			\begin{equation*}
				\lbrace \xi\in\R^d: \frac{N_2}{N_1}-\frac{1}{N_1}\le |\xi|\le \frac{N_2}{N_1}+\frac{1}{N_1}\rbrace.
			\end{equation*}
			For a finitely overlapping covering of the ball \(B := \{|\xi| \leq 1\}\) by caps \(\{\theta\}\), \(|\theta| = 1/N_1\), we have the following estimates. For any small \(\epsilon > 0\), when \(2\le d \le 3\),
			\begin{equation}\label{trivial}
				\|E f_1 E f_2\|_{L^2(B_{N_1^2})} \lesssim_{\epsilon} N_2^{\epsilon} \prod_{j=1}^2 \left( \sum_{|\theta|=1/N_1} \|E f_{j, \theta}\|_{L^4(B_{N_1^2})}^2 \right)^{1/2},
			\end{equation}
			and when \(d \geq 4\),
			\begin{equation}\label{bilinear estimate}
				\|E f_1 E f_2\|_{L^2(B_{N_1^2})}
				\lesssim_\epsilon
				N_2^{\epsilon}N_2^{\frac{2d-7}{4}}
				\left(1+\frac{N_2^2}{N_1}\right)^{1/4}
				\prod_{j=1}^2 \left( \sum_{|\theta|=1/N_1} \|E f_{j, \theta}\|_{L^4(B_{N_1^2})}^2 \right)^{1/2}.
			\end{equation}
		\end{conj}
		\begin{rem}\label{comparison}
			The shell-type bilinear estimate (\ref{bilinear estimate}) can be viewed as an improvement over its linear analogue (\ref{linear}) when $p=4$. In fact, by directly applying $(\ref{linear})$ and H\"older's inequality, one can only derive
			\begin{equation}\label{Holder}
				\|E f_1 E f_2\|_{L^2(B_{N_1^2})} \lessapprox \max\lbrace 1, (N_1 N_2)^{\frac{d-3}{4}}\rbrace \prod_{j=1}^2 \left( \sum_{|\theta|=1/ {N_1}} \|E f_{j, \theta}\|_{L^4(B_{N_1^2})}^2 \right)^{1/2}.
			\end{equation}
			Clearly, our factor is strictly better when $N_2\ll N_1$, which will play an important role in controlling frequency interactions in the Zakharov system.
			
			Also note that the estimate (\ref{trivial}) provides no significant improvement over the trivial estimate (\ref{Holder}). Thus, we focus on the genuinely shell-improved range $d\ge4$.
		\end{rem}
		\begin{rem}\label{Sharpness}
			The bilinear decoupling inequality (\ref{bilinear estimate}) is in fact sharp up to a factor $N_2^{\epsilon}$. We present the corresponding examples in Section~\ref{sec:decoupling-sharpness}. 
		\end{rem}

		One of our main results can be stated as follows:
		\begin{thm}\label{main bilinear}
			Conjecture \ref{c} is true.
		\end{thm}

	For periodic estimates, let
	$\boldsymbol\ell:=(\ell_1,\ldots,\ell_d)\in(0,\infty)^d$ and define
	\begin{equation*}
		\mathbb T_{\boldsymbol\ell}^d
		:=\prod_{j=1}^d(\mathbb R/\ell_j\mathbb Z),
		\qquad
		\Lambda_{\boldsymbol\ell}
		:=\prod_{j=1}^d\ell_j^{-1}\mathbb Z.
	\end{equation*}
	When $\boldsymbol\ell=(1,\ldots,1)$, we write
	$\mathbb T^d:=(\mathbb R/\mathbb Z)^d$ for the standard unit square
	torus. We call
	$\mathbb T_{\boldsymbol\ell}^d$ rational if
	$\ell_j/\ell_1\in\mathbb Q$ for every $j=1,\ldots,d$,
	and irrational otherwise.

		As a consequence of Theorem~\ref{main bilinear}, we obtain the following bilinear Strichartz estimate on general rectangular tori:
		\begin{coro}\label{bilinear Strichartz}
			Fix a rectangular torus $\mathbb T_{\boldsymbol\ell}^d$. Let $d\ge 4$, $N_1\ge N_2\gg 1$, and $\phi_1,\phi_2\in L^2(\mathbb T_{\boldsymbol\ell}^d)$ such that
			\begin{equation*}
				\supp \widehat{\phi_1}\subseteq \lbrace k\in\Lambda_{\boldsymbol\ell}: N_1-O_d(1)\le |k|\le N_1+O_d(1)\rbrace,
			\end{equation*}
			\begin{equation*}
				\supp \widehat{\phi_2}\subseteq \lbrace k\in\Lambda_{\boldsymbol\ell}: N_2-O_d(1)\le |k|\le N_2+O_d(1)\rbrace.
			\end{equation*}
			Then, for any $\epsilon>0$, we have the following estimate:
			\begin{equation}\label{bilinear Strichartz true}
				\|e^{it\Delta}\phi_1 e^{it\Delta}\phi_2\|_{L_{t,x}^2([0,1]\times\mathbb T_{\boldsymbol\ell}^d)}\lesssim_{\epsilon,\boldsymbol\ell} N_2^{\epsilon}N_2^{\frac{2d-7}{4}}\left(1+\frac{N_2^2}{N_1}\right)^{1/4} \|\phi_1\|_{L^2(\mathbb T_{\boldsymbol\ell}^d)}\|\phi_2\|_{L^2(\mathbb T_{\boldsymbol\ell}^d)}.
			\end{equation}
		\end{coro}
		\begin{proof}
			We first give the proof on the unit square torus. The proof follows the argument originating in \cite{bourgain2013moment,1}. We also refer to \cite[Section 2]{fan2018bilinear} for the bilinear version.
			
			By decomposing each of the two frequency shells into \(O_d(1)\)
			subshells, and replacing \(N_j\) by \(N_j+O_d(1)\), it suffices to
			consider the case
			\begin{equation*}
				\operatorname{supp}\widehat{\phi_j}
				\subset
				\left\{
				k\in\mathbb Z^d:
				\bigl||k|-N_j\bigr|\leq 1
				\right\},
				\qquad j=1,2.
			\end{equation*}
			Since \(N_2\gg1\), this finite decomposition does not affect the
			claimed estimate.
			
			We identify \(\mathbb T^d\) with \((\mathbb R/\mathbb Z)^d\) and use
			the Fourier series
			\begin{equation*}
				\phi_j(x)
				=
				\sum_{k\in\mathbb Z^d}
				\widehat{\phi_j}(k)e^{2\pi i k\cdot x}.
			\end{equation*}
			Thus,
			\begin{equation*}
				e^{it\Delta}\phi_j(x)
				=
				\sum_{k\in\mathbb Z^d}
				\widehat{\phi_j}(k)
				e^{2\pi i k\cdot x-it|k|^2}.
			\end{equation*}
			
			Define the discrete measures
			\begin{equation*}
				h_j
				:=
				\sum_{k\in\mathbb Z^d}
				\widehat{\phi_j}(k)\,
				\delta_{k/N_1},
				\qquad j=1,2,
			\end{equation*}
			and, for a finite measure \(h\) on \(\mathbb R^d\), define
			\begin{equation*}
				Eh(X,T)
				:=
				\int_{\mathbb R^d}
				e^{2\pi i(X\cdot\xi+T|\xi|^2)}
				\,dh(\xi).
			\end{equation*}
			Then we have
			\begin{equation}\label{2.1}
				e^{it\Delta}\phi_j(x)
				=
				Eh_j(N_1x,-N_1^2t/(2\pi)).
			\end{equation}
			
			Strictly speaking, the assumed shell-type decoupling estimate is
			stated for functions rather than measures. One may replace each
			Dirac mass above by an \(L^1\)-normalized smooth bump of radius
			\(o(N_1^{-2})\), apply the estimate uniformly, and then pass to the
			limit. We shall use this standard approximation without further
			comment.
			
			The Fourier support assumptions imply
			\begin{equation*}
				\operatorname{supp}h_1
				\subset
				\left\{
				\xi\in\mathbb R^d:
				1-\frac1{N_1}
				\leq|\xi|
				\leq
				1+\frac1{N_1}
				\right\},
			\end{equation*}
			and
			\begin{equation*}
				\operatorname{supp}h_2
				\subset
				\left\{
				\xi\in\mathbb R^d:
				\frac{N_2}{N_1}-\frac1{N_1}
				\leq|\xi|
				\leq
				\frac{N_2}{N_1}+\frac1{N_1}
				\right\}.
			\end{equation*}
			Hence \(h_1,h_2\) satisfy the support assumptions in
			Conjecture~\ref{c}.
			
			Set
			\begin{equation*}
				I:=[-N_1^2/(2\pi),0],
				\qquad
				Q_0:=I\times[0,N_1]^d.
			\end{equation*}
			The change of variables
			\begin{equation*}
				X:=N_1x,
				\qquad
				T:=-N_1^2t/(2\pi)
			\end{equation*}
			in \eqref{2.1} gives
			\begin{equation}\label{2.2}
				\left\|
				e^{it\Delta}\phi_1
				e^{it\Delta}\phi_2
				\right\|_{L^2([0,1]\times\mathbb T^d)}
				\sim
				N_1^{-\frac{d+2}{2}}
				\left\|
				Eh_1Eh_2
				\right\|_{L^2(Q_0)}.
			\end{equation}
			
			For each \(m\in\mathbb Z^d\), the functions \(Eh_j\) satisfy
			\begin{equation*}
				Eh_j(X+N_1m,T)=Eh_j(X,T),
			\end{equation*}
			since the frequencies of \(h_j\) belong to \(N_1^{-1}\mathbb Z^d\).
			Let
			\begin{equation*}
				L:=\lfloor N_1\rfloor
			\end{equation*}
			and define
			\begin{equation*}
				\Omega
				:=
				I\times[0,LN_1]^d.
			\end{equation*}
			Since \(\Omega\) consists of \(L^d\) spatial translates of \(Q_0\),
			the spatial periodicity gives
			\begin{equation}\label{2.3}
				\left\|
				Eh_1Eh_2
				\right\|_{L^2(\Omega)}
				=
				L^{d/2}
				\left\|
				Eh_1Eh_2
				\right\|_{L^2(Q_0)}.
			\end{equation}
			
			The diameter of \(\Omega\) is \(O_d(N_1^2)\). We may therefore cover
			\(\Omega\) by \(O_d(1)\) translated $N_1^2$-balls
			\begin{equation*}
				\Omega\subset\bigcup_{\alpha=1}^{C_d}B_\alpha.
			\end{equation*}
			By translation invariance, Conjecture~\ref{c} applies on every
			\(B_\alpha\), and hence
			\begin{equation}\label{2.4}
				\|Eh_1Eh_2\|_{L^2(B_\alpha)}
				\lesssim_{\epsilon}
				N_2^\epsilon
				N_2^{\frac{2d-7}{4}}\left(1+\frac{N_2^2}{N_1}\right)^{1/4}
				\prod_{j=1}^2
				\left(
				\sum_{|\theta|=1/N_1}
				\|Eh_{j,\theta}\|_{L^4(B_\alpha)}^2
				\right)^{1/2}.
			\end{equation}
			
			The points of \(N_1^{-1}\mathbb Z^d\) are \(N_1^{-1}\)-separated.
			Consequently, every cap \(\theta\) contains at most \(O_d(1)\) points
			of the form \(k/N_1\). Therefore, by Cauchy--Schwarz,
			\begin{equation*}
				|Eh_{j,\theta}(X,T)|
				\lesssim_d
				\left(
				\sum_{\substack{k\in\mathbb Z^d\\ k/N_1\in\theta}}
				|\widehat{\phi_j}(k)|^2
				\right)^{1/2}.
			\end{equation*}
			It follows that
			\begin{equation*}
				\sum_{\theta}
				\|Eh_{j,\theta}\|_{L^4(B_\alpha)}^2
				\lesssim_d
				|B_\alpha|^{1/2}
				\sum_{\theta}
				\sum_{\substack{k\in\mathbb Z^d\\ k/N_1\in\theta}}
				|\widehat{\phi_j}(k)|^2
				\lesssim_d
				|B_\alpha|^{1/2}
				\sum_{k\in\mathbb Z^d}
				|\widehat{\phi_j}(k)|^2.
			\end{equation*}
			Since $|B_\alpha|\sim N_1^{2(d+1)}$, by Parseval,
			\begin{equation*}
				\sum_{k\in\mathbb Z^d}
				|\widehat{\phi_j}(k)|^2
				=
				\|\phi_j\|_{L^2(\mathbb T^d)}^2,
			\end{equation*}
			we obtain
			\begin{equation}\label{2.5}
				\left(
				\sum_{\theta}
				\|Eh_{j,\theta}\|_{L^4(B_\alpha)}^2
				\right)^{1/2}
				\lesssim_d
				N_1^{\frac{d+1}{2}}
				\|\phi_j\|_{L^2(\mathbb T^d)}.
			\end{equation}
			
			Substituting \eqref{2.5} into \eqref{2.4} and then summing over the
			\(O_d(1)\) balls \(B_\alpha\), we obtain
			\begin{equation}\label{2.6}
				\|Eh_1Eh_2\|_{L^2(\Omega)}
				\lesssim_{\epsilon}
				N_1^{d+1}
				N_2^\epsilon
				N_2^{\frac{2d-7}{4}}\left(1+\frac{N_2^2}{N_1}\right)^{1/4}
				\|\phi_1\|_{L^2(\mathbb T^d)}
				\|\phi_2\|_{L^2(\mathbb T^d)}.
			\end{equation}
			
			Combining \eqref{2.2}, \eqref{2.3}, and \eqref{2.6}, we conclude that
			\begin{equation*}
				\left\|
				e^{it\Delta}\phi_1
				e^{it\Delta}\phi_2
				\right\|_{L^2([0,1]\times\mathbb T^d)}
				\lesssim_{\epsilon}
				N_1^{-\frac{d+2}{2}}
				L^{-\frac d2}
				N_1^{d+1}
				N_2^\epsilon
				N_2^{\frac{2d-7}{4}}\left(1+\frac{N_2^2}{N_1}\right)^{1/4}
				\|\phi_1\|_{L^2(\mathbb T^d)}
				\|\phi_2\|_{L^2(\mathbb T^d)}.
			\end{equation*}
			Since \(L\sim N_1\),
			\begin{equation*}
				N_1^{-\frac{d+2}{2}}
				L^{-\frac d2}
				N_1^{d+1}
				\sim 1.
			\end{equation*}
			Therefore,
			\begin{equation*}
				\left\|
				e^{it\Delta}\phi_1
				e^{it\Delta}\phi_2
				\right\|_{L^2([0,1]\times\mathbb T^d)}
				\lesssim_{\epsilon}
				N_2^\epsilon
				N_2^{\frac{2d-7}{4}}\left(1+\frac{N_2^2}{N_1}\right)^{1/4}
				\|\phi_1\|_{L^2(\mathbb T^d)}
				\|\phi_2\|_{L^2(\mathbb T^d)}.
			\end{equation*}
			For a general fixed rectangular torus, replace the frequency lattice
			$\mathbb Z^d$ by $\Lambda_{\boldsymbol\ell}$.
			The relevant sums are now periodic in $X_j$ with period $N_1\ell_j$.
			 Each $N_1^{-1}$-cap contains at most
			$O_{\boldsymbol\ell,d}(1)$ points of $N_1^{-1}\Lambda_{\boldsymbol\ell}$.
			The cap estimate and coefficient Plancherel identity are therefore
			unchanged apart from constants depending on the fixed torus.
			No time periodicity, rationality, or integrality of $|k|^2$ is used in
			this argument. This proves Corollary~\ref{bilinear Strichartz}.

		\end{proof}
		
		The discrete estimate obtained from decoupling need not inherit the
		sharpness. We now formulate the
		stronger estimate suggested by the shell geometry on the standard
		square torus. The distinction between this torus and a general
		rectangular torus is essential in our proof.
		
		\begin{conj}[Shell-type bilinear Strichartz estimate]
			\label{conj:shell-strichartz}
			Let $d\ge5$, $N_1\ge N_2\ge1$, and $C_0>0$ be fixed. On the standard
			square torus $\mathbb T^d=(\mathbb R/\mathbb Z)^d$, suppose that
			\begin{equation*}
				\operatorname{supp}\widehat\phi_j
				\subset\{k\in\mathbb Z^d:\bigl||k|-N_j\bigr|\le O_d(1)\},
				\qquad j=1,2.
			\end{equation*}
			Then, for every $\varepsilon>0$,
			\begin{equation}\label{eq:conjectured-shell-strichartz}
				\|e^{it\Delta}\phi_1\,e^{it\Delta}\phi_2\|_{L^2([0,1]\times\mathbb T^d)}
				\lesssim_{\varepsilon}N_2^\varepsilon
				\left(N_2^{d-4}+\frac{N_2^{d-2}}{N_1}\right)^{1/2}
				\|\phi_1\|_{L^2(\mathbb T^d)}\|\phi_2\|_{L^2(\mathbb T^d)}.
			\end{equation}
		\end{conj}
		\begin{rem}
			In fact, the coefficient in \eqref{eq:conjectured-shell-strichartz} is sharp up to
			the factor \(N_2^\varepsilon\). The corresponding examples
			are given in Section~\ref{sec:strichartz-sharpness}.
			In contrast, when \(d=4\) and \(N_1\geq N_2^2\), the optimal
			coefficient is \(N_2^{1/4}\), up to \(N_2^\varepsilon\),
			so the specialization of
			\eqref{eq:conjectured-shell-strichartz} to \(d=4\) fails.
			A lattice example establishing this lower bound is also
			presented in Section~\ref{sec:strichartz-sharpness}.
		\end{rem}
		
		\begin{thm}\label{thm:improved-bilinear-strichartz}
			Conjecture~\ref{conj:shell-strichartz} holds for every $d\ge6$.
		\end{thm}
		
		The proof is organized in Section~\ref{sec:framework} and completed in
		Section~\ref{sec:strichartz-proof}. It uses the integer 
		energies of the square torus, together with localized linear shell
		decoupling. It is not a consequence of an improvement of
		Theorem~\ref{main bilinear}.
		Indeed, the ratio of the sharp decoupling coefficient to the new
		Strichartz coefficient is
		\begin{equation*}
			\frac{N_2^{(2d-7)/4}(1+N_2^2/N_1)^{1/4}}
			{(N_2^{d-4}+N_2^{d-2}/N_1)^{1/2}}
			=\left(\frac{N_2}{1+N_2^2/N_1}\right)^{1/4}.
		\end{equation*}
		In particular, when $N_1\ge N_2^2$, the new estimate gains a factor
		comparable to $N_2^{1/4}$ over the bound obtained from decoupling.
		
		The scope of the two results is different. Corollary~\ref{bilinear Strichartz}
		applies to every fixed rectangular torus, including irrational ones.
		Theorem~\ref{thm:improved-bilinear-strichartz} is proved here on the
		standard square torus, which is rational. We do not assert its extension
		to arbitrary rational rectangular tori or to irrational tori.
		Accordingly, Section~\ref{sec:applications} retains the applications
		valid on general rectangular tori and records the square-torus
		improvements separately.
		
		   \begin{rem}\label{rem:decoupling-strichartz-gap}
		   	The difference between the two estimates can be understood in terms
		   	of the interaction between the shell geometry and the frequency
		   	lattice. After rescaling, the support condition is a fixed-width
		   	neighborhood of a sphere of radius $N_j$. Unlike a thick annulus,
		   	this region retains information about the underlying sphere. 
		   	
		   	The distinction is particularly visible at the $L^4$-estimate. On the square torus, integration over a full space-time
		   	period selects the relations
		   	\begin{equation*}
		   		k_1+k_2=k_3+k_4,
		   		\qquad
		   		|k_1|^2+|k_2|^2=|k_3|^2+|k_4|^2.
		   	\end{equation*}
		   	Thus the relevant information concerns not only the number of
		   	frequencies in a shell, but also their additive and quadratic
		   	relations. The proof in Section~\ref{sec:strichartz-proof} exploits
		   	the integer energies through Poisson summation and
		   	quadratic exponential sums. General bilinear decoupling,
		   	on the other hand, must accommodate relatively arbitrary 
		   	frequency distributions and cannot impose this structure.
		   	Consequently, a sharp decoupling coefficient need not
		   	remain optimal for the associated periodic estimate.
		   	
		   	For comparison, a thick annulus
		   	$\{\xi\in\mathbb R^d:N\le|\xi|\le2N\}$ contains
		   	$\sim N^d$ points of any fixed rectangular
		   	frequency lattice $\Lambda_{\boldsymbol\ell}$ as $N\gg 1$,
		   	regardless of the rationality of the torus. Its total lattice-point
		   	count is therefore governed by volume. This observation does not
		   	by itself determine a bilinear Strichartz coefficient, but it
		   	highlights the additional structure retained by the shell
		   	restriction. The thick-annulus estimates of
		   	Fan--Staffilani--Wang--Wilson, recalled in the next remark,
		   	provide a natural comparison with both shell-type estimates.
		   \end{rem}

		\begin{rem}
			For comparison, we recall the result in \cite{fan2018bilinear}. In fact, \cite{fan2018bilinear} established the following bilinear decoupling inequality and the corresponding bilinear Strichartz estimate:
			\begin{itemize}
				\item 
				Given \(N_1 \geq N_2 \gg 1\), let \(f_1:\R^d\to \C\) be supported on  
				\begin{equation*}
					\lbrace \xi\in\R^d: |\xi|\sim 1\rbrace,
				\end{equation*}
				and let \(f_2:\R^d\to \C\) be supported on 
				\begin{equation*}
					\lbrace \xi\in\R^d: |\xi|\sim \frac{N_2}{N_1}\rbrace.
				\end{equation*}
				For a finitely overlapping covering of the ball \(B := \{|\xi| \leq 1\}\) by caps \(\{\theta\}\), \(|\theta| = 1/{ N_1}\), we have the following estimates. For any small \(\epsilon > 0\), when \( d=2\),
				\begin{equation*}
					\|E f_1 E f_2\|_{L^2(B_{N_1^2})} \lesssim_{\epsilon} (N_2)^{\epsilon} \prod_{j=1}^2 \left( \sum_{|\theta|=1/{N_1}} \|E f_{j, \theta}\|_{L^4(B_{N_1^2})}^2 \right)^{1/2},
				\end{equation*}
				and when \(d \ge 3\),
				\begin{equation}\label{FSWW}
					\|E f_1 E f_2\|_{L^2(B_{N_1^2})} \lesssim_\epsilon (N_2)^\epsilon \left( N_2^{d-3}+ \frac{N_2^{d-1}}{N_1} \right)^{1/2} \prod_{j=1}^2 \left( \sum_{|\theta|=1/ {N_1}} \|E f_{j, \theta}\|_{L^4(B_{N_1^2})}^2 \right)^{1/2}.
				\end{equation}

				\item 
				Let $d\ge 2$, $N_1\ge N_2\gg 1$, and $\phi_1,\phi_2\in L^2(\T^d)$ such that 
				\begin{equation*}
					\supp \widehat{\phi_1}\subseteq \lbrace k\in\Z^d:  |k|\sim N_1\rbrace,
				\end{equation*}
				\begin{equation*}
					\supp \widehat{\phi_2}\subseteq \lbrace k\in\Z^d:  |k|\sim N_2\rbrace.
				\end{equation*}
				Then, for any $\epsilon>0$, we have the following estimates:
				\begin{equation*}
					\|e^{it\Delta}\phi_1 e^{it\Delta}\phi_2\|_{L_{t,x}^2([0,1]\times\T^d)}\lesssim N_2^{\epsilon} \|\phi_1\|_{L^2(\T^d)}\|\phi_2\|_{L^2(\T^d)}, \quad d=2,
				\end{equation*}
				\begin{equation}\label{82}
					\|e^{it\Delta}\phi_1 e^{it\Delta}\phi_2\|_{L_{t,x}^2([0,1]\times\T^d)}\lesssim N_2^{\epsilon}\left( N_2^{d-3}+ \frac{N_2^{d-1}}{N_1} \right)^{1/2}\|\phi_1\|_{L^2(\T^d)}\|\phi_2\|_{L^2(\T^d)}, \quad d\ge 3.
				\end{equation}
			\end{itemize}
			Comparing \eqref{FSWW} with \eqref{bilinear estimate}, the two coefficients satisfy
			\begin{equation*}
				\frac{
					\left(N_2^{d-3}+\frac{N_2^{d-1}}{N_1}\right)^{1/2}
				}{
					N_2^{\frac{2d-7}{4}}\left(1+\frac{N_2^2}{N_1}\right)^{1/4}
				}
				=
				N_2^{1/4}\left(1+\frac{N_2^2}{N_1}\right)^{1/4}.
			\end{equation*}
			Thus the shell restriction yields an additional gain of at least $N_2^{1/4}$ over the corresponding thick-annulus estimate of Fan--Staffilani--Wang--Wilson.
		\end{rem}
		
		\medskip
		\noindent
		\textbf{Disclosure of AI assistance.}
		The author used ChatGPT-6 Astra during the preparation of this
		paper to assist with developing partial arguments (especially the construction of examples and appendices), checking calculations,
		and revising the exposition and \LaTeX{} code.
		The author takes full responsibility for the mathematical claims,
		proofs, and references presented in the paper.
		
		\vspace{10pt}
		\noindent
		\textbf{Organization.}
		
		Sections~\ref{sec:reductions} and \ref{sec:decoupling-proof}
		reduce and prove the shell-type bilinear decoupling estimate, with the
		general-position argument in Appendix~\ref{sec:general-position}.
		Section~\ref{sec:decoupling-sharpness} proves sharpness of the decoupling
		coefficient. Section~\ref{sec:framework} gives the framework for the
		stronger square-torus Strichartz estimate, and
		Section~\ref{sec:strichartz-proof} supplies its proof. Section~\ref{sec:strichartz-sharpness} presents examples
		demonstrating the sharpness of the square-torus Strichartz estimate.
		Section~\ref{sec:applications} presents the general rectangular-torus
		applications and their square-torus improvements separately.
		Appendix~\ref{sec:lambda-strichartz} extends the square-torus estimate
		to rescaled square tori. Appendix~\ref{sec:euclidean-shell-strichartz}
		proves the Euclidean estimate and explains its relation to the
		large-torus limit.
		Appendix~\ref{sec:fsww-proof} gives a simpler proof of the bilinear
		decoupling estimates of Fan--Staffilani--Wang--Wilson.
		
		\vspace{10pt}
		\noindent
		\textbf{Notation.}
		\begin{itemize}
			\item By $u\in C([0,T]; B) \; ( \mbox{or} \; L^{p}([0,T];B))$ for a Banach space $B,$ we mean that $u$ is a continuous $(\mbox{or} \; L^{p})$ map from $[0,T]$ to $B;$ see \cite[page 301]{evans2022partial}.
			\item By $A\lesssim B$ (resp. $A\sim B$), we mean that there exists a positive constant $C$ such that $A\le CB$ (resp. $C^{-1}B\le A \le C B$). If the constant $C$ depends on $p,$ then we write $A\lesssim_{p}B$ (resp. $A\sim_{p} B$).
			\item By $A\ll B$, we mean that $\frac{A}{B}$ is sufficiently small.
			\item By $A\lessapprox B$, we mean that $A\lesssim_{\epsilon} R^{\epsilon} B$ or $A\lesssim_{\epsilon} N_2^{\epsilon} B$ for any $\epsilon>0$.
		\end{itemize}
		
		\vspace{15pt}
		\section{Further reductions}
		\label{sec:reductions}
		In this section, we perform a series of reductions of the shell-type bilinear decoupling estimate (\ref{bilinear estimate}), all of which preserve equivalence with the original estimate.
		
		First, by applying $L^2$-orthogonality and rotational invariance, we obtain the following reduction ($e_d:=(0,0,\cdots,0, 1)$): 
		\begin{conj}\label{821}
			Given \(N_1 \gg N_2 \geq 1\), let \(f_1:\R^d\to \C\) be supported on  
			\begin{equation}\label{new region}
				\lbrace \xi\in\R^d: 1-\frac{1}{N_1}\le |\xi|\le 1+\frac{1}{N_1}\rbrace\cap B\left(e_d, \frac{N_2}{N_1}\right),
			\end{equation}
			and let \(f_2:\R^d\to \C\) be supported on 
			\begin{equation*}
				\lbrace \xi\in\R^d: \frac{N_2}{N_1}-\frac{1}{N_1}\le |\xi|\le \frac{N_2}{N_1}+\frac{1}{N_1}\rbrace.
			\end{equation*}
			For a finitely overlapping covering of the ball \(B := \{|\xi| \leq 1\}\) by caps \(\{\theta\}\), \(|\theta| = 1/{ N_1}\), we have the following estimate. For any small \(\epsilon > 0\), when \(d \ge 4\),
			\begin{equation*}
				\|E f_1 E f_2\|_{L^2(\omega_{B_{N_1^2}})} \lesssim_\epsilon (N_2)^\epsilon N_2^{\frac{2d-7}{4}}\left(1+\frac{N_2^2}{N_1}\right)^{1/4} \prod_{j=1}^2 \left( \sum_{|\theta|=1/ {N_1}} \|E f_{j, \theta}\|_{L^4(\omega_{B_{N_1^2}})}^2 \right)^{1/2}.
			\end{equation*}
		\end{conj}
		
		\begin{rem}
			The reduction from (\ref{region1}) to (\ref{new region}) is justified by the observation that (\ref{region1}) admits a cover by $\frac{N_2}{N_1}$-balls with finite overlap. Furthermore, since the region is radially symmetric, it suffices to restrict to the smaller region (\ref{new region}).
			
			We also replace $L^2(B_{N_1^2})$ with its weighted version $L^{2}(\omega_{B_{N_1^2}})$. Here we can take 
			\begin{equation*}
				\omega_{B_{N_1^2}}(x):=\omega\left(\frac{x}{N_1^2}\right),
			\end{equation*}
			where $\widehat{\omega}\subseteq B(0,2)$, $\omega(x) \ge1, \forall x\in B(0,1).$ The equivalence follows from a standard reduction argument. We refer to \cite[Remark 5.2]{1} for more details.
			
			We also record a basic property that is useful in establishing estimates involving such weights.
			
			\begin{itemize}
				\item 
				
				Let \( B_R \) be a ball centered at the origin, and let \( \mu_{B_R} \) be a function whose Fourier transform is comparable to \( 1/(m(B_{1/R})) \) on \( B_{1/R} \) and supported in \( B_{2/R} \). Then \( \mu_{B_R} \sim 1 \) on \( B_R \) and decays rapidly outside \( B_R \). Moreover, \( \mu_{B_R}^2 \) is positive, decays faster than any polynomial outside \( B_R \), and its Fourier support lies in \( B_{4/R} \).
				
				We cover the whole space by translates \( B' \) of \( B_R \). Let \( \mu_{B'} \) denote the corresponding translate of \( \mu_{B_R} \), and set \( w_{B_R}(B') := \max_{x\in B'} w_{B_R}(x) \). The following useful estimate holds:	 
				\begin{equation}\label{weight}
					w_{B_R}(x) \le \sum_{B'} w_{B_R}(B') \chi_{B'}(x)
					\le \sum_{B'} w_{B_R}(B') \mu_{B'}^2(x)
					\lesssim w_{B_R}(x). 
				\end{equation}
				
				The last inequality follows from the fact that \( \mu_{B'} \) decays rapidly outside \( B' \).
			\end{itemize}

		\end{rem}
		
		Next, we use the following lemma from \cite{fan2018bilinear,ramos2018trilinear}, which provides a ``cap-plate'' decomposition without any loss. Here, a $d$-dimensional rectangle $\tau$ is called a $(v,\widetilde{v})$-plate if $\widetilde{v}\le v$ and its side lengths are
		\begin{equation*}
			\underbrace{v \times v \times \cdots \times v}_{d-1}\times \widetilde{v} .
		\end{equation*}
		
		\begin{lemma}\label{plate}
			Assume \( 0<v< 1 \). Let \( f_1 \) be supported in a cap of radius \(v\) centered at \(e_d\), and let \( f_2 \) be supported in a cap of radius \(v\) centered at \(0\). Given a covering \(\{ \tau_i \}\) of \(\operatorname{supp} f_i\) by \((v,v^2)\)-plates, with the shorter side along the \(e_{d}\)-direction, the following decoupling estimate holds for every \(R>v^{-2}\):
			\begin{equation*}
				\int_{\R^{d+1}} |E f_1 \, E f_2|^2 w_{B_R} \; dx \lesssim \sum_{\tau_1,\tau_2} \int_{\R^{d+1}} |E f_{1,\tau_1} \, E f_{2,\tau_2}|^2 w_{B_R} \; dx.
			\end{equation*}
		\end{lemma}
		\begin{proof}
			For completeness, we present the proof from \cite{fan2018bilinear}. By inequality (\ref{weight}), it suffices to show that, for every translate \(B'\) of \(B_R\),
			\begin{equation*}
				\int_{B'} |Ef_1 \, Ef_2|^2 \lesssim \sum_{\tau_1,\tau_2} \int |Ef_{1,\tau_1} \, Ef_{2,\tau_2}|^2 \mu_{B'}^2 .
			\end{equation*}
			Now,
			\begin{equation*}
				\int_{B'} |Ef_1 \, Ef_2|^2
				\leq \sum_{\tau_1,\tau_2,\tau_3,\tau_4}
				\int_{B'} Ef_{1,\tau_1} \, Ef_{2,\tau_2} \, \overline{Ef_{1,\tau_3}} \, \overline{Ef_{2,\tau_4}} \, \mu_{B'}^2 .
			\end{equation*}
			Write \(\xi_i \in \tau_i\) as
			\begin{equation*}
				\xi_i = \left(\xi_{i,1},\ldots,\xi_{i,d-1},\, \xi_{i,d}\right)
				\equiv \left(\xi_i', \xi_{i,d}\right), \qquad i=1,2,3,4.
			\end{equation*}
			Then we have
			\begin{equation*}
				|\xi_i'| \lesssim v \quad i=1,2,3,4,
			\end{equation*}
			\begin{equation*}
				|\xi_{i,d} - 1| \lesssim v, \quad i=1,3,\qquad
				|\xi_{i,d}| \lesssim v, \quad i=2,4.
			\end{equation*}
			Crucially, for any quadruple \(\tau_1,\tau_2,\tau_3,\tau_4\) such that
			\begin{equation*}
				\int Ef_{1,\tau_1} \, Ef_{2,\tau_2} \, \overline{Ef_{1,\tau_3}} \, \overline{Ef_{2,\tau_4}} \, \mu_B^2 \neq 0,
			\end{equation*}
			there must exist \(\xi_i \in \tau_i\) satisfying
			\begin{equation*}
				\xi_1 - \xi_3 = \xi_4 - \xi_2 + O(R^{-1}),
			\end{equation*}
			\begin{equation*}
				|\xi_1|^2 - |\xi_3|^2 = |\xi_4|^2 - |\xi_2|^2 + O(R^{-1}).
			\end{equation*}
			The second identity implies
			\begin{equation*}
				(\xi_{1,d} - \xi_{3,d})(\xi_{1,d} + \xi_{3,d})
				= O(|\xi_2'|^2 + |\xi_4'|^2) + O(|\xi_1'|^2 + |\xi_3'|^2) + O(R^{-1}).
			\end{equation*}
			Using the bounds obtained above, we get \(|\xi_{1,d} - \xi_{3,d}| \lesssim v^2\), and consequently \(|\xi_{2,d} - \xi_{4,d}| \lesssim v^2\).
			
			Thus, a nonvanishing integral forces \(\tau_1\) and \(\tau_3\), as well as \(\tau_2\) and \(\tau_4\), to lie within distance \(O(v^2)\) of each other, which essentially means that \(\tau_i = \tau_{i+2}\) for \(i=1,2\). Applying this observation to the above inequality proves Lemma \ref{plate}.

		\end{proof}
		\begin{rem}\label{84}
			The above proof also shows that the following ``cap-plate'' decomposition holds:
			\begin{itemize}
				\item 
				Assume \( 0<v^2\le \widetilde{v} \le v< 1\). Let \( f_1 \) be supported in a cap of radius \(v\) centered at \(e_d\), and let \( f_2 \) be supported in a cap of radius \(v\) centered at \(0\). Given a covering \(\{ \tau_i \}\) of \(\operatorname{supp} f_i\) by \((v,\widetilde{v})\)-plates, with the shorter side along the \(e_{d}\)-direction, the following decoupling estimate holds for every \(R>\widetilde{v}^{-1}\):
				\begin{equation*}
					\int_{\R^{d+1}} |E f_1 \, E f_2|^2 w_{B_R} \; dx \lesssim \sum_{\tau_1,\tau_2} \int_{\R^{d+1}} |E f_{1,\tau_1} \, E f_{2,\tau_2}|^2 w_{B_R} \; dx.
				\end{equation*}
			\end{itemize}
		\end{rem}

		Applying (\ref{weight}) and Remark \ref{84}, we can further reduce Conjecture \ref{821} to the following equivalent Conjecture \ref{8211}.
		\begin{conj}\label{8211}
			With the same notation as in Conjecture \ref{821}, if $\boldsymbol{N_1 \ge N_2^2}$, then for a covering \(\{ \tau_i \}\) of \(\operatorname{supp} f_i\) by $\left(\frac{N_2}{N_1},\frac{1}{N_1}\right)$-plates, with the shorter side along the \(e_{d}\)-direction, we have the following estimate: for any small \(\epsilon > 0\), when \(d \ge 4\),
			\begin{equation}\label{case1 estimate}
				\|E f_{1,\tau_1} E f_{2,\tau_2}\|_{L^2(\omega_{B_{N_1^2}})} \lesssim_\epsilon (N_2)^\epsilon  N_2^{\frac{2d-7}{4}} \prod_{j=1}^2 \left( \sum_{\substack{\theta\subseteq \tau_j\\ |\theta|=1/{N_1}}} \|E f_{j, \theta}\|_{L^4(\omega_{B_{N_1^2}})}^2 \right)^{1/2};
			\end{equation}
			if $\boldsymbol{N_1 \le N_2^2}$, then for a covering \(\{ \tau_i \}\) of \(\operatorname{supp} f_i\) by $\left(\frac{N_2}{N_1},\frac{N_2^2}{N_1^2}\right)$-plates, with the shorter side along the \(e_{d}\)-direction, we have the analogous estimate:
			\begin{equation}\label{case2 estimate}
				\|E f_{1,\tau_1} E f_{2,\tau_2}\|_{L^2(\omega_{B_{N_1^2}})} \lesssim_\epsilon (N_2)^\epsilon  N_2^{\frac{2d-5}{4}}N_1^{-1/4} \prod_{j=1}^2 \left( \sum_{\substack{\theta\subseteq \tau_j\\ |\theta|=1/{N_1}}} \|E f_{j, \theta}\|_{L^4(\omega_{B_{N_1^2}})}^2 \right)^{1/2},
			\end{equation}
			
			where all the implicit constants above are independent of the choice of $(\tau_1, \tau_2)$.
		\end{conj}

		\vspace{15pt}
		\section{Proof of Conjecture \ref{8211} $(d\ge 4)$}
		\label{sec:decoupling-proof}
		We split our discussion into two cases: $\boldsymbol{N_1 \ge N_2^2}$ and $\boldsymbol{N_1 \le N_2^2}$. Note that there are many possible pairs $(\tau_1,\tau_2)$. For convenience, we focus on the two extreme positions of $\tau_2$: equatorial and polar, while $\tau_1$ is allowed to be arbitrary. Here, we call $\tau_2$ equatorial if the $d$-th coordinate of its center is $0$; we call $\tau_2$ polar if the $d$-th coordinate of its center is $\frac{N_2}{N_1}$.
		
		The general-position case can be handled similarly, and interested readers can find the complete proof in Appendix~\ref{sec:general-position}. 
		
		As in the previous section, we write
		\begin{equation*}
			\delta:=\frac1{N_1},\qquad v:=\frac{N_2}{N_1}=N_2\delta.
		\end{equation*}
		We first record some linear decoupling results that will be used repeatedly. 
		
		\begin{lemma}[Linear lifting]\label{lifting}
			Let $\Omega'\subset\R^{d-1}$, $a\in\R$, and suppose that
			\begin{equation*}
				\supp \; g\subseteq\left\{(\xi',\xi_d)\in \R^d:\xi'\in\Omega',\ |\xi_d-a|\lesssim\delta\right\}.
			\end{equation*}
			Assume that every $h$ supported in $\Omega'$ satisfies
			\begin{equation*}
				\|E_{\mathbb P^{d-1}}h\|_{L^4(\omega_{B_{\delta^{-2}}^d})}
				\lesssim A
				\left(\sum_{|\theta'|\sim\delta}\|E_{\mathbb P^{d-1}}h_{\theta'}\|_{L^4(\omega_{B_{\delta^{-2}}^d})}^2\right)^{1/2}.
			\end{equation*}
			Then we have
			\begin{equation*}
				\|E_{\mathbb{P}^d}g\|_{L^4(\omega_{B_{\delta^{-2}}^{d+1}})}
				\lesssim A
				\left(\sum_{|\theta|\sim\delta}\|E_{\mathbb{P}^d}g_{\theta}\|_{L^4(\omega_{B_{\delta^{-2}}^{d+1}})}^2\right)^{1/2}.
			\end{equation*}
			The implicit constant is independent of $\delta$, $a$, and the position of $\Omega'$.
		\end{lemma}
		\begin{proof}
			Applying a Galilean transformation, we can assume that $a=0$. Then we see the Fourier support of $Eg$ is contained in 
			\begin{equation*}
				\left\lbrace (\xi',\xi_d,\tau)\in \R^d\times \R: \xi'\in \Omega', |\xi_d|\lesssim \delta, ||\xi'|^2-\tau|\lesssim \delta^2 \right\rbrace,
			\end{equation*} 
			which is exactly the $\delta^2$-neighborhood of a parabolic cylinder. Then by Fubini's theorem and the uncertainty principle, one can directly obtain the linear lifting lemma.

		\end{proof}
		
		We shall also use the following localized form of the shell-type linear decoupling estimate. It follows from \cite[Corollary 3.3]{kinoshita2023decoupling}; see also Theorem \ref{variant}.
		\begin{lemma}\label{localized-shell}
			Let $m\ge3$, $1\le Q\le\delta^{-1}$, and suppose that $g:\R^m\to\C$ is supported in a radial $\delta$-shell and in a ball of radius $O(Q\delta)$, i.e.,
			\begin{equation*}
				\supp \;g\subseteq \lbrace \xi\in \R^m: 1-\delta\le|\xi|\le 1+\delta\rbrace\cap B_{Q\delta},
			\end{equation*}
			where $B_{Q\delta}$ is a ball of radius $Q\delta$ with arbitrary center.
			
			Then we have 
			\begin{equation}\label{localized-shell-estimate}
				\|Eg\|_{L^4(\omega_{B_{\delta^{-2}}^{m+1}})}
				\lesssim_\epsilon Q^{\frac{m-3}{4}+\epsilon}D_\delta(g),
			\end{equation}
			where
			\begin{equation*}
				D_\delta(g):=\left(\sum_{|\theta|\sim\delta}\|Eg_\theta\|_{L^4(\omega_{B_{\delta^{-2}}^{m+1}})}^2\right)^{1/2}.
			\end{equation*}
			In particular, if the frequency dimension is $m=d-1$, the exponent in \eqref{localized-shell-estimate} is $(d-4)/4$.
		\end{lemma}
		
		Finally, we shall repeatedly use the following inequality. Suppose that
		the Fourier supports of $G_1,\ldots,G_M$ lie in parallel strips of equal
		width, whose centers are equally spaced, with uniformly bounded overlap.
		Then
		\begin{equation}\label{eq:strip-interpolation}
			\left\|\sum_{j=1}^M G_j\right\|_{L^4}
			\lesssim M^{1/4}
			\left(\sum_{j=1}^M\|G_j\|_{L^4}^2\right)^{1/2}.
		\end{equation}
		Indeed, the Fourier support of $G_jG_k$ lies in a strip whose position
		depends only on $j+k$. These strips have bounded overlap as $j+k$
		varies, and there are at most $M$ pairs with a prescribed sum.
		Plancherel's theorem and Cauchy--Schwarz therefore give
		\begin{equation*}
			\left\|\sum_{j=1}^M G_j\right\|_{L^4}^4
			=
			\left\|\sum_{j,k=1}^M G_jG_k\right\|_{L^2}^2
			\lesssim M\sum_{j,k=1}^M\|G_jG_k\|_{L^2}^2
			\le M\left(\sum_{j=1}^M\|G_j\|_{L^4}^2\right)^2.
		\end{equation*}
		
		The same estimate holds with the weights $\omega_{B_{\delta^{-2}}}$ used below, provided that the strip width is at least $\delta^2$. We also use the following standard consequence of $L^2$-orthogonality: if $g$ has spatial Fourier support of diameter $O(r)$, where $r\gtrsim \delta^{2}$, and $f=\sum_qf_q$, where the $f_q$ are supported in finitely overlapping $r$-balls, then
		\begin{equation}\label{product-orthogonality}
			\|Ef\,Eg\|_{L^2(\omega_{B_{\delta^{-2}}})}^2
			\lesssim\sum_q\|Ef_q\,Eg\|_{L^2(\widetilde\omega_{B_{\delta^{-2}}})}^2.
		\end{equation}
		Indeed, the Minkowski sums $\supp \; f_q+\supp\; g$ have bounded overlap; the Fourier enlargement caused by the physical cutoff is $O(\delta^2)$ and is harmless.
		
		\subsection{$\boldsymbol{N_1 \ge N_2^2}$ \& $\boldsymbol{\tau_2}$ is polar}
		
		\begin{figure}
			\centering
			\includegraphics[width=0.7\linewidth]{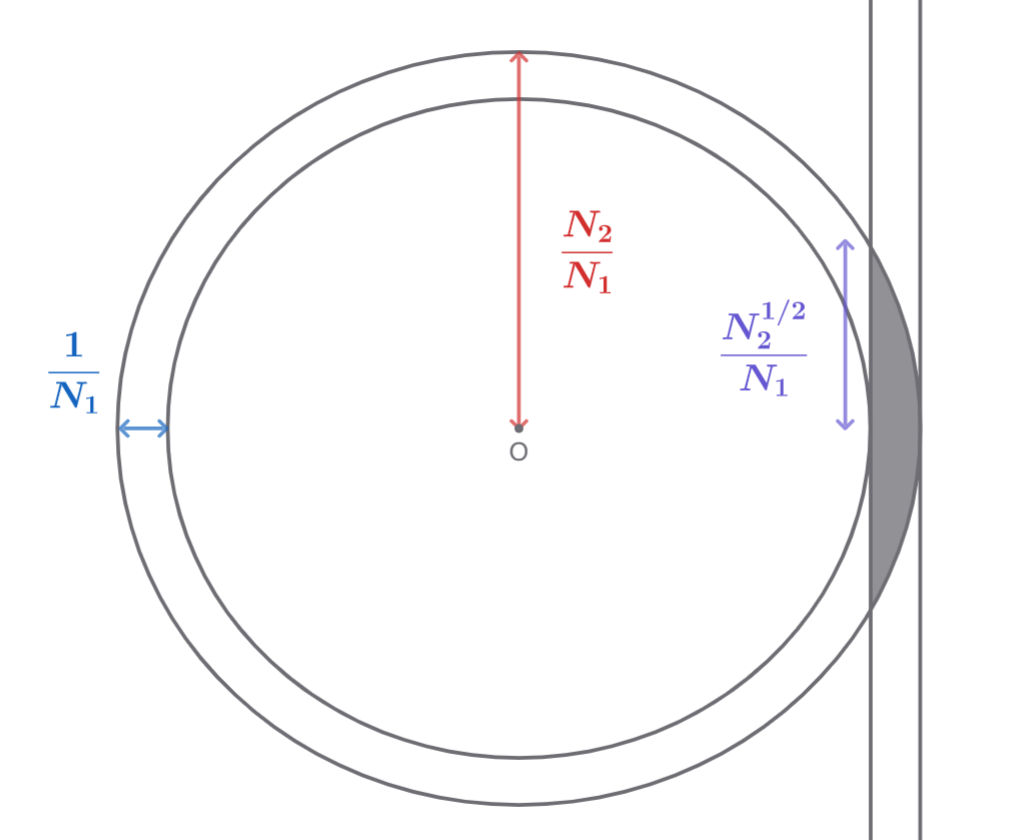}
			\caption{Polar $\left(\frac{N_2}{N_1}, \frac{1}{N_1}\right)$-plate \& $N_1 \ge N_2^2$}
			\label{P1}
		\end{figure}
		Since $v^2\le\delta$, the relevant plates have dimensions $(v,\delta)$. From the shell condition and the polar position of $\tau_2$ (see Figure \ref{P1}),
		\begin{equation*}
			\supp f_{2,\tau_2}
			\subseteq
			\left\{\xi=(\xi',\xi_d):|\xi'|\lesssim (v\delta)^{1/2},\ |\xi_d-v|\lesssim\delta\right\}.
		\end{equation*}
		Set $K:=N_2^{1/2}$, so $(v\delta)^{1/2}=K\delta$. By Lemma \ref{lifting} and the Bourgain--Demeter decoupling estimate (\ref{BD}) in $d-1$ dimensions,
		\begin{equation}\label{polar-low-linear}
			\|Ef_{2,\tau_2}\|_{L^4(\omega_{B_{\delta^{-2}}})}
			\lesssim_\epsilon K^{\frac{d-3}{4}+\epsilon}D_\delta(f_{2,\tau_2}).
		\end{equation}
		
		We next decompose the tangential support of $f_{1,\tau_1}$ into balls $q$ of radius $K\delta$ and write $f_{1,\tau_1}=\sum_qf_{1,q}$. Since $f_1$ is contained in the unit shell and in a $v$-cap around $e_d$, each $f_{1,q}$ is contained in a full $O(K\delta)$-ball. Hence Lemma \ref{localized-shell} gives
		\begin{equation}\label{polar-high-linear}
			\|Ef_{1,q}\|_{L^4(\omega_{B_{\delta^{-2}}})}
			\lesssim_\epsilon K^{\frac{d-3}{4}+\epsilon}D_\delta(f_{1,q}).
		\end{equation}
		Moreover, $\supp f_{2,\tau_2}$ has diameter $O(K\delta)$, so \eqref{product-orthogonality}, H\"older's inequality, \eqref{polar-low-linear}, and \eqref{polar-high-linear} imply
		\begin{equation*}
			\|Ef_{1,\tau_1}Ef_{2,\tau_2}\|_{L^2(\omega_{B_{\delta^{-2}}})}
			\lesssim_\epsilon K^{\frac{d-3}{2}+\epsilon}D_\delta(f_{1,\tau_1})D_\delta(f_{2,\tau_2}).
		\end{equation*}
		Since $K=N_2^{1/2}$ and $d\ge4$,
		\begin{equation*}
			K^{\frac{d-3}{2}}=N_2^{\frac{d-3}{4}}\le N_2^{\frac{2d-7}{4}}.
		\end{equation*}
		This proves \eqref{case1 estimate} in the polar case.
		
		\subsection{$\boldsymbol{N_1 \ge N_2^2}$ \& $\boldsymbol{\tau_2}$ is equatorial}
		
		\begin{figure}
			\centering
			\includegraphics[width=0.7\linewidth]{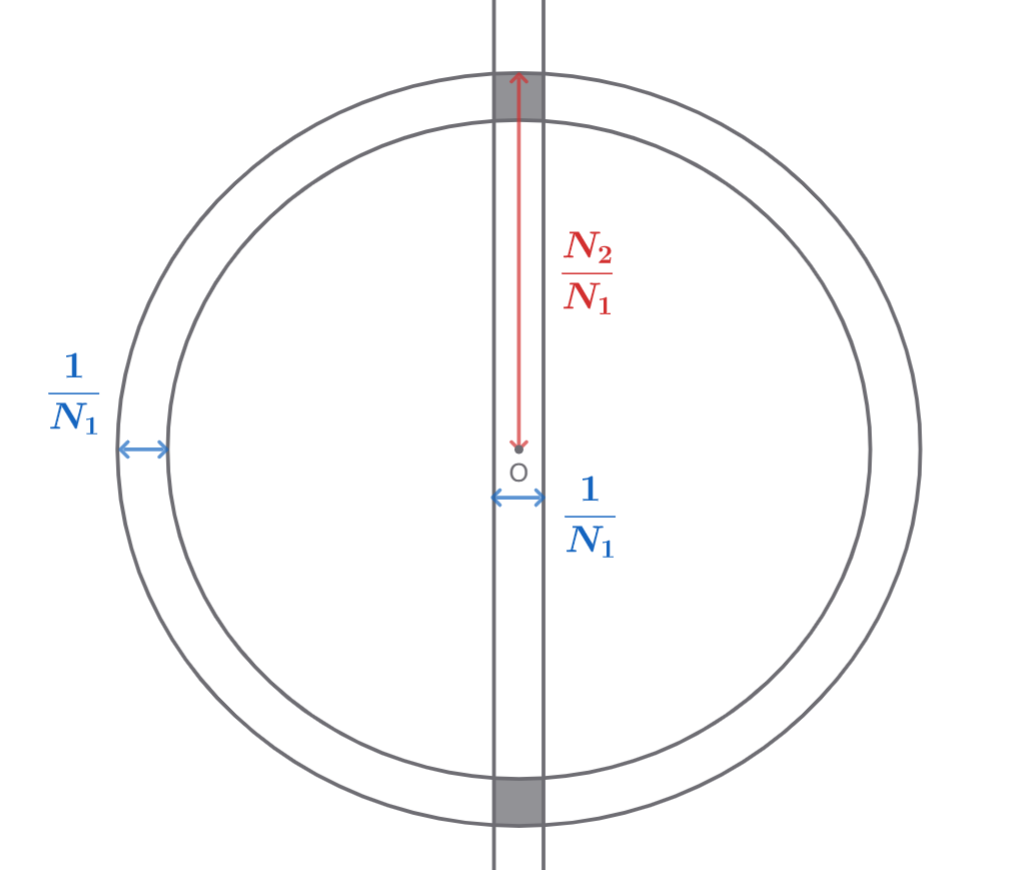}
			\caption{Equatorial $\left(\frac{N_2}{N_1}, \frac{1}{N_1}\right)$-plate \& $N_1 \ge N_2^2$}
			\label{E1}
		\end{figure}
		In this case (see Figure \ref{E1}),
		\begin{equation*}
			\supp f_{2,\tau_2}
			\subseteq
			\left\{\xi=(\xi',\xi_d):\big||\xi'|-v\big|\lesssim\delta,\ |\xi_d|\lesssim\delta\right\}.
		\end{equation*}
		By Lemma \ref{lifting}, linear shell-type decoupling (\ref{linear}) and rescaling,
		\begin{equation}\label{equatorial-low-linear}
			\|Ef_{2,\tau_2}\|_{L^4(\omega_{B_{\delta^{-2}}})}
			\lesssim_\epsilon N_2^{\frac{d-4}{4}+\epsilon}D_\delta(f_{2,\tau_2}).
		\end{equation}
		On the other hand, $f_{1,\tau_1}$ is supported in the unit shell and in a ball of radius $O(v)=O(N_2\delta)$. Hence
		\begin{equation}\label{equatorial-high-linear}
			\|Ef_{1,\tau_1}\|_{L^4(\omega_{B_{\delta^{-2}}})}
			\lesssim_\epsilon N_2^{\frac{d-3}{4}+\epsilon}D_\delta(f_{1,\tau_1}).
		\end{equation}
		H\"older's inequality gives
		\begin{equation*}
			\|Ef_{1,\tau_1}Ef_{2,\tau_2}\|_{L^2(\omega_{B_{\delta^{-2}}})}
			\lesssim_\epsilon
			N_2^{\frac{2d-7}{4}+\epsilon}
			D_\delta(f_{1,\tau_1})D_\delta(f_{2,\tau_2}),
		\end{equation*}
		which is exactly \eqref{case1 estimate}.
		
		\subsection{$\boldsymbol{N_1 \le N_2^2}$ \& $\boldsymbol{\tau_2}$ is polar}
		
		Now $\delta\le v^2$, and the lossless cap--plate decomposition gives $(v,v^2)$-plates. 
		
		For $K$ dyadic, we decompose $f_{2,\tau_2}$ according to its projection size:
		\begin{equation*}
			f_{2,\tau_2}=\sum_{N_2^{1/2}\le K\le N_2} f_{2,K},
		\end{equation*}
		where $f_{2, N_2^{1/2}}$ corresponds to $|\xi'|\lesssim N_2^{1/2}\delta$, while the remaining $f_{2,K}$ satisfy
		\begin{equation*}
			\supp \;f_{2,K}\subseteq\lbrace|\xi'|\sim K\delta\rbrace.
		\end{equation*}
		
		For each $K$, we further decompose the short direction of $\tau_2$ into intervals $\lbrace I_j\rbrace$ of length $\delta$. More precisely, 
		\begin{equation*}
			f_{2,K}=\sum_{j=1}^{H}f_{2,K,j}, \quad H:=\frac{v^2}{\delta}=\frac{N_2^2}{N_1}\ge 1,
		\end{equation*}
		where $\lbrace a_j\rbrace$ denotes the centers of $\lbrace I_j\rbrace$,
		\begin{equation*}
			\supp\; f_{2,K,j} \subseteq \lbrace|\xi_d-a_j|\le \delta\rbrace
		\end{equation*}

		For the case $K=N_2^{1/2}$, the argument of the first subsection gives
		\begin{equation*}
			\|Ef_{2,K,j}\|_4\lesssim_\epsilon K^{\frac{d-3}{4}+\epsilon}D_\delta(f_{2,K,j}).
		\end{equation*}
		For $K>N_2^{1/2}$, set
		\begin{equation*}
			L:=\frac{N_2}{K}.
		\end{equation*}
		Then $1\le L\le K$ and $KL=N_2$. The shell condition and $|\xi_d-a_j|\lesssim\delta$ imply
		\begin{equation*}
			\left||\xi'|^2-(v^2-a_j^2)\right|\lesssim v\delta.
		\end{equation*}
		Since $|\xi'|\sim K\delta$, the projected radial thickness is $O(L\delta)$. Decomposing it into $O(L)$ radial $\delta$-shells, applying Lemma \ref{localized-shell} in $d-1$ dimensions and then Lemma \ref{lifting}, we obtain
		\begin{equation}\label{general-slice-linear}
			\|Ef_{2,K,j}\|_4
			\lesssim_\epsilon
			L^{1/4}K^{\frac{d-4}{4}+\epsilon}D_\delta(f_{2,K,j}).
		\end{equation}
		Using \eqref{eq:strip-interpolation} in the $\xi_d$-slices yields
		\begin{equation}\label{f2K-linear}
			\|Ef_{2,K}\|_4
			\lesssim_\epsilon
			H^{1/4}L^{1/4}K^{\frac{d-4}{4}+\epsilon}D_\delta(f_{2,K}).
		\end{equation}
		
		For each fixed $K$, decompose the support of $f_{1,\tau_1}$ into $K\delta$-balls $q$. As above,
		\begin{equation*}
			\|Ef_{1,q}\|_4\lesssim_\epsilon K^{\frac{d-3}{4}+\epsilon}D_\delta(f_{1,q}),
		\end{equation*}
		and $\supp f_{2,K}$ has diameter $O(K\delta)$. Therefore \eqref{product-orthogonality} and H\"older give, for $K>N_2^{1/2}$,
		\begin{equation*}
			\|Ef_{1,\tau_1}Ef_{2,K}\|_2
			\lesssim_\epsilon
			H^{1/4}L^{1/4}K^{\frac{2d-7}{4}+\epsilon}
			D_\delta(f_{1,\tau_1})D_\delta(f_{2,K}).
		\end{equation*}
		Since $KL=N_2$,
		\begin{equation*}
			L^{1/4}K^{\frac{2d-7}{4}}
			=N_2^{1/4}K^{\frac{d-4}{2}}
			\le N_2^{\frac{2d-7}{4}}.
		\end{equation*}
		
		Summing over the $O(\log N_2)$ values of $K$, and absorbing the logarithm into $N_2^\epsilon$, we obtain
		\begin{equation*}
			\|Ef_{1,\tau_1}Ef_{2,\tau_2}\|_2
			\lesssim_\epsilon
			H^{1/4}N_2^{\frac{2d-7}{4}+\epsilon}
			D_\delta(f_{1,\tau_1})D_\delta(f_{2,\tau_2}).
		\end{equation*}
		Since
		\begin{equation*}
			H^{1/4}N_2^{\frac{2d-7}{4}}
			=N_2^{\frac{2d-5}{4}}N_1^{-1/4},
		\end{equation*}
		this proves \eqref{case2 estimate} in the polar case.
		
		\begin{figure}
			\centering
			\includegraphics[width=0.7\linewidth]{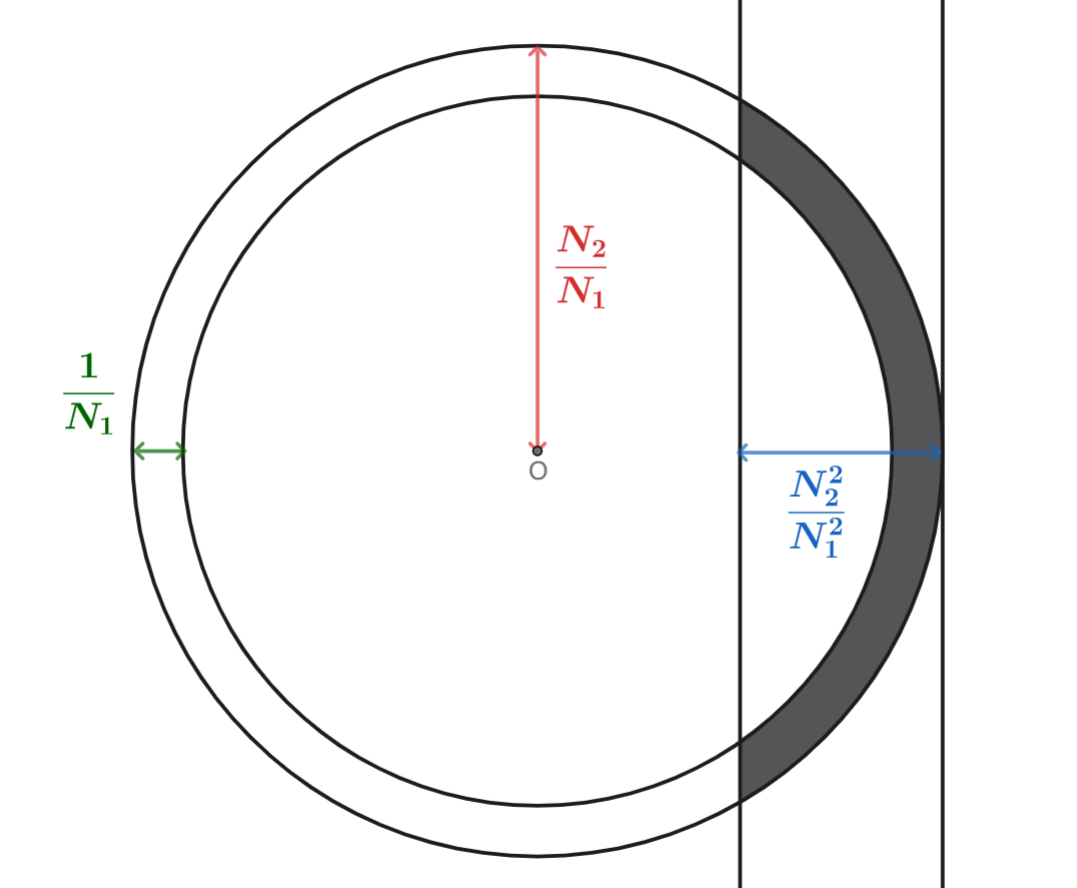}
			\caption{Polar $\left(\frac{N_2}{N_1}, \frac{N_2^2}{N_1^2}\right)$-plate \& $N_1 \le N_2^2$}
			\label{P2}
		\end{figure}
		
		\subsection{$\boldsymbol{N_1 \le N_2^2}$ \& $\boldsymbol{\tau_2}$ is equatorial}
		
		For an equatorial $(v,v^2)$-plate, decompose only its short direction as before:
		\begin{equation*}
			f_{2,\tau_2}=\sum_{j=1}^{H}f_{2,\tau_2,j},\qquad \supp f_{2,\tau_2,j}\subseteq\lbrace|\xi_d-a_j|\le\delta\rbrace.
		\end{equation*}
		Since $|a_j|\lesssim v^2\ll v$, each projected support is contained in a radial $\delta$-shell of radius comparable to $v$. Hence Lemma \ref{lifting} and Lemma \ref{localized-shell} give
		\begin{equation*}
			\|Ef_{2,\tau_2,j}\|_4\lesssim_\epsilon N_2^{\frac{d-4}{4}+\epsilon}D_\delta(f_{2,\tau_2,j}).
		\end{equation*}
		Using \eqref{eq:strip-interpolation},
		\begin{equation*}
			\|Ef_{2,\tau_2}\|_4
			\lesssim_\epsilon
			H^{1/4}N_2^{\frac{d-4}{4}+\epsilon}D_\delta(f_{2,\tau_2}).
		\end{equation*}
		Together with \eqref{equatorial-high-linear} and H\"older's inequality, this gives
		\begin{equation*}
			\|Ef_{1,\tau_1}Ef_{2,\tau_2}\|_2
			\lesssim_\epsilon
			H^{1/4}N_2^{\frac{2d-7}{4}+\epsilon}
			D_\delta(f_{1,\tau_1})D_\delta(f_{2,\tau_2}),
		\end{equation*}
		which is precisely \eqref{case2 estimate}.
		
		\begin{figure}
			\centering
			\includegraphics[width=0.7\linewidth]{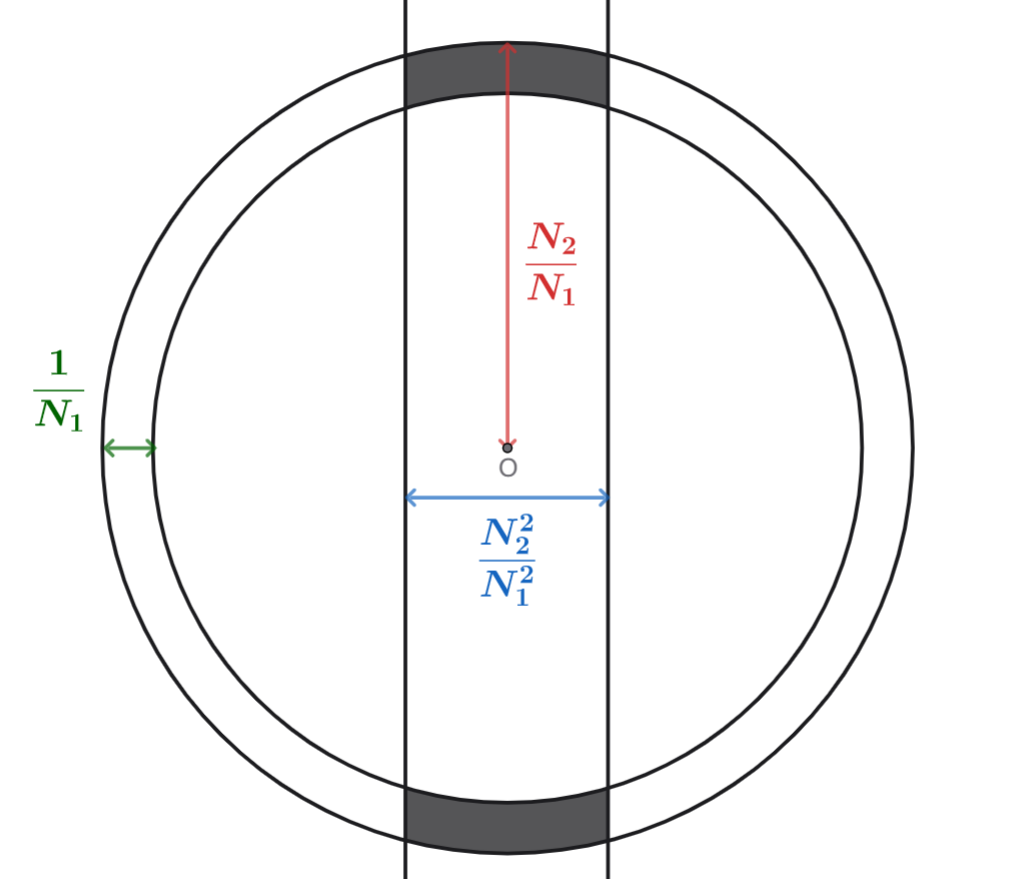}
			\caption{Equatorial $\left(\frac{N_2}{N_1}, \frac{N_2^2}{N_1^2}\right)$-plate \& $N_1 \le N_2^2$}
			\label{E2}
		\end{figure}

		\vspace{10pt}
		\section{Sharpness of shell-type bilinear decoupling}
		\label{sec:decoupling-sharpness}
		In this section, we show that the factor
		\begin{equation*}
			N_2^{\frac{2d-7}{4}}
			\left(1+\frac{N_2^2}{N_1}\right)^{1/4}
		\end{equation*}
		in \eqref{bilinear estimate} is sharp up to a factor of $N_2^\epsilon$. Recall that
		\begin{equation*}
			N_2^{\frac{2d-7}{4}}
			\left(1+\frac{N_2^2}{N_1}\right)^{1/4}
			\sim
			\begin{cases}
				N_2^{\frac{2d-5}{4}}N_1^{-1/4},&N_1\le N_2^2,\\
				N_2^{\frac{2d-7}{4}},&N_1\ge N_2^2.
			\end{cases}
		\end{equation*}
		
		\begin{thm}\label{prop:sharpness-bilinear-shell}
			Let $d\ge4$. Suppose that, for every $N_1\ge N_2\ge1$, one has
			\begin{equation}\label{eq:abstract-bilinear-shell}
				\|Ef_1Ef_2\|_{L^2(B_{N_1^2})}
				\le C(N_1,N_2)
				\prod_{j=1}^2
				\left(\sum_{|\theta|=1/N_1}\|Ef_{j,\theta}\|_{L^4(B_{N_1^2})}^2\right)^{1/2}
			\end{equation}
			whenever $f_1$ and $f_2$ satisfy the shell assumptions in Conjecture \ref{c}. Then
			\begin{equation*}
				C(N_1,N_2)
				\gtrsim
				N_2^{\frac{2d-7}{4}}
				\left(1+\frac{N_2^2}{N_1}\right)^{1/4}.
			\end{equation*}
			Consequently, the factor in \eqref{bilinear estimate} is sharp up to $N_2^\epsilon$.
		\end{thm}
		
		\begin{proof}
			Set
			\begin{equation*}
				\delta:=\frac1{N_1},\qquad v:=\frac{N_2}{N_1}=N_2\delta.
			\end{equation*}
			We first calculate the $L^4$-norm of one anisotropic bump. Let $0<\kappa\le\delta$, let $\varphi\in C_c^\infty(\R^d)$ be nonnegative and $L^1$-normalized, and define
			\begin{equation*}
				\varphi_{\delta,\kappa}(\xi',\xi_d)
				:=\delta^{-(d-1)}\kappa^{-1}
				\varphi\left(\frac{\xi'}{\delta},\frac{\xi_d}{\kappa}\right).
			\end{equation*}
			For any center $\xi_0$, the change of variables
			\begin{equation*}
				X':=\delta(x'+2t\xi_0'),\qquad X_d:=\kappa(x_d+2t\xi_{0,d}),\qquad T:=\delta^2t
			\end{equation*}
			gives
			\begin{equation}\label{anisotropic-L4}
				\|E[\varphi_{\delta,\kappa}(\cdot-\xi_0)]\|_{L^4(B_{\delta^{-2}})}
				\lesssim
				\delta^{-\frac{d+1}{4}}\kappa^{-1/4}.
			\end{equation}
			Indeed, after this change of variables the remaining oscillatory integral is the Fourier transform of
			\begin{equation*}
				e^{2\pi iT(|u'|^2+(\kappa/\delta)^2u_d^2)}\varphi(u',u_d).
			\end{equation*}
			Since $|T|\lesssim1$ and $0<\kappa/\delta\le1$, this is a uniformly Schwartz family in the spatial variables. The Jacobian is $\delta^{d+1}\kappa$, which proves \eqref{anisotropic-L4}.
			
			\medskip
			\noindent
			\textbf{First example (Figure \ref{fig:example-1}): }$N_1\le N_2^2$.
			Choose a $C\delta$-separated set
			\begin{equation*}
				U\subset B_{\R^{d-1}}(0,cv),
				\qquad
				M:=\#U\sim\left(\frac v\delta\right)^{d-1}=N_2^{d-1},
			\end{equation*}
			where $c>0$ is sufficiently small. For $u\in U$, set
			\begin{equation*}
				a_u:=\left(u, \sqrt{1-|u|^2}\right),
				\qquad
				b_u:=\left(u,\sqrt{v^2-|u|^2}\right).
			\end{equation*}
			\begin{figure}
				\centering
				\includegraphics[width=0.7\linewidth]{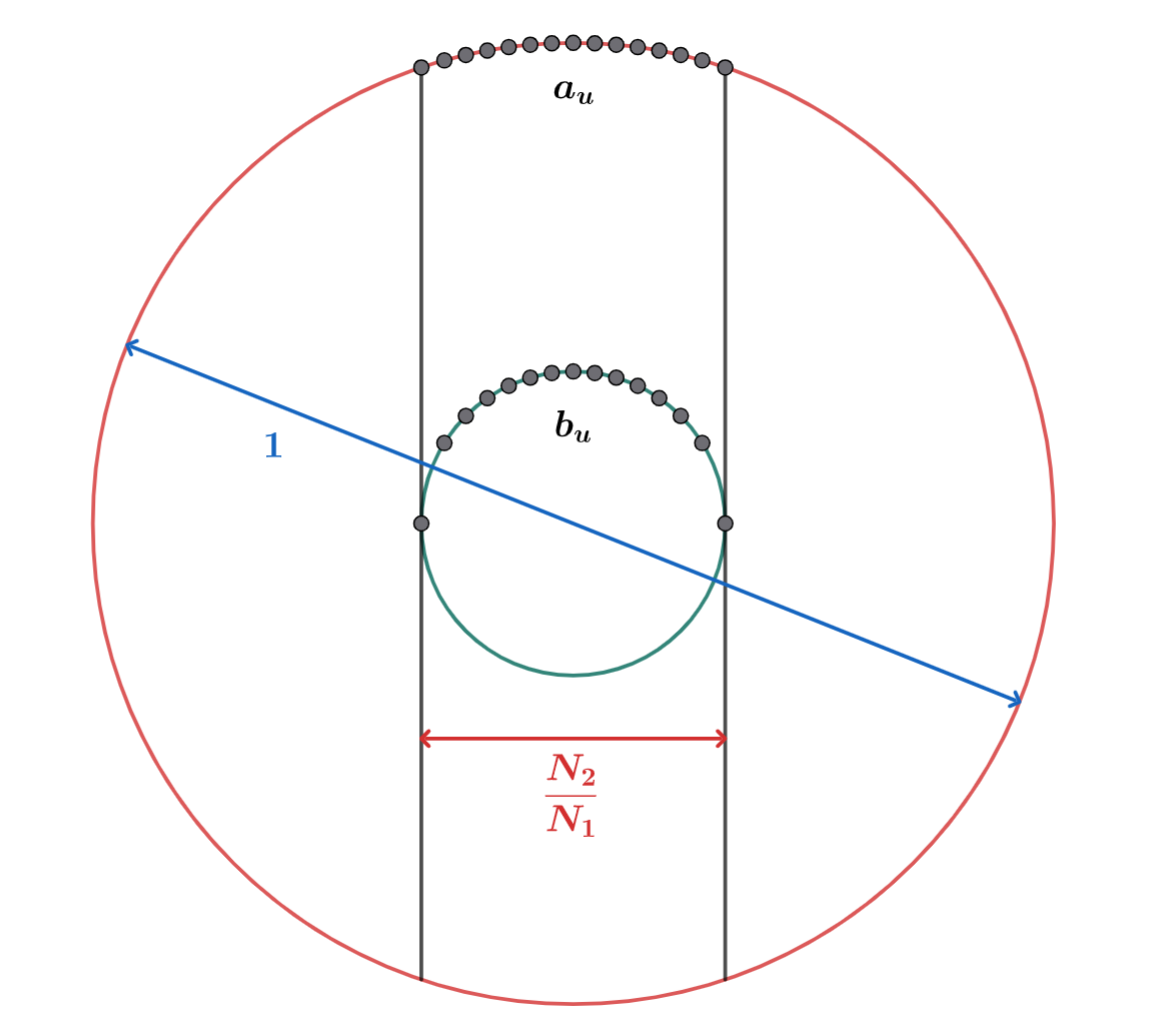}
				\caption{First example}
				\label{fig:example-1}
			\end{figure}
			
			Thus $|a_u|=1$ and $|b_u|=v$ exactly. Put $h:=v\delta$ and define
			\begin{equation*}
				f_1:=\sum_{u\in U}\varphi_{\delta,h}(\cdot-a_u),
			\end{equation*}
			where the short direction is the $d$-th coordinate, and
			\begin{equation*}
				f_2:=\sum_{u\in U}\psi_\delta(\cdot-b_u),
			\end{equation*}
			where $\psi_\delta$ is an isotropic $L^1$-normalized $\delta$-bump. By taking the supports of the fixed bumps sufficiently small, both functions satisfy the shell assumptions in Conjecture \ref{c}.
			
			Consider
			\begin{equation*}
				\Omega_1:=\left\{(x,t):\ |x'|\le c_0v^{-1},|x_d|\le c_0v^{-1}, \ |t|\le c_0(v\delta)^{-1}\right\}.
			\end{equation*}
			We claim that
			\begin{equation}\label{first-example-coherence}
				|Ef_1(x,t)|\gtrsim M,
				\qquad
				|Ef_2(x,t)|\gtrsim M,
				\qquad (x,t)\in\Omega_1.
			\end{equation}
			To see this, the extension operator of the atom centered at $a_u$ can be written as
			\begin{equation*}
				e^{2\pi i(x\cdot a_u+t)}\Phi_u(x,t),
			\end{equation*}
			where the phase of $\Phi_u$ is bounded by
			\begin{equation*}
				\delta|x'+2tu|+h|x_d+2t\sqrt{1-|u|^2}|+|t|(h^2+\delta^2)\lesssim c_0
			\end{equation*}
			on $\Omega_1$. Hence $|\Phi_u-1|\ll1$. Moreover, for $u,\widetilde u\in U$,
			\begin{equation*}
				|x\cdot(a_u-a_{\widetilde u})|
				\lesssim |x'||u-\widetilde u|+|x_d|\,\big|\sqrt{1-|u|^2}-\sqrt{1-|\widetilde{u}|^2}\big|
				\lesssim c_0.
			\end{equation*}
			The temporal phases agree exactly because $|a_u|=1$. Thus all atomic contributions to $Ef_1$ have nearly the same phase. The same argument applies to $f_2$: the phase is $O(c_0)$, the temporal phase is common because $|b_u|=v$, and
			\begin{equation*}
				|x\cdot(b_u-b_{\widetilde u})|\lesssim c_0.
			\end{equation*}
			This proves \eqref{first-example-coherence}.
			
			Since
			\begin{equation*}
				|\Omega_1|\sim v^{-d}(v\delta)^{-1}=v^{-(d+1)}\delta^{-1},
			\end{equation*}
			we obtain
			\begin{equation}\label{first-example-LHS}
				\|Ef_1Ef_2\|_{L^2(B_{\delta^{-2}})}
				\gtrsim M^2v^{-\frac{d+1}{2}}\delta^{-1/2}.
			\end{equation}
			Every $\delta$-cap contains $O(1)$ atoms. Hence \eqref{anisotropic-L4} gives
			\begin{equation*}
				D_\delta(f_1)
				\lesssim M^{1/2}\delta^{-\frac{d+2}{4}}v^{-1/4},
				\qquad
				D_\delta(f_2)
				\lesssim M^{1/2}\delta^{-\frac{d+2}{4}}.
			\end{equation*}
			Dividing \eqref{first-example-LHS} by these two bounds and using $M\sim(v/\delta)^{d-1}$ yields
			\begin{equation}\label{first-sharp-lower}
				C(N_1,N_2)
				\gtrsim
				v^{\frac{2d-5}{4}}\delta^{-\frac{d-3}{2}}
				=N_2^{\frac{2d-5}{4}}N_1^{-1/4}.
			\end{equation}
			This is the desired lower bound when $N_1\le N_2^2$.

			\vspace{7pt}
			\medskip
			\noindent
			\textbf{Second example (Figure \ref{fig:example-2}): }$N_1\ge N_2^2$.
			Now $v^2\le\delta$. Choose a $C\delta$-separated set
			\begin{equation*}
				U\subset B_{\R^{d-1}}(0,cv),
				\qquad
				M_1:=\#U\sim N_2^{d-1},
			\end{equation*}
			and define
			\begin{equation*}
				a_u:=\left(u,\sqrt{1-|u|^2}\right).
			\end{equation*}
			Choose also a $C\delta$-separated set $Z$ in a fixed angular patch of
			\begin{equation*}
				\{z\in\R^{d-1}:|z|=v\},
				\qquad
				M_2:=\#Z\sim N_2^{d-2},
			\end{equation*}
			and put $b_z:=(z,0)$. 
			\begin{figure}
				\centering
				\includegraphics[width=0.7\linewidth]{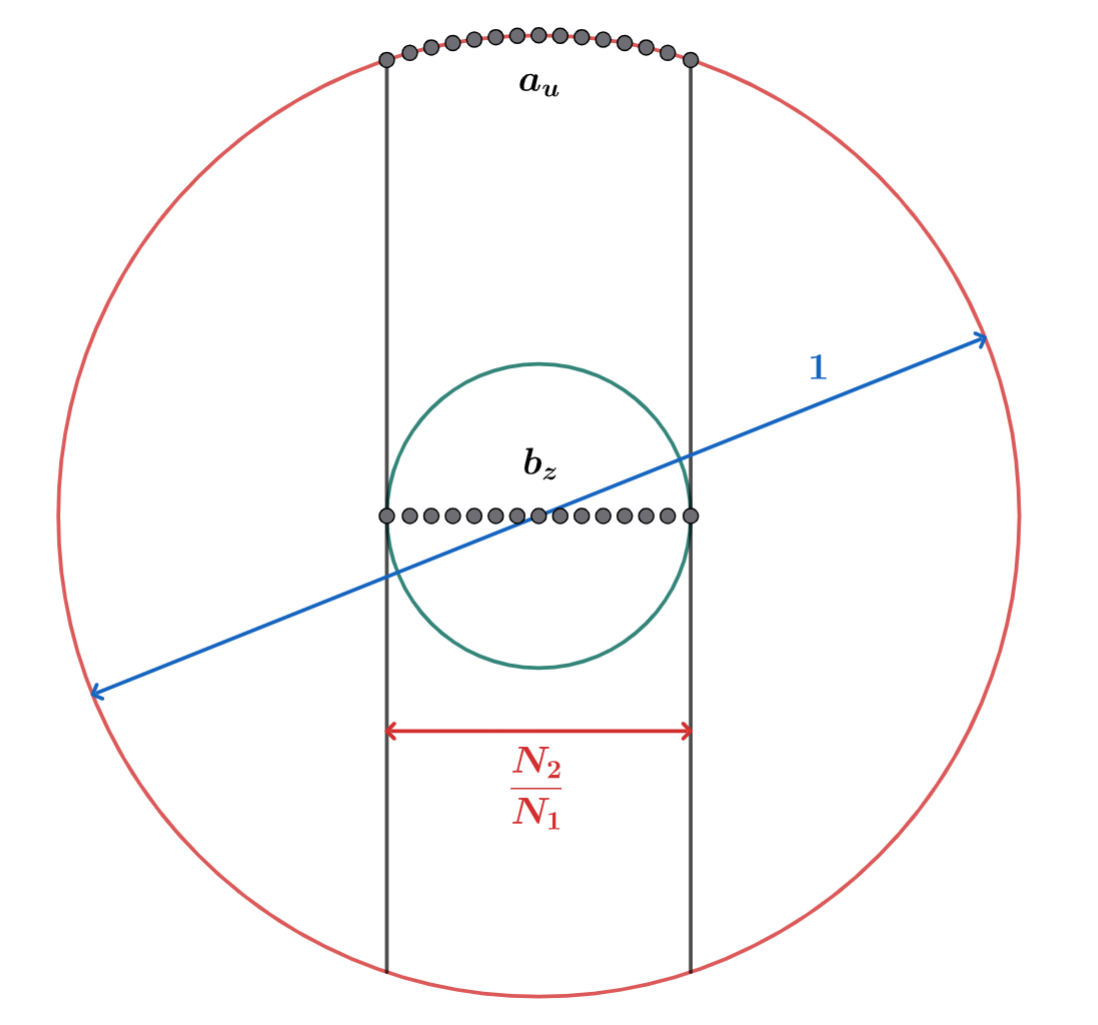}
				\caption{Second example}
				\label{fig:example-2}
			\end{figure}

			Define
			\begin{equation*}
				f_1:=\sum_{u\in U}\varphi_{\delta,v\delta}(\cdot-a_u),
				\qquad
				f_2:=\sum_{z\in Z}\psi_{\delta,v^2}(\cdot-b_z),
			\end{equation*}
			where the short direction is the $d$-th coordinate. Since $v^2\le\delta$, both functions are admissible. In particular, if $\xi=b_z+(\eta',\eta_d)$ with $|\eta'|\lesssim\delta$ and $|\eta_d|\lesssim v^2$, then
			\begin{equation*}
				\big||\xi|-v\big|\lesssim\delta+v^3\lesssim\delta.
			\end{equation*}
			
			Let
			\begin{equation*}
				\Omega_2:=\left\{(x,t):|x'|\le c_0v^{-1},\ |x_d|\le c_0v^{-2},\ |t|\le c_0(v\delta)^{-1}\right\}.
			\end{equation*}
			Since $|a_u|=1$ and $|b_z|=v$, the temporal phases are constant within each family. Moreover,
			\begin{equation*}
				(a_u)_d=\sqrt{1-|u|^2}=1+O(v^2),
			\end{equation*}
			while $(b_z)_d=0$. The same phase calculation as before therefore gives
			\begin{equation*}
				|Ef_1|\gtrsim M_1,
				\qquad
				|Ef_2|\gtrsim M_2
			\end{equation*}
			on $\Omega_2$. Since
			\begin{equation*}
				|\Omega_2|\sim v^{-(d-1)}v^{-2}(v\delta)^{-1}=v^{-(d+2)}\delta^{-1},
			\end{equation*}
			we obtain
			\begin{equation*}
				\|Ef_1Ef_2\|_2
				\gtrsim M_1M_2v^{-\frac{d+2}{2}}\delta^{-1/2}.
			\end{equation*}
			On the other hand, \eqref{anisotropic-L4} gives
			\begin{equation*}
				D_\delta(f_1)
				\lesssim M_1^{1/2}\delta^{-\frac{d+2}{4}}v^{-1/4},
				\quad
				D_\delta(f_2)
				\lesssim M_2^{1/2}\delta^{-\frac{d+1}{4}}v^{-1/2}.
			\end{equation*}
			Since
			\begin{equation*}
				(M_1M_2)^{1/2}\sim\left(\frac v\delta\right)^{\frac{2d-3}{2}},
			\end{equation*}
			we conclude that
			\begin{equation}\label{second-sharp-lower}
				C(N_1,N_2)
				\gtrsim
				\left(\frac v\delta\right)^{\frac{2d-7}{4}}
				=N_2^{\frac{2d-7}{4}}.
			\end{equation}
			Combining \eqref{first-sharp-lower} and \eqref{second-sharp-lower} proves the theorem.

		\end{proof}
		
		\vspace{10pt}
		\section{Proof framework of Theorem \ref{thm:improved-bilinear-strichartz}}
		\label{sec:framework}
		
		We give the framework for Theorem~\ref{thm:improved-bilinear-strichartz}
		on the unit square torus. Let $d\ge6$ and $N_1\ge N_2\ge2$.
		For finitely supported
		coefficients $a:\mathbb Z^d\to\mathbb C$, define
		\begin{equation*}
			U_a(x,t)
			:=
			\sum_{k\in\mathbb Z^d}
			a(k)e^{2\pi i(x\cdot k+t|k|^2)},
			\qquad x\in\mathbb T^d,\quad t\in\mathbb T.
		\end{equation*}
		Then we see the relation
		\begin{equation*}
			e^{it\Delta}\phi=U_{\widehat{\phi}}(x,-\frac{t}{2\pi}).
		\end{equation*}
		All torus measures are normalized, and $L_t^pL_x^p$ denotes
		$L^p(\mathbb T;L^p(\mathbb T^d))$. Suppose that
		\begin{equation*}
			\operatorname{supp}a_j
			\subset
			\left\{
			k\in\mathbb Z^d:
			\bigl||k|-N_j\bigr|\le O_d(1)
			\right\},
			\qquad j=1,2.
		\end{equation*}
		We aim to prove
		\begin{equation}\label{eq:framework-target}
			\|U_{a_1}U_{a_2}\|_{L_t^2L_x^2}
			\lesssim_\varepsilon
			N_2^\varepsilon
			\left(
			N_2^{d-4}+\frac{N_2^{d-2}}{N_1}
			\right)^{1/2}
			\|a_1\|_{\ell^2}\|a_2\|_{\ell^2}.
		\end{equation}
		The corresponding estimate for $e^{it\Delta}$ follows by a fixed
		time rescaling and periodicity. We record below the statements
		used in the proof; their proofs are given in Section~\ref{sec:strichartz-proof}.
		
		\medskip
		\noindent\textbf{Step 1: Localize the high input in spatial frequency and energy.}
		
		Partition $\mathbb R^d$ into finitely overlapping cubes $Q$ of side
		length $N_2$, and partition $\mathbb R$ into 
		intervals $I$ of length $N_2$. Define
		\begin{equation*}
			a_{Q,I}(k)
			:=
			a_1(k)\mathbf 1_Q(k)\mathbf 1_I(|k|^2),
		\end{equation*}
		where $\mathbf 1_X$ denotes the characteristic function of the set $X$.
		
		The required orthogonality statement is
		\begin{equation}\label{eq:framework-high-sum}
			\|U_{a_1}U_{a_2}\|_{L_t^2L_x^2}^2
			\lesssim
			\sum_{Q,I}
			\|U_{a_{Q,I}}U_{a_2}\|_{L_t^2L_x^2}^2,
		\end{equation}
		with
		\begin{equation*}
			\sum_{Q,I}\|a_{Q,I}\|_{\ell^2}^2
			\sim
			\|a_1\|_{\ell^2}^2.
		\end{equation*}
		Thus it remains to estimate each localized high input uniformly.
		
		\medskip
		\noindent\textbf{Step 2: Identify the strip containing the high input.}
		
		Fix a nonempty piece $a_{Q,I}$, choose
		$k_0\in\operatorname{supp}a_{Q,I}$, and set
		\begin{equation*}
			v:=\frac{k_0}{|k_0|}.
		\end{equation*}
		The spatial and energy localizations imply
		\begin{equation}\label{eq:framework-high-strip}
			\operatorname{supp}a_{Q,I}
			\subset
			\left\{
			k\in\mathbb Z^d:
			|v\cdot(k-k_0)|
			\lesssim 1+\frac{N_2^2}{N_1}
			\right\}.
		\end{equation}
		This strip structure determines the next decomposition.
		The high input has spatial frequency diameter $O(N_2)$, but its
		width in the direction $v$ is only
		$O(1+N_2^2/N_1)$. We therefore divide the low input into strips
		of this same width and direction. Adding the high support
		enlarges each low strip by only a fixed multiple of its width,
		so the corresponding products have bounded overlap in spatial
		frequency. This gives almost $L^2$-orthogonality.
		The strip restriction also gives the improved linear estimate
		in Step 5.
		
		\medskip
		\noindent\textbf{Step 3: Decompose the low input into matching strips.}
		
		For this fixed high piece, partition $\mathbb R$ into disjoint
		intervals $J_\nu$ of length $1+N_2^2/N_1$, and define
		\begin{equation*}
			b_\nu(l)
			:=
			a_2(l)\mathbf 1_{J_\nu}(v\cdot l).
		\end{equation*}
		The required orthogonality statement is
		\begin{equation}\label{eq:framework-low-sum}
			\|U_{a_{Q,I}}U_{a_2}\|_{L_t^2L_x^2}^2
			\lesssim
			\sum_\nu
			\|U_{a_{Q,I}}U_{b_\nu}\|_{L_t^2L_x^2}^2,
		\end{equation}
		with
		\begin{equation*}
			\sum_\nu\|b_\nu\|_{\ell^2}^2
			\sim
			\|a_2\|_{\ell^2}^2.
		\end{equation*}
		The constant is uniform in the chosen high piece and its direction.
		
		\medskip
		\noindent\textbf{Step 4: Estimate the localized high input in $L_t^4L_x^4$.}
		
		The first principal linear estimate is the following. If $E_0\in\mathbb Z$, $k_0\in \Z^d$, $\bigl||k_0|-N_1\bigr|\le C$, and
		\begin{equation*}
			\operatorname{supp}a
			\subset
			\left\{
			k\in\mathbb Z^d:
			|k-k_0|\lesssim N_2,\quad
			\bigl||k|-N_1\bigr|\lesssim 1,\quad
			\bigl||k|^2-E_0\bigr|\lesssim N_2
			\right\},
		\end{equation*}
		then, for $d\ge6$,
		\begin{equation}\label{eq:framework-high-fourth}
			\|U_a\|_{L_t^4L_x^4}^4
			\lesssim_\varepsilon
			N_2^\varepsilon
			\left(
			N_2^{d-4}+\frac{N_2^{d-2}}{N_1}
			\right)
			\|a\|_{\ell^2}^4.
		\end{equation}
		The constant is uniform in $k_0$ and $E_0$.
		Every piece $a_{Q,I}$ satisfies these hypotheses with fixed
		enlarged constants.
		
		\medskip
		\noindent\textbf{Step 5: Estimate the low input restricted to a strip.}
		
		The second principal linear estimate holds in every dimension
		$d\ge5$. Let $v$ be any unit vector, $A\in\mathbb R$, and
		$1\le h\lesssim N_2$. If
		\begin{equation*}
			\operatorname{supp}b
			\subset
			\left\{
			l\in\mathbb Z^d:
			\bigl||l|-N_2\bigr|\lesssim 1,\quad
			|v\cdot l-A|\lesssim h
			\right\},
		\end{equation*}
		then
		\begin{equation}\label{eq:framework-low-fourth}
			\|U_b\|_{L_t^4L_x^4}^4
			\lesssim_\varepsilon
			N_2^{d-4+\varepsilon}h
			\|b\|_{\ell^2}^4,
		\end{equation}
		uniformly in $v$, $A$, and $h$.
		
		For each $b_\nu$ from Step 3, take $h:=1+N_2^2/N_1$.
		Since $N_1\ge N_2\ge 1$, this choice satisfies $1\le h\lesssim N_2$.
		Consequently,
		\begin{equation*}
			\|U_{b_\nu}\|_{L_t^4L_x^4}^4
			\lesssim_\varepsilon
			N_2^\varepsilon
			\left(
			N_2^{d-4}+\frac{N_2^{d-2}}{N_1}
			\right)
			\|b_\nu\|_{\ell^2}^4.
		\end{equation*}
		Thus the two linear estimates have the same $L^4$-norm
		coefficient at the strip width dictated by the high input.
		
		\medskip
		\noindent\textbf{Step 6: Combine the estimates and sum the pieces.}
		
		H\"older's inequality together with
		\eqref{eq:framework-high-fourth} and
		\eqref{eq:framework-low-fourth} gives
		\begin{equation*}
			\|U_{a_{Q,I}}U_{b_\nu}\|_{L_t^2L_x^2}^2
			\lesssim_\varepsilon
			N_2^\varepsilon
			\left(
			N_2^{d-4}+\frac{N_2^{d-2}}{N_1}
			\right)
			\|a_{Q,I}\|_{\ell^2}^2\|b_\nu\|_{\ell^2}^2.
		\end{equation*}
		Applying \eqref{eq:framework-low-sum} and then
		\eqref{eq:framework-high-sum}, we obtain
		\begin{equation*}
			\|U_{a_1}U_{a_2}\|_{L_t^2L_x^2}^2
			\lesssim_\varepsilon
			N_2^\varepsilon
			\left(
			N_2^{d-4}+\frac{N_2^{d-2}}{N_1}
			\right)
			\|a_1\|_{\ell^2}^2\|a_2\|_{\ell^2}^2.
		\end{equation*}
		Taking square roots and relabeling $\varepsilon$ gives
		\eqref{eq:framework-target}. 
		
		\vspace{10pt}
		\section{Proof of the shell-type bilinear Strichartz estimate}
		\label{sec:strichartz-proof}
		
		We prove the five assertions in Section~\ref{sec:framework}, in the
		order in which they are used there. One may also assume that $N_2$ is sufficiently large to avoid triviality. Replacing $N_2$ by a comparable dyadic value in the partitions and
		estimates allows us to further assume that it is dyadic. 
		
		\medskip\noindent\textbf{Step 1: Orthogonality for the high-frequency decomposition.}
		
		The space-time Fourier support of $U_{a_{Q,I}}U_{a_2}$ is contained in
		\begin{equation*}
			\left(Q+B(0,CN_2)\right)
			\times\left(I+[N_2^2-CN_2,N_2^2+CN_2]\right).
		\end{equation*}
		The first coordinate here is spatial frequency and the second is temporal
		frequency. The enlarged cubes have bounded overlap, as do the enlarged
		intervals. At each space-time Fourier frequency, Cauchy--Schwarz in the
		bounded number of contributing pairs $(Q,I)$, followed by Plancherel,
		proves \eqref{eq:framework-high-sum}. 
		
		\medskip\noindent\textbf{Step 2: The strip containing the high-frequency support.}
		
		For $k,k_0\in\operatorname{supp}a_{Q,I}$, we have
		$|k-k_0|\lesssim N_2$ and $\bigl||k|^2-|k_0|^2\bigr|\le N_2$.
		Since $|k_0|\sim N_1$, the identity
		\begin{equation*}
			2|k_0|v\cdot(k-k_0)=|k|^2-|k_0|^2-|k-k_0|^2
		\end{equation*}
		gives
		\begin{equation*}
			|v\cdot(k-k_0)|\lesssim\frac{N_2+N_2^2}{N_1}
			\lesssim1+\frac{N_2^2}{N_1}.
		\end{equation*}
		This proves \eqref{eq:framework-high-strip}. It also explains why the
		next decomposition uses strips perpendicular to $v$: their width is
		chosen so that adding the high-frequency support preserves bounded
		overlap.
		
		\medskip\noindent\textbf{Step 3: Orthogonality for the low-frequency strips.}
		
		If $n$ is the spatial frequency of $U_{a_{Q,I}}U_{b_\nu}$, then
		\begin{equation*}
			v\cdot n\in v\cdot k_0+J_\nu
			+\left[-C\left(1+\frac{N_2^2}{N_1}\right),
			C\left(1+\frac{N_2^2}{N_1}\right)\right].
		\end{equation*}
		These intervals have bounded overlap. Spatial Plancherel and
		Cauchy--Schwarz therefore give, for every $t$,
		\begin{equation*}
			\|U_{a_{Q,I}}(\cdot,t)U_{a_2}(\cdot,t)\|_{L_x^2}^2
			\lesssim\sum_\nu
			\|U_{a_{Q,I}}(\cdot,t)U_{b_\nu}(\cdot,t)\|_{L_x^2}^2.
		\end{equation*}
		Integration in $t$ proves \eqref{eq:framework-low-sum}. 
		
		\medskip\noindent\textbf{Step 4: The localized high-frequency estimate.}
		
		We first prove a geometric estimate used to integrate the Gaussian
		kernels below. 
		
		\begin{lemma}\label{lem:gaussian-line}
			For $c>0$, $M\gtrsim1$, $L\ge0$, $x\in\mathbb R^d$, and a unit vector $v$,
			\begin{equation*}
				\int_0^L\sum_{m\in\mathbb Z^d}
				e^{-cM^2|x+uv-m|^2}\,du\lesssim_{c}\frac{1+L}{M},
			\end{equation*}
			where the implicit constant is independent of $v$.
		\end{lemma}
	\begin{proof}
		Partitioning \([0,L]\) into at most \(1+L\) intervals of
		length at most one, it suffices to consider \(0<L\leq1\).
		Let \(S:=\{x+uv:0\leq u\leq L\}\). Splitting the exponent
		and integrating along the line $S$ gives
		\begin{equation*}
			\int_0^L e^{-cM^2|x+uv-m|^2}\,du
			\leq
			e^{-\frac c2M^2\operatorname{dist}(m,S)^2}
			\int_{\mathbb R}e^{-\frac c2M^2|x+uv-m|^2}\,du
			\lesssim_c
			M^{-1}e^{-\frac c2M^2\operatorname{dist}(m,S)^2},
		\end{equation*}
		uniformly in \(|v|=1\). Since \(S\subset B(x,1)\), elementary
		lattice counting yields
		\begin{equation*}
			\#\{m\in\mathbb Z^d:\operatorname{dist}(m,S)<j+1\}
			\leq \#\bigl(\mathbb Z^d\cap B(x,j+2)\bigr)
			\lesssim_d (j+3)^d.
		\end{equation*}
		Thus summing over the distance
		ranges \(j\leq\operatorname{dist}(m,S)<j+1\), we obtain
		\begin{equation*}
			\int_0^L\sum_{m\in\mathbb Z^d}
			e^{-cM^2|x+uv-m|^2}\,du
			\lesssim_{c,d}
			\frac1M\sum_{j=0}^\infty
			(j+3)^d e^{-cj^2/2}
			\lesssim_{c,d}\frac1M.
		\end{equation*}

	\end{proof}
		
		We now prove \eqref{eq:framework-high-fourth}. The decomposition and  correction are inspired by the
		circle-method argument in \cite{pezzi2026sharp}. Take $k_0\in\mathbb Z^d$
		as in the framework and normalize $\|a\|_{\ell^2}=1$. We may also assume
		$E_0$ is an integer. Choose a cutoff $\chi\in C_c^\infty(\mathbb R)$ and set
		\begin{equation*}
			G(x,t):=\sum_{k\in\mathbb Z^d}e^{-\pi|k-k_0|^2/N_2^2}
			e^{2\pi i(x\cdot k+t|k|^2)},\qquad
			D_{N_2}(t):=\sum_{\omega\in\mathbb Z}
			\chi\left(\frac{\omega-E_0}{N_2}\right)e^{2\pi i\omega t}.
		\end{equation*}
		Define $K:=D_{N_2}*_tG$, where $*_t$ denotes convolution in time only.
		Convolution by $K$ below is in both space and time, and is written
		$*_{x,t}$. Its multiplier is
		\begin{equation}\label{eq:filtered-multiplier}
			\widehat K(k,\omega)=e^{-\pi|k-k_0|^2/N_2^2}
			\chi\left(\frac{\omega-E_0}{N_2}\right)
			\mathbf1_{\{\omega=|k|^2\}}.
		\end{equation}
		This multiplier is nonnegative and bounded from below on the Fourier support
		of $U_a$. For $E_\alpha:=\{(x,t):|U_a(x,t)|>\alpha\}$, let
		$f_\alpha:=\mathbf1_{E_\alpha}U_a/|U_a|$. Weighted Cauchy--Schwarz and space-time Plancherel yield
		\begin{equation}\label{eq:level-set-kernel}
			\alpha^2|E_\alpha|^2
			\le|\langle U_a,f_\alpha\rangle_{L_t^2L_x^2}|^2
			\lesssim\langle K*_{x,t}f_\alpha,f_\alpha\rangle_{L_t^2L_x^2}.
		\end{equation}
		
		We next estimate this convolution operator. Recall the following Poisson summation formula:
		\begin{equation*}
			\sum_{n\in\mathbb Z^m}g(n+y)
			=\sum_{\ell\in\mathbb Z^m}\widehat g(\ell)e^{2\pi i\ell\cdot y},
			\qquad
			\widehat g(\xi):=\int_{\mathbb R^m}g(z)e^{-2\pi iz\cdot\xi}\,dz,
			\quad g\in\mathcal S(\mathbb R^m).
		\end{equation*}
		Since $E_0$ is an integer, this gives
		\begin{equation*}
			D_{N_2}(t)=e^{2\pi iE_0t}N_2
			\sum_{j\in\mathbb Z}\widehat\chi\bigl(N_2(j-t)\bigr).
		\end{equation*}
		The rapid decay of $\widehat\chi$ and integration over $\T$ give
		\begin{equation}\label{eq:energy-kernel-bounds}
			|D_{N_2}(t)|\lesssim_A
			N_2(1+N_2\operatorname{dist}(t,\mathbb Z))^{-A},
			\qquad \|D_{N_2}\|_{L_t^1(\mathbb T)}\lesssim1, \quad \forall A\gg1.
		\end{equation}
		As a direct consequence, for an integer $1\le q\le N_2$ and $\varphi\in \R$,
		\begin{equation}\label{eq:energy-rational-sum}
			\sum_{\substack{r\bmod q\\(r,q)=1}}
			|D_{N_2}(t-r/q-\varphi)|
			\lesssim_A N_2\sum_{\ell\in\mathbb Z}
			(1+N_2|t-\varphi-\ell/q|)^{-A}
			\lesssim N_2.
		\end{equation}
		
		Choose an even $\eta\in C_c^\infty(\mathbb R)$, nonincreasing on
		$[0,\infty)$, equal to one on $[-1,1]$, supported in $[-2,2]$, and
		with $0\le\eta\le1$. Let $\eta_0(t)$ be the periodization of
		$\eta(N_2^2t)$ and set $K_0:=D_{N_2}*_t(\eta_0G)$.
		The Fourier coefficients of $\eta_0$ have absolute value $O(N_2^{-2})$,
		so
		\begin{equation*}
			\|K_0*_{x,t}\|_{L_t^2L_x^2\to L_t^2L_x^2}\lesssim N_2^{-2}.
		\end{equation*}

		Translating \(k=k_0+n\) yields
		\begin{equation*}
			G(x,s)
			=e^{2\pi i(x\cdot k_0+s|k_0|^2)}
			\sum_{n\in\mathbb Z^d}
			e^{-\pi(N_2^{-2}-2is)|n|^2}
			e^{2\pi i n\cdot(x+2sk_0)}.
		\end{equation*}
		Also recall the Fourier transform of a Gaussian:
		\begin{equation*}
			\mathcal F_{\mathbb R^d}
			\bigl(e^{-\pi z|\cdot|^2}\bigr)(\xi)
			=z^{-d/2}e^{-\pi|\xi|^2/z},
			\qquad \operatorname{Re}z>0,
		\end{equation*}
		Poisson summation therefore gives the identity
		\begin{equation*}
			G(x,s)
			=
			e^{2\pi i(x\cdot k_0+s|k_0|^2)}
			(N_2^{-2}-2is)^{-d/2}
			\sum_{m\in\mathbb Z^d}
			\exp\left(
			-\frac{\pi|x+2sk_0-m|^2}{N_2^{-2}-2is}
			\right).
		\end{equation*}
		Using
		\begin{equation*}
			|N_2^{-2}-2is|^{-d/2}
			=\frac{N_2^d}{(1+4N_2^4s^2)^{d/4}},
			\qquad
			\operatorname{Re}\frac1{N_2^{-2}-2is}
			=\frac{N_2^2}{1+4N_2^4s^2},
		\end{equation*}
		we obtain 
		\begin{equation*}
			|G(x,s)|
			\leq
			\frac{N_2^d}{(1+4N_2^4s^2)^{d/4}}
			\sum_{m\in\mathbb Z^d}
			\exp\left(
			-\frac{\pi N_2^2|x+2sk_0-m|^2}
			{1+4N_2^4s^2}
			\right)
		\end{equation*}
		\begin{equation*}
			\le
			N_2^d\sum_{m\in\mathbb Z^d}
			e^{-cN_2^2|x+2sk_0-m|^2},
		\end{equation*}
		where the last inequality uses
		\(|s|\leq2N_2^{-2}\), because of the cutoff $\eta_0$.

		Using \eqref{eq:energy-kernel-bounds}, then Lemma~\ref{lem:gaussian-line}
		with $u:=2|k_0|s$, gives
		\begin{equation*}
			\|K_0\|_{L_t^\infty L_x^\infty}
			\lesssim N_2^{d+1}\frac{1+N_1/N_2^2}{N_1N_2}
			=N_2^{d-2}+\frac{N_2^d}{N_1}.
		\end{equation*}
		Interpolating the resulting $L_t^1L_x^1\to L_t^\infty L_x^\infty$
		bound with the $L_t^2L_x^2$ operator bound gives
		\begin{equation}\label{eq:central-operator}
			\|K_0*_{x,t}\|_{L_t^{4/3}L_x^{4/3}\to L_t^4L_x^4}
			\lesssim\left(N_2^{d-4}+\frac{N_2^{d-2}}{N_1}\right)^{1/2}.
		\end{equation}
		
		Next, we estimate the contributions localized near the other reduced fractions. For dyadic $1\le Q\le N_2/128$ and reduced fractions $r/q$ ($0\le r<q$, $(r,q)=1$) with $Q\le q<2Q$, we consider the periodized cutoff
		$\eta(N_2Q(t-r/q))$. 
		
		Distinct such cutoffs have disjoint supports:
		the distance between their centers is at least $1/(qq')$, whereas the
		sum of their radii is at most $4(q+q')/(N_2qq')<1/(qq')$.
		For dyadic $Q\le S<N_2$, put
		\begin{equation*}
			h_S(\varphi):=\eta(N_2S\varphi)-\eta(2N_2S\varphi),
			\qquad h_{N_2}(\varphi):=\eta(N_2^2\varphi).
		\end{equation*}
		Then their sum over $S=Q,2Q,\ldots,N_2$ is just  $\eta(N_2Q\varphi)$.
		
		We define
		\begin{equation*}
			\eta_{Q,S}(t)
			:=
			\sum_{\substack{q\in\mathbb N\\Q\leq q<2Q}}
			\ \sum_{\substack{0\leq r<q,\ (r,q)=1\\
					(q,r,S)\neq(1,0,N_2)}}
			\ \sum_{j\in\mathbb Z}
			h_S\left(t-\frac rq-j\right).
		\end{equation*}
		The excluded term is precisely
		\(\sum_{j\in\mathbb Z}h_{N_2}(t-j)=\eta_0(t)\),
		which has already been treated.
		
		Then we can decompose the convolution kernel $K$ as follows: 
		\begin{equation*}
			\rho:=1-\eta_0-\sum_{Q,S}\eta_{Q,S},\qquad
			K_{Q,S}:=D_{N_2}*_t(\eta_{Q,S}G),\qquad
			K_\rho:=D_{N_2}*_t(\rho G).
		\end{equation*}
		Directly, the functions $\rho$, \(\eta_0\) and \(\eta_{Q,S}\) are  all nonnegative. Also note that
		\begin{equation*}
			\int_{\mathbb T}
			\left(\eta_0(t)+\sum_{Q,S}\eta_{Q,S}(t)\right)\,dt
			=
			\sum_{\substack{Q\leq N_2/128\\Q\ \mathrm{dyadic}}}
			\sum_{Q\leq q<2Q}
			\sum_{\substack{0\leq r<q\\(r,q)=1}}
			\frac{1}{N_2Q}\int_{\mathbb R}\eta(u)\,du
			\leq
			\frac8{N_2}
			\sum_{\substack{Q\leq N_2/128\\Q\ \mathrm{dyadic}}}Q
			\leq\frac18.
		\end{equation*}
	    Consequently, we have
		\begin{equation*}
			\widehat\rho(0)
			=1-\int_{\mathbb T}
			\left(\eta_0(t)+\sum_{Q,S}\eta_{Q,S}(t)\right)\,dt
			\geq\frac78.
		\end{equation*}
		
		We next estimate \(K_{Q,S}\) and \(K_\rho\). For \(K_{Q,S}\),
		we first prove a pointwise bound and then control its corresponding Fourier
		multiplier. The latter estimate will require subtracting a
		suitable multiple of \(K_\rho\). Interpolating these two bounds
		will give the operator estimate needed in
		\eqref{eq:level-set-kernel}. 
		
		\medskip\noindent
		\textbf{The pointwise estimate for \(G(x,r/q+\varphi)\).}
		Fix relatively prime integers \(r,q\), with \(0\leq r<q\).
		We first derive an identity valid for every
		\(\varphi\in\mathbb R\), and then restrict \(\varphi\) to the
		support of \(h_S\).
        For fixed $x,\varphi$, define
        \begin{equation*}
            g(y):=e^{-\pi(N_2^{-2}-2i\varphi)|y|^2+2\pi iy\cdot(x+2\varphi k_0)},
            \qquad y\in\mathbb R^d.
        \end{equation*}
        Translating $k=k_0+n$ and separating $n$ into residue classes gives
        \begin{equation*}
            G(x,r/q+\varphi)
            =e^{2\pi i(x\cdot k_0+(r/q+\varphi)|k_0|^2)}
            \sum_{b\in\{0,\ldots,q-1\}^d}
            e^{2\pi ir(|b|^2+2k_0\cdot b)/q}
            \sum_{n\in b+q\mathbb Z^d}g(n).
        \end{equation*}
        The residue-class form of Poisson summation is
        \begin{equation*}
            \sum_{n\in b+q\mathbb Z^d}g(n)
            =q^{-d}\sum_{m\in\mathbb Z^d}
            e^{2\pi im\cdot b/q}\widehat g(m/q).
        \end{equation*}
        Here the Gaussian Fourier transform gives
        \begin{equation*}
            \widehat g(m/q)
            =(N_2^{-2}-2i\varphi)^{-d/2}
            \exp\left(-\frac{\pi|x+2\varphi k_0-m/q|^2}
            {N_2^{-2}-2i\varphi}\right).
        \end{equation*}
        Consequently,
        \begin{equation}\label{eq:rational-gaussian-identity}
            \frac{q^dG(x,r/q+\varphi)}
            {e^{2\pi i(x\cdot k_0+(r/q+\varphi)|k_0|^2)}}
            =\sum_{m\in\mathbb Z^d}\widehat g(m/q)
            \sum_{b\in\{0,\ldots,q-1\}^d}
            e^{2\pi i(r|b|^2+(2rk_0+m)\cdot b)/q}.
        \end{equation}

		To estimate the last sum, observe that, for every integer
		\(\ell\),
		\begin{equation*}
		\left|
				\sum_{w=0}^{q-1}
				e^{2\pi i(rw^2+\ell w)/q}
				\right|^2
				=
				\sum_{h=0}^{q-1}
				e^{2\pi i(rh^2+\ell h)/q}
				\sum_{w=0}^{q-1}e^{2\pi i(2rh)w/q}
	\end{equation*}
	then we obtain
	\begin{equation*}
		\left|
				\sum_{w=0}^{q-1}
				e^{2\pi i(rw^2+\ell w)/q}
				\right|^2
\leq
				q\,\#\{0\leq h<q:q\mid 2rh\}
				\leq2q.
	\end{equation*}
		Here the inner sum vanishes unless \(q\mid2rh\), and we used
		\((r,q)=1\). Factoring the \(d\)-dimensional sum into its
		coordinate sums therefore yields
		\begin{equation*}
			\left|
			\sum_{b\in\{0,\ldots,q-1\}^d}
			e^{\frac{2\pi i}{q}
				(r|b|^2+(2rk_0+m)\cdot b)}
			\right|
			\leq (2q)^{d/2}.
		\end{equation*}
		Taking absolute values in
		\eqref{eq:rational-gaussian-identity}, we consequently obtain
		\begin{equation}\label{eq:rational-gaussian-general}
		|G(x,r/q+\varphi)|
				\lesssim
				q^{-d/2}
				\frac{N_2^d}{(1+4N_2^4\varphi^2)^{d/4}}
				\sum_{m\in\mathbb Z^d}
				\exp\left(
				-\frac{\pi N_2^2|x+2\varphi k_0-m/q|^2}
				{1+4N_2^4\varphi^2}
				\right).
	\end{equation}
		
		For \(S<N_2\), the definition of \(h_S\) implies
		\begin{equation*}
			\operatorname{supp}h_S
			\subset
			\left\{\varphi:
			\frac1{2N_2S}\leq|\varphi|\leq\frac2{N_2S}
			\right\},
		\end{equation*}
		whereas
		\(\operatorname{supp}h_{N_2}
		\subset\{|\varphi|\leq2N_2^{-2}\}\).
		In either case, on \(\operatorname{supp}h_S\),
		\begin{equation*}
			1+4N_2^4\varphi^2\sim\left(\frac{N_2}{S}\right)^2,
			\qquad
			\frac{N_2^d}{(1+4N_2^4\varphi^2)^{d/4}}
			\sim(N_2S)^{d/2},
			\qquad
			\frac{N_2^2}{1+4N_2^4\varphi^2}\sim S^2.
		\end{equation*}
		Thus \eqref{eq:rational-gaussian-general} gives
		\begin{equation}
			\label{eq:rational-gaussian}
			|G(x,r/q+\varphi)|
			\lesssim
			q^{-d/2}(N_2S)^{d/2}
			\sum_{m\in\mathbb Z^d}
			e^{-cS^2|x+2\varphi k_0-m/q|^2},
			\qquad \forall\varphi\in\operatorname{supp}h_S.
		\end{equation}
		Note that the right-hand side is independent of \(r\). 
		
		\medskip\noindent
		\textbf{The \(L_t^\infty L_x^\infty\) estimate for \(K_{Q,S}\).}
	Expanding $\eta_{Q,S}$, unfolding the time integral,
	and applying \eqref{eq:rational-gaussian} together with
	$0\le h_S\le1$, we obtain
		\begin{equation*}
			|K_{Q,S}(x,t)|
			\lesssim (N_2S)^{d/2}
			\sum_{Q\le q<2Q}q^{-d/2}
			\int_{|\varphi|\le2/(N_2S)}
			\left(
			\sum_{\substack{0\le r<q\\(r,q)=1}}
			|D_{N_2}(t-r/q-\varphi)|
			\right)
		\end{equation*}
		\begin{equation*}
				\times\sum_{m\in\mathbb Z^d}
			e^{-cS^2|x+2\varphi k_0-m/q|^2}\,d\varphi.
		\end{equation*}
		 By \eqref{eq:energy-rational-sum}, we also have
		\begin{equation*}
			\sum_{\substack{0\leq r<q\\(r,q)=1}}
			|D_{N_2}(t-r/q-\varphi)|
			\lesssim N_2.
		\end{equation*}
		It remains to estimate the integral of the Gaussian sum.
		
		Write \(v:=k_0/|k_0|\). Substituting 
		\(u:=2q|k_0|\varphi\) into Lemma \ref{lem:gaussian-line}, we obtain
		\begin{equation*}
		\int_{-2/(N_2S)}^{2/(N_2S)}
				\sum_{m\in\mathbb Z^d}
				e^{-cS^2|x+2\varphi k_0-m/q|^2}\,d\varphi
                =
				\frac1{2q|k_0|}
				\int_{-4q|k_0|/(N_2S)}^{4q|k_0|/(N_2S)}
				\sum_{m\in\mathbb Z^d}
				e^{-c(S/q)^2
					|qx+uv-m|^2}\,du
	\end{equation*}
	\begin{equation*}
		\lesssim
				\frac{1+q|k_0|/(N_2S)}{|k_0|S}
				\lesssim
				\frac{1+N_1q/(N_2S)}{N_1S}.
	\end{equation*}
		The first inequality comes from Lemma~\ref{lem:gaussian-line} and the second uses
		\(|k_0|\sim N_1\).
		Summing over \(Q\leq q<2Q\) now gives 
		\begin{equation*}
		\|K_{Q,S}\|_{L_t^\infty L_x^\infty}
				\lesssim
				N_2(N_2S)^{d/2}Q^{1-d/2}
				\frac{1+N_1Q/(N_2S)}{N_1S}
	\end{equation*}
	\begin{equation}\label{eq:rational-kernel-bound}
				=N_2(N_2S)^{d/2-1}Q^{1-d/2}
				\left(\frac{N_2}{N_1}+\frac QS\right).
	\end{equation}
		
		\medskip\noindent
		\textbf{The \(L_t^\infty L_x^\infty\) estimate for \(K_\rho\).}
		We next estimate the $L^\infty$-norm of $\rho G$.
		For any \(s\in [0,1]\), Dirichlet
		approximation provides relatively prime integers \(r,q\)
		such that
		\begin{equation*}
			1\leq q\leq N_2,
			\qquad
			\left|s-\frac rq\right|\leq\frac1{N_2q}.
		\end{equation*}
		If \(q<N_2/64\), for the dyadic number \(Q\) satisfying
		\(Q\leq q<2Q\), we have 
		\begin{equation*}
			N_2Q\left|s-\frac rq\right|\leq\frac Qq\leq1.
		\end{equation*}
		The corresponding cutoff
		\(\eta(N_2Q(s-r/q))\) is therefore equal to one,
		which forces \(\rho(s)=0\) (recall that in the earlier decomposition, we choose $1\le Q\le N_2/128$).
		Consequently, when \(\rho(s)\neq0\), the approximating
		fraction satisfies
		\begin{equation*}
			N_2/64\leq q\leq N_2,
			\qquad
			s=r/q+\varphi,
			\qquad
			|\varphi|\leq64N_2^{-2}.
		\end{equation*}
		Applying \eqref{eq:rational-gaussian-general} in this range,
		we find
		\begin{equation*}
			|G(x,s)|
			\lesssim
			q^{-d/2}N_2^d
			\sum_{m\in\mathbb Z^d}
			e^{-c(N_2/q)^2|qx+2q\varphi k_0-m|^2}
			\lesssim N_2^{d/2}.
		\end{equation*}
		The last Gaussian sum is uniformly bounded, since
		\(1\leq N_2/q\leq64\).
		Using \(\|D_{N_2}\|_{L_t^1}\lesssim1\), we conclude that
		\begin{equation}
			\label{eq:remainder-kernel}
			\|K_\rho\|_{L_t^\infty L_x^\infty}
			\leq
			\|D_{N_2}\|_{L_t^1}
			\|\rho G\|_{L_t^\infty L_x^\infty}
			\lesssim N_2^{d/2}.
		\end{equation}
		
		\medskip\noindent
		\textbf{The Fourier coefficients of the time cutoffs.}
		To obtain an \(L_t^2L_x^2\) operator bound, we now estimate
		the Fourier multiplier of \(K_{Q,S}\).
		It is given by
		\begin{equation*}
			\widehat K_{Q,S}(k,\omega)
			=
			e^{-\pi|k-k_0|^2/N_2^2}
			\chi\left(\frac{\omega-E_0}{N_2}\right)
			\widehat\eta_{Q,S}(\omega-|k|^2).
		\end{equation*}
		Thus we need bounds for
		\(\widehat\eta_{Q,S}(m)\), distinguishing \(m=0\) from
		\(m\neq0\).
		
		For an integer \(m\), define the Ramanujan sum
		\begin{equation*}
			c_q(m)
			:=
			\sum_{\substack{0\leq r<q\\(r,q)=1}}
			e^{-2\pi imr/q}.
		\end{equation*}
		Except for \((Q,S)=(1,N_2)\), where
		\(\eta_{Q,S}=0\), periodization gives
		\begin{equation*}
			\widehat\eta_{Q,S}(m)
			=
			\widehat h_S(m)\sum_{Q\leq q<2Q}c_q(m),
			\qquad
			\widehat h_S(m)
			=
			\int_{\mathbb R}h_S(\varphi)e^{-2\pi im\varphi}\,d\varphi.
		\end{equation*}
		The definitions of \(h_S\) and the rapid decay of
		\(\widehat\eta\) imply, for every \(A>0\),
		\begin{equation*}
			|\widehat h_S(m)|
			\lesssim_A
			\frac1{N_2S}
			\left(1+\frac{|m|}{N_2S}\right)^{-A}.
		\end{equation*}
		At \(m=0\), counting the numerators gives
		\begin{equation*}
			|\widehat\eta_{Q,S}(0)|
			\lesssim
			\frac1{N_2S}\sum_{Q\leq q<2Q}q
			\lesssim\frac{Q^2}{N_2S}.
		\end{equation*}
		
		For \(m\neq0\), the identity
		\begin{equation*}
			c_q(m)=\sum_{b\mid(q,m)}b\,\mu(q/b)
		\end{equation*}
		follows by inserting
		\(\mathbf1_{(r,q)=1}=\sum_{b\mid(r,q)}\mu(b)\)
		into the defining sum. Here \(\mu\) is the M\"obius
		function. Hence, writing \(\tau(|m|)\) for the number of
		positive divisors of \(|m|\),
		\begin{equation*}
			\sum_{Q\leq q<2Q}|c_q(m)|
			\leq
			\sum_{\substack{b\mid |m|\\b<2Q}}
			b\,\#\{q:Q\leq q<2Q,\ b\mid q\}
			\lesssim Q\tau(|m|).
		\end{equation*}
		For every \(\delta>0\), the divisor bound
		\(\tau(|m|)\lesssim_\delta |m|^{\delta/2}\) and
		\(N_2S\leq N_2^2\) give, with \(A>\delta/2\),
		\begin{equation*}
			|m|^{\delta/2}
			\left(1+\frac{|m|}{N_2S}\right)^{-A}
			\lesssim_\delta
			(N_2S)^{\delta/2}
			\leq N_2^\delta.
		\end{equation*}
		We have therefore proved
		\begin{equation}
			\label{eq:arc-fourier}
			|\widehat\eta_{Q,S}(0)|
			\lesssim\frac{Q^2}{N_2S},
			\qquad
			|\widehat\eta_{Q,S}(m)|
			\lesssim_\delta
			N_2^\delta\frac{Q}{N_2S},
			\quad m\neq0.
		\end{equation}
		
		Since
		\(\rho=1-\eta_0-\sum_{Q,S}\eta_{Q,S}\), for \(m\neq0\)
		we also have
		\begin{equation*}
			|\widehat\rho(m)|
			\leq|\widehat\eta_0(m)|
			+\sum_{Q,S}|\widehat\eta_{Q,S}(m)|.
		\end{equation*}
		Using \(|\widehat\eta_0(m)|\lesssim N_2^{-2}\) and
		\eqref{eq:arc-fourier} with exponent \(\delta/2\), we obtain
		\begin{equation*}
			|\widehat\rho(m)|
			\lesssim_\delta
			N_2^{-2}
			+N_2^{-1+\delta/2}
			\sum_Q\sum_{S=Q,2Q,\ldots,N_2}\frac QS
			\lesssim_\delta
			N_2^{-1+\delta},
			\qquad m\neq0.
		\end{equation*}
		Here the inner dyadic sum is bounded and there are
		\(O(1+\log N_2)\) choices of \(Q\).
		
		\medskip\noindent
		\textbf{The subtraction needed for the \(L_t^2L_x^2\) bound.}
		The first bound in \eqref{eq:arc-fourier} is larger than
		the second by a factor \(Q\). It applies to the multiplier
		precisely when \(\omega=|k|^2\), so we cannot use the
		second bound alone to estimate convolution by \(K_{Q,S}\).
		We remove this value by setting
		\begin{equation*}
			b_{Q,S}
			:=
			\frac{\widehat\eta_{Q,S}(0)}{\widehat\rho(0)},
			\qquad
			\widetilde K_{Q,S}
			:=
			K_{Q,S}-b_{Q,S}K_\rho.
		\end{equation*}
		Since \(\widehat\rho(0)\geq7/8\),
		\begin{equation*}
			0\leq b_{Q,S}\lesssim\frac{Q^2}{N_2S},
			\qquad
			\widehat\eta_{Q,S}(0)-b_{Q,S}\widehat\rho(0)=0.
		\end{equation*}
		For \(m\neq0\), the preceding Fourier coefficient bounds
		give 
		\begin{equation*}
		|\widehat\eta_{Q,S}(m)-b_{Q,S}\widehat\rho(m)|
				\lesssim_\delta
				N_2^\delta\frac{Q}{N_2S}
				+\frac{Q^2}{N_2S}N_2^{-1+\delta}
	    \lesssim N_2^\delta\frac{Q}{N_2S},
	\end{equation*}
		where we used \(Q\leq N_2\).
		Thus the multiplier
		\begin{equation*}
			\widehat{\widetilde K}_{Q,S}(k,\omega)
			=
			e^{-\pi|k-k_0|^2/N_2^2}
			\chi\left(\frac{\omega-E_0}{N_2}\right)
			\left[
			\widehat\eta_{Q,S}(\omega-|k|^2)
			-b_{Q,S}\widehat\rho(\omega-|k|^2)
			\right]
		\end{equation*}
		vanishes when \(\omega=|k|^2\) and satisfies the desired
		bound otherwise. By space-time Plancherel,
		\begin{equation*}
			\|\widetilde K_{Q,S}*_{x,t}\|_
			{L_t^2L_x^2\to L_t^2L_x^2}
			\lesssim_\delta
			N_2^\delta\frac{Q}{N_2S}.
		\end{equation*}
		
		We must also check that this subtraction preserves
		\eqref{eq:rational-kernel-bound}. By
		\eqref{eq:remainder-kernel},
		\begin{equation*}
			b_{Q,S}\|K_\rho\|_{L_t^\infty L_x^\infty}
			\lesssim
			\frac{Q^2N_2^{d/2}}{N_2S}
			\leq
			N_2(N_2S)^{d/2-1}Q^{1-d/2}\frac QS.
		\end{equation*}
		Indeed, the ratio of the middle expression to the last is
		\((Q/N_2)(Q/S)^{d/2-1}\leq1\).
		Consequently,
		\begin{equation*}
			\|\widetilde K_{Q,S}\|_{L_t^\infty L_x^\infty}
			\lesssim
			N_2(N_2S)^{d/2-1}Q^{1-d/2}
			\left(\frac{N_2}{N_1}+\frac QS\right).
		\end{equation*}
		Interpolating the resulting
		\(L_t^1L_x^1\to L_t^\infty L_x^\infty\) convolution
		bound with the \(L_t^2L_x^2\) bound gives 
		\begin{equation*}
		\|\widetilde K_{Q,S}*_{x,t}\|_
				{L_t^{4/3}L_x^{4/3}\to L_t^4L_x^4}^{\,2}
               \lesssim_\delta
				N_2^\delta
				\left[
				\frac{N_2^{d/2}}{N_1}
				\left(\frac SQ\right)^{d/2-2}
				+
				N_2^{d/2-1}
				\left(\frac SQ\right)^{d/2-3}
				\right]
	\end{equation*}
	\begin{equation}\label{eq:arc-operator}
		  \lesssim
				N_2^\delta
				\left(
				\frac{N_2^{d-2}}{N_1}+N_2^{d-4}
				\right).
	\end{equation}
		For the last inequality, both exponents of \(S/Q\) are
		nonnegative when \(d\geq6\), and \(1\leq S/Q\leq N_2\).
		
		\medskip\noindent
		\textbf{The estimate for the distribution of \(|U_a|\).}
		We now return to \eqref{eq:level-set-kernel}.
		Since
		\begin{equation*}
			\sum_{Q,S}b_{Q,S}
			=
			\frac{\sum_{Q,S}\widehat\eta_{Q,S}(0)}
			{\widehat\rho(0)}
			\lesssim
			\frac1{N_2}
			\sum_{\substack{Q\leq N_2/128\\Q\ \mathrm{dyadic}}}
			Q^2\sum_{S=Q,2Q,\ldots,N_2}\frac1S
			\lesssim
			\frac1{N_2}
			\sum_{\substack{Q\leq N_2/128\\Q\ \mathrm{dyadic}}}Q
			\lesssim1,
		\end{equation*}
		the definition of \(\widetilde K_{Q,S}\) gives the exact
		decomposition
		\begin{equation*}
			K
			=
			K_0+\sum_{Q,S}\widetilde K_{Q,S}
			+\left(1+\sum_{Q,S}b_{Q,S}\right)K_\rho.
		\end{equation*}
		There are \(O((1+\log N_2)^2)\) pairs \((Q,S)\).
		Applying \eqref{eq:central-operator} and
		\eqref{eq:arc-operator} with a sufficiently small positive
		exponent, and absorbing the logarithmic factor, we obtain,
		for every \(\delta>0\),
		\begin{equation*}
			\left\|
			\left(K_0+\sum_{Q,S}\widetilde K_{Q,S}\right)
			*_{x,t}
			\right\|_{L_t^{4/3}L_x^{4/3}\to L_t^4L_x^4}
			\lesssim_\delta
			N_2^{\delta/2}
			\left(N_2^{d-4}+\frac{N_2^{d-2}}{N_1}\right)^{1/2}.
		\end{equation*}
		Recall that \(E_\alpha=\{(x,t):|U_a(x,t)|>\alpha\}\)
		and \(|f_\alpha|=\mathbf1_{E_\alpha}\). Hence
		\begin{equation*}
			\|f_\alpha\|_{L_t^{4/3}L_x^{4/3}}
			=|E_\alpha|^{3/4},
			\qquad
			\|f_\alpha\|_{L_t^1L_x^1}=|E_\alpha|.
		\end{equation*}
	Using the decomposition of \(K\), H\"older's inequality,
	and \(1+\sum_{Q,S}b_{Q,S}\lesssim1\), we obtain from
	\eqref{eq:level-set-kernel} 
	\begin{equation*}
		\alpha^2|E_\alpha|^2
			\lesssim
			\left|
			\left\langle
			\left(K_0+\sum_{Q,S}\widetilde K_{Q,S}\right)
			*_{x,t}f_\alpha,f_\alpha
			\right\rangle_{L_t^2L_x^2}
			\right|
			+\left|
			\langle K_\rho*_{x,t}f_\alpha,f_\alpha
			\rangle_{L_t^2L_x^2}
			\right|
	\end{equation*}
	\begin{equation*}
		\lesssim
			\left\|
			\left(K_0+\sum_{Q,S}\widetilde K_{Q,S}\right)
			*_{x,t}f_\alpha
			\right\|_{L_t^4L_x^4}
			\|f_\alpha\|_{L_t^{4/3}L_x^{4/3}}
			+
			\|K_\rho\|_{L_t^\infty L_x^\infty}
			\|f_\alpha\|_{L_t^1L_x^1}^{\,2}
	\end{equation*}
	\begin{equation*}
		\lesssim_\delta
			N_2^{\delta/2}
			\left(N_2^{d-4}+\frac{N_2^{d-2}}{N_1}\right)^{1/2}
			|E_\alpha|^{3/2}
			+N_2^{d/2}|E_\alpha|^2.
	\end{equation*}

		Choose
		\(\alpha_0:=C_dN_2^{d/4}\), with \(C_d\) sufficiently
		large. Then we have 
		\begin{equation}
			\label{eq:high-level-distribution}
			|E_\alpha|
			\lesssim_\delta
			N_2^\delta
			\left(N_2^{d-4}+\frac{N_2^{d-2}}{N_1}\right)
			\alpha^{-4},
			\qquad \alpha\geq\alpha_0.
		\end{equation}
		
		\medskip\noindent
		\textbf{Completion of the \(L_t^4L_x^4\) estimate.}
		The preceding distribution estimate only applies for
		\(\alpha\geq\alpha_0\). To estimate the integral where
		\(|U_a|\leq\alpha_0\), we use
		\cite[Theorem 1.7]{kinoshita2023decoupling}, which gives
		\begin{equation}
			\label{eq:kns-localized-input}
			\|U_a\|_{L_t^{p_0}L_x^{p_0}}
			\lesssim_\delta N_2^\delta\|a\|_{\ell^2},
			\qquad
			p_0:=\frac{2(d+1)}{d-1}.
		\end{equation}
		The support of \(a\) lies in a shell of fixed radial
		thickness and radius comparable to \(N_1\), intersected
		with a ball of radius \(O(N_2)\). 
		
		Using \eqref{eq:kns-localized-input} with exponent
		\(\delta/p_0\) and recalling \(\|a\|_{\ell^2}=1\), we get
		\begin{equation*}
		\int_{\mathbb T}
				\int_{\{x:|U_a(x,t)|\leq\alpha_0\}}
				|U_a(x,t)|^4\,dx\,dt
				\leq
				\alpha_0^{4-p_0}
				\|U_a\|_{L_t^{p_0}L_x^{p_0}}^{p_0}
          \lesssim_\delta
				N_2^{\frac{d(d-3)}{2(d-1)}+\delta}
				\leq N_2^{d-4+\delta}.
	\end{equation*}
		The final inequality follows from
		\begin{equation*}
			d-4-\frac{d(d-3)}{2(d-1)}
			=
			\frac{d^2-7d+8}{2(d-1)}
			\geq0,
			\qquad \text{when} \; d\geq6.
		\end{equation*}
		
		Also note that 
		\begin{equation*}
			\|U_a\|_{L_t^\infty L_x^\infty}
			\leq\sum_k|a(k)|
			\leq
			\bigl(\#\operatorname{supp}a\bigr)^{1/2}
			\|a\|_{\ell^2}
			\lesssim N_2^{d/2}.
		\end{equation*}
		Then from \eqref{eq:high-level-distribution}, the remaining integral can be controlled as follows
		\begin{equation*}
				\int_{\mathbb T}
				\int_{\{x:|U_a(x,t)|>\alpha_0\}}
				|U_a(x,t)|^4\,dx\,dt
				\lesssim 
				\alpha_0^4|E_{\alpha_0}|
				+\int_{\alpha_0}^{\|U_a\|_{L_t^\infty L_x^\infty}}
				\alpha^3|E_\alpha|\,d\alpha
		\end{equation*}
		\begin{equation*}
			\lesssim_\delta
			N_2^\delta
			\left(N_2^{d-4}+\frac{N_2^{d-2}}{N_1}\right)
			(1+\log N_2).
		\end{equation*}
		Combining the two integrals, choosing \(\delta>0\)
		sufficiently small relative to \(\varepsilon\), we finally conclude that
		\begin{equation*}
			\|U_a\|_{L_t^4L_x^4}^4
			\lesssim_\varepsilon
			N_2^\varepsilon
			\left(N_2^{d-4}+\frac{N_2^{d-2}}{N_1}\right)
			\|a\|_{\ell^2(\mathbb Z^d)}^4.
		\end{equation*}
		This proves \eqref{eq:framework-high-fourth}.
		
		\medskip\noindent\textbf{Step 5: The estimate on the low-frequency strips.}
		
		Projection perpendicular to a strip produces real frequencies, even
		though the original frequencies are integral. We therefore first
		state the needed estimate for separated real frequencies and prove
		the passage back to the original periodic sum explicitly.
		
		For a suitable function $H$ on $\mathbb R^D\times \R$, write its spatial mean as
		\begin{equation*}
			(\mathcal M_x H)(t)
			:=\lim_{L\to\infty}L^{-D}\int_{\mathbb R^D}
			e^{-\pi|x|^2/L^2}H(x,t)\,dx.
		\end{equation*}
		We also have 
		\begin{equation*}
			L^{-D}\int_{\mathbb R^D}
			e^{-\pi|x|^2/L^2}e^{2\pi ix\cdot\zeta}\,dx
			=e^{-\pi L^2|\zeta|^2}
			\longrightarrow
			\mathbf1_{\{\zeta=0\}}.
		\end{equation*}
		Thus \(\mathcal M_x\) retains precisely the terms with
		zero spatial frequency.
		
		We first establish the following lemma.
		\begin{lemma}\label{lem:real-energy-interval}
			Let $D\ge4$. Suppose that a finite set $\Xi\subset\mathbb R^D$ is
			$1$-separated, $|\xi|\le CN_2$ for $\xi\in\Xi$, and all the numbers
			$|\xi|^2$ belong to an interval of length $CN_2$. Then
			\begin{equation}\label{eq:real-energy-interval}
				\int_{\mathbb R}e^{-\pi t^2}\mathcal M_x
				\left|\sum_{\xi\in\Xi}c_\xi
				e^{2\pi i(x\cdot\xi+t|\xi|^2)}\right|^4\,dt
				\lesssim_\varepsilon N_2^{D-3+\varepsilon}
				\left(\sum_{\xi\in\Xi}|c_\xi|^2\right)^2.
			\end{equation}
		\end{lemma}
		\begin{proof}
			Let \(D\geq4\), and let \(c\) be supported on a
			\(1\)-separated subset of \(\mathbb R^D\).
			We first establish the following two estimates, uniformly
			over intervals \(I\) of unit length:
			\begin{equation}
				\label{eq:real-shell-mean}
				\int_I
				\mathcal M_x
				\left|\sum_\xi c_\xi
				e^{2\pi i(x\cdot\xi+t|\xi|^2)}\right|^4\,dt
				\lesssim_\delta
				r^{(D-3)+\delta}\|c\|_{\ell^2}^4,
				\qquad
				\operatorname{supp}c
				\subset\{\xi:\bigl||\xi|-r\bigr|\leq C\},
			\end{equation}
			and
			\begin{equation}
				\label{eq:real-ball-mean}
				\int_I
				\mathcal M_x
				\left|\sum_\xi c_\xi
				e^{2\pi i(x\cdot\xi+t|\xi|^2)}\right|^4\,dt
				\lesssim_\delta
				r^{(D-2)+\delta}\|c\|_{\ell^2}^4,
				\qquad
				\operatorname{supp}c\subset B(0,r),
			\end{equation}
			where \(r\geq1\) and \(C\) is fixed.
			
			To prove \eqref{eq:real-shell-mean}, write
			\begin{equation*}
				x\cdot\xi+t|\xi|^2
				=(rx)\cdot(\xi/r)+(r^2t)|\xi/r|^2.
			\end{equation*}
			Under this rescaling, the frequencies lie in a shell of
			radius one and thickness \(O(r^{-1})\).
			We may therefore apply \eqref{linear} at frequency scale
			\(r^{-1}\). 
			
			Thus, for every \(y\in\mathbb R^D\),
			\begin{equation*}
				\int_I\int_{y+[0,r]^D}
				\left|\sum_\xi c_\xi
				e^{2\pi i(x\cdot\xi+t|\xi|^2)}\right|^4\,dx\,dt
				\lesssim_\delta
				r^D r^{D-3+4\delta}\|c\|_{\ell^2}^4.
			\end{equation*}
			The application to exponential sums follows by replacing
			each point mass with a smooth approximation to the identity and passing
			to the limit.
			
			Then from the definition of \(\mathcal M_x\), one can directly derive  \eqref{eq:real-shell-mean}. The same argument, using \eqref{BD} in place of
			\eqref{linear}, proves \eqref{eq:real-ball-mean}.
			
			\medskip\noindent
			\textbf{Decomposition according to the energy interval.}
		    We denote the left endpoint of the energy interval by \(a\). Its length is at most \(CN_2\), so
			\begin{equation*}
				a\leq|\xi|^2\leq a+CN_2,
				\qquad \xi\in\operatorname{supp}c.
			\end{equation*}
			We distinguish whether this interval is close to zero.
			
			If \(a\leq2CN_2\), then
			\(\operatorname{supp}c\subset B(0,(3CN_2)^{1/2})\).
			Applying \eqref{eq:real-ball-mean}, we obtain
			\begin{equation*}
				\int_I
				\left(\mathcal M_x
				\left|\sum_\xi c_\xi
				e^{2\pi i(x\cdot\xi+t|\xi|^2)}\right|^4\right)(t)\,dt
				\lesssim_\delta
				N_2^{(D-2)/2+\delta}\|c\|_{\ell^2}^4
				\leq
				N_2^{D-3+\delta}\|c\|_{\ell^2}^4,
			\end{equation*}
			where the last inequality uses \(D\geq4\).
			This bound is uniform in \(I\); multiplying by the
			prescribed Gaussian time weight $e^{-\pi t^2}$ on each
			unit interval and summing proves the required
			estimate (\ref{eq:real-energy-interval}).
			
			Suppose now that \(a>2CN_2\), and put \(r:=\sqrt a\).
			The assumed bound \(|\xi|\lesssim N_2\) implies
			\begin{equation*}
				N_2^{1/2}\lesssim r\lesssim N_2,
				\qquad
				r\leq|\xi|\leq\sqrt{r^2+CN_2}\lesssim r.
			\end{equation*}
			Partition the energy interval into disjoint
			intervals \(I_j\) of length \(r\) and define
			\begin{equation*}
				F_j(x,t)
				:=
				\sum_{|\xi|^2\in I_j}
				c_\xi e^{2\pi i(x\cdot\xi+t|\xi|^2)},
				\qquad
				1\leq j\leq J,
				\qquad
				J\lesssim\frac{N_2}{r}.
			\end{equation*}
			Note that each \(F_j\) is supported in a shell of bounded radial
			thickness and radius comparable to \(r\). Hence \eqref{eq:real-shell-mean} applies to each \(F_j\).
			
			It remains to combine these estimates while using the
			separation of the temporal frequencies \(|\xi|^2\).
			We will prove
			\begin{equation}
				\label{eq:energy-recombination}
				\int_I
				\mathcal M_x\left|\sum_{j=1}^J F_j\right|^4
				\,dt
				\lesssim_\delta
				Jr^{(D-3)+\delta}\|c\|_{\ell^2}^4.
			\end{equation}
			This is the required summation bound, since
			\begin{equation*}
				Jr^{(D-3)+\delta}
				\lesssim
				N_2r^{(D-4)+\delta}
				\lesssim
				N_2^{(D-3)+\delta},
				\qquad D\geq4.
			\end{equation*}
			
			Choose a Schwartz function $\zeta$, bounded below in absolute value
			on $[-1,1]$, whose Fourier transform is supported in
			$[-1,1]$. Then the temporal Fourier supports of
			\begin{equation*}
				\zeta(t-t_I)\sum_{j+k=m}F_j(x,t)F_k(x,t)
			\end{equation*}
			lie in intervals of length $O(r)$ with centers separated by $r$, where $t_I$ is the center of the interval $I$. Thus these supports have bounded overlap.
			
			 Plancherel and Cauchy--Schwarz give
			\begin{equation*}
				\int_{\mathbb R}|\zeta(t-t_I)|^2\mathcal M_x
				\left|\sum_jF_j\right|^4dt
				\lesssim J\sum_{j,k}\int_{\mathbb R}|\zeta(t-t_I)|^2
				\mathcal M_x|F_jF_k|^2dt.
			\end{equation*}
			Apply H\"older to each term on the right. Applying
			\eqref{eq:real-shell-mean} on each unit interval and summing the
			rapidly decaying weights give
			\begin{equation*}
				\int_{\mathbb R}|\zeta(t-t_I)|^2
				\mathcal M_x|F_j|^4dt
				\lesssim_\delta r^{(D-3)+\delta}\|c_j\|_{\ell^2}^4.
			\end{equation*}
			Thus, we have proved (\ref{eq:energy-recombination}). 
			Summing these bounds with weight $e^{-\pi t^2}$
			proves \eqref{eq:real-energy-interval}.

		\end{proof}
		
		We apply this lemma first to a strip of width at most two. Write
		\begin{equation*}
			l=\xi_l+(A+\gamma_l)v,\qquad
			\xi_l\in v^\perp,\qquad |\gamma_l|\le1.
		\end{equation*}
		Nonempty support forces $|A|\lesssim N_2$, and the shell condition gives
		\begin{equation*}
			|\xi_l|\lesssim N_2,\qquad
			|\xi_l|^2=|l|^2-A^2-2A\gamma_l-\gamma_l^2
			\in[N_2^2-A^2-CN_2,N_2^2-A^2+CN_2].
		\end{equation*}
		Expanding the $L^4$-norm of $U_b$, we see only quadruples satisfying 
		\begin{equation}\label{resonance1}
			l_1+l_2=l_3+l_4,\qquad
			|l_1|^2+|l_2|^2=|l_3|^2+|l_4|^2
		\end{equation}
		can contribute. Such a quadruple has
		$\xi_{l_1}+\xi_{l_2}=\xi_{l_3}+\xi_{l_4}$ and
		$\gamma_{l_1}+\gamma_{l_2}=\gamma_{l_3}+\gamma_{l_4}$. Moreover,
		\begin{equation}\label{resonance}
			\left||\xi_{l_1}|^2+|\xi_{l_2}|^2
			-|\xi_{l_3}|^2-|\xi_{l_4}|^2\right|
			=|\gamma_{l_1}^2+\gamma_{l_2}^2
			-\gamma_{l_3}^2-\gamma_{l_4}^2|\le4.
		\end{equation}
		Now define, with $x'\in v^\perp$ and $t\in\mathbb R$,
		\begin{equation*}
			F(x',t):=\sum_l|b(l)|e^{2\pi i(x'\cdot\xi_l+t|\xi_l|^2)}.
		\end{equation*}
		Roughly speaking, \(F\) is obtained by projecting \(U_b\)
		onto \(v^\perp\).
		
	Expanding \(|F|^4\) and using the defining property of
	\(\mathcal M_{x'}\), we obtain
	\begin{equation*}
		\bigl(\mathcal M_{x'}|F|^4\bigr)(t)
		=
		\sum_{\substack{l_1,l_2,l_3,l_4\\
				\xi_{l_1}+\xi_{l_2}=\xi_{l_3}+\xi_{l_4}}}
		\left(\prod_{j=1}^4|b(l_j)|\right)
		e^{2\pi it\Omega},
		\qquad
		\Omega:=
		|\xi_{l_1}|^2+|\xi_{l_2}|^2
		-|\xi_{l_3}|^2-|\xi_{l_4}|^2.
	\end{equation*}
	Since
	\(\int_{\mathbb R}e^{-\pi t^2}e^{2\pi it\Omega}\,dt
	=e^{-\pi\Omega^2}\), it follows that
	\begin{equation*}
		\int_{\mathbb R}e^{-\pi t^2}
		\bigl(\mathcal M_{x'}|F|^4\bigr)(t)\,dt
		=
		\sum_{\substack{l_1,l_2,l_3,l_4\\
				\xi_{l_1}+\xi_{l_2}=\xi_{l_3}+\xi_{l_4}}}
		e^{-\pi\Omega^2}\prod_{j=1}^4|b(l_j)|.
	\end{equation*}
	All summands on the right are nonnegative. We may therefore
	restrict this sum to quadruples satisfying
	\begin{equation*}
		l_1+l_2=l_3+l_4,
		\qquad
		|l_1|^2+|l_2|^2=|l_3|^2+|l_4|^2.
	\end{equation*}

	Hence from (\ref{resonance}), we have 
	\begin{equation*}
		|\Omega|
		=|-\gamma_{l_1}^2-\gamma_{l_2}^2
		+\gamma_{l_3}^2+\gamma_{l_4}^2|\le 4.
	\end{equation*}
	Consequently,
	\begin{equation*}
		\int_{\mathbb R}e^{-\pi t^2}
		\bigl(\mathcal M_{x'}|F|^4\bigr)(t)\,dt
		\gtrsim
		\sum_{\substack{l_1+l_2=l_3+l_4\\
				|l_1|^2+|l_2|^2=|l_3|^2+|l_4|^2}}
		\prod_{j=1}^4|b(l_j)|.
	\end{equation*}
	Expanding the $L^4$-norm and using (\ref{resonance1}), we see
	\begin{equation*}
		\|U_b\|_{L_t^4(\mathbb T)L_x^4(\mathbb T^d)}^4
			=
			\sum_{\substack{l_1+l_2=l_3+l_4\\
					|l_1|^2+|l_2|^2=|l_3|^2+|l_4|^2}}
			b(l_1)b(l_2)\overline{b(l_3)b(l_4)}
	     \leq
			\sum_{\substack{l_1+l_2=l_3+l_4\\
					|l_1|^2+|l_2|^2=|l_3|^2+|l_4|^2}}
			\prod_{j=1}^4|b(l_j)|.
	\end{equation*}
	Combining these inequalities proves 
	\begin{equation}\label{eq:projection-fourth-moment}
		\|U_b\|_{L_t^4(\mathbb T)L_x^4(\mathbb T^d)}^4
		\lesssim \int_{\mathbb R}e^{-\pi t^2}
		\bigl(\mathcal M_{x'}|F|^4\bigr)(t)\,dt.
	\end{equation}
		
		An important observation is that the projected set $\lbrace\xi_{l}\rbrace_l$ is still (roughly) a $1$-separated set. Indeed, if
		$|\xi_l-\xi_{l'}|\lesssim 1$, then
		\begin{equation*}
			 |l-l'|^2=|\xi_l-\xi_{l'}|^2+|\gamma_l-\gamma_{l'}|^2\lesssim 1.
		\end{equation*}
		
		Applying Lemma~\ref{lem:real-energy-interval} with
		$D=d-1\ge4$ together with
		\eqref{eq:projection-fourth-moment}, one can obtain
		\begin{equation}\label{eq:unit-strip-fourth}
			\|U_b\|_{L_t^4L_x^4}^4
			\lesssim_\varepsilon N_2^{d-4+\varepsilon}\|b\|_{\ell^2}^4
			\qquad\text{if }|v\cdot l-A|\le1.
		\end{equation}
		
		To treat a strip of width \(h\geq1\), partition the range
		of \(v\cdot l\) into disjoint intervals \(I_j\) of length
		at most one and set
		\begin{equation*}
			b_j(l):=b(l)\mathbf1_{I_j}(l\cdot v),
			\qquad
			b=\sum_{j=1}^J b_j,
			\qquad J\lesssim h.
		\end{equation*}
		One can apply the same orthogonality argument as in
		\eqref{eq:strip-interpolation} and deduce
		\begin{equation}\label{eq:strip-recombination}
			\left\|\sum_jU_{b_j}\right\|_{L_t^4L_x^4}
			\lesssim J^{1/4}
			\left(\sum_j\|U_{b_j}\|_{L_t^4L_x^4}^2\right)^{1/2}.
		\end{equation}
		Finally, we can combine \eqref{eq:unit-strip-fourth} with (\ref{eq:strip-recombination}), which yields the desired 
		\eqref{eq:framework-low-fourth}.
		
		\medskip\noindent\textbf{Step 6: Completion of the bilinear estimate.}
		
		For each fixed high-frequency piece, use the strip width
		$1+N_2^2/N_1$ in \eqref{eq:framework-low-fourth}.
		H\"older and \eqref{eq:framework-high-fourth} give
		\begin{equation*}
			\|U_{a_{Q,I}}U_{b_\nu}\|_{L_t^2L_x^2}^2
			\lesssim_\varepsilon N_2^\varepsilon
			\left(N_2^{d-4}+\frac{N_2^{d-2}}{N_1}\right)
			\|a_{Q,I}\|_{\ell^2}^2\|b_\nu\|_{\ell^2}^2.
		\end{equation*}
		Summing first in $\nu$ using Step 3, and then in $(Q,I)$ using Step 1 proves Theorem~\ref{thm:improved-bilinear-strichartz}.
		
		\vspace{10pt}
		\section{Sharpness of the shell-type bilinear Strichartz estimate}
		\label{sec:strichartz-sharpness}
		
		We give two examples on the unit square torus showing that the
		coefficient
		\begin{equation*}
			\left(N_2^{d-4}+\frac{N_2^{d-2}}{N_1}\right)^{1/2}
		\end{equation*}
		cannot be improved. The examples correspond to the regimes
		\(N_1\leq N_2^2\) and \(N_1\geq N_2^2\), respectively.
		
		\begin{prop}
			\label{prop:strichartz-sharpness}
			Let \(d\geq5\), and let \(C(N_1,N_2)\) be the best constant in
			\begin{equation}
				\label{eq:strichartz-best-constant}
				\left\|e^{it\Delta}\phi_1\,e^{it\Delta}\phi_2
				\right\|_{L_t^2([0,1];L_x^2(\mathbb T^d))}
				\leq C(N_1,N_2)
				\|\phi_1\|_{L_x^2(\mathbb T^d)}\|\phi_2\|_{L_x^2(\mathbb T^d)},
			\end{equation}
			where
			\begin{equation*}
				\operatorname{supp}\widehat{\phi_j}
				\subset
				\left\{k\in\mathbb Z^d:\bigl||k|-N_j\bigr|\le O_d(1)\right\},
				\qquad j=1,2.
			\end{equation*}
			There exists \(c_d>0\) such that, for all sufficiently large
			integers \(N_1\geq N_2\),
			\begin{equation}
				\label{eq:strichartz-best-constant-lower}
				C(N_1,N_2)
				\geq c_d
				\left(N_2^{d-4}+\frac{N_2^{d-2}}{N_1}\right)^{1/2}.
			\end{equation}
		\end{prop}
		
		\begin{proof}
			For finite sets \(A_j\subset\mathbb Z^d\), we take 
			\begin{equation}
				\label{eq:strichartz-example-data}
				\phi_j(x):=\sum_{k\in A_j}e^{2\pi i k\cdot x},
				\qquad
				e^{it\Delta}\phi_j(x)
				=\sum_{k\in A_j}e^{2\pi i k\cdot x-it|k|^2},
				\qquad
				\|\phi_j\|_{L_x^2(\mathbb T^d)}^2=\#A_j.
			\end{equation}
			In each example, we find a space-time region on which $e^{it\Delta}\phi_1$ and $e^{it\Delta}\phi_2$ have significant interference. Fix a sufficiently small constant \(c>0\), and choose
			\(c_0>0\) sufficiently small depending only on \(d\).
			
			\medskip
			\noindent
			\textbf{First example: \(N_1\leq N_2^2\).}
			Define
			\begin{equation*}
				A_1:=
				\left\{
				\left(u,\left\lfloor\sqrt{N_1^2-|u|^2}\right\rfloor\right):
				u\in\mathbb Z^{d-1},\ |u|\leq cN_2
				\right\}
			\end{equation*}
			and
			\begin{equation*}
				A_2:=
				\left\{
				\left(
				z,\left\lfloor\sqrt{N_2^2-|z|^2-\ell^2}\right\rfloor,
				\ell
				\right):
				z\in\mathbb Z^{d-2},\ |z|\leq cN_2,\ 
				\ell\in\mathbb Z,\ 0\leq\ell\leq cN_2^2/N_1
				\right\}.
			\end{equation*}
			 Before rounding, each point lies on
			the sphere of radius \(N_j\); rounding one 
			coordinate down moves it by less than one. Consequently,
			\begin{equation}
				\label{eq:strichartz-first-shell}
				N_j-1<|k|\leq N_j,
				\qquad
				0\leq N_j^2-|k|^2<2N_j,
				\qquad k\in A_j,\quad j=1,2.
			\end{equation}
			We also have 
			\begin{equation}
				\label{eq:strichartz-first-count}
				\#A_1\sim N_2^{d-1},
				\qquad
				\#A_2
				\sim N_2^{d-2}
				\left(1+\left\lfloor cN_2^2/N_1\right\rfloor\right)
				\sim \frac{N_2^d}{N_1}.
			\end{equation}
			
			Write \(x=(x',x_d)\). For \(a=(a',a_d)\in A_1\) and
			\(b=(b',b_d)\in A_2\), the definitions give
			\begin{equation*}
				0\leq N_1-a_d
				\leq 1+\frac{|a'|^2}
				{N_1+\sqrt{N_1^2-|a'|^2}}
				\lesssim \frac{N_2^2}{N_1},
				\qquad
				|b'|\leq N_2,
				\qquad
				0\leq b_d\leq c\frac{N_2^2}{N_1}.
			\end{equation*}
			Note that both sets have last-coordinate variation bounded by
			 \(N_2^2/N_1\). This allows a significant interference on the region with $|x_d|$ bounded by
			\(N_1/N_2^2\). More precisely, let
			\begin{equation}
				\label{eq:strichartz-first-region}
				\Omega_1:=
				\left\{
				(x,t)\in\mathbb T^d\times[0,1]:
				|x_r|\leq\frac{c_0}{N_2},\
				|x_d|\leq\frac{c_0N_1}{N_2^2},\
				0\leq t\leq\frac{c_0}{N_1}, \; 1\leq r\leq d-1
				\right\}.
			\end{equation}
			After removing \(e^{2\pi iN_1x_d-itN_1^2}\) from the
			first evolution and \(e^{-itN_2^2}\) from the second,
			the remaining phases satisfy, on \(\Omega_1\),
			\begin{equation*}
				\left|
				2\pi x'\cdot a'+2\pi x_d(a_d-N_1)
				-t(|a|^2-N_1^2)
				\right|
				\leq C_dc_0,
			\end{equation*}
			\begin{equation*}
				\left|
				2\pi x'\cdot b'+2\pi x_db_d
				-t(|b|^2-N_2^2)
				\right|
				\leq C_dc_0.
			\end{equation*}
			Taking \(c_0\) sufficiently small, we obtain
			\begin{equation*}
				|e^{it\Delta}\phi_j(x)|\geq\frac12\#A_j
				\quad\text{on }\Omega_1,
				\qquad
				|\Omega_1|
				\sim N_2^{-(d-1)}\frac{N_1}{N_2^2}\frac1{N_1}
				=N_2^{-(d+1)}.
			\end{equation*}
			Integrating the squared product over \(\Omega_1\), and
			then using \eqref{eq:strichartz-example-data} and
			\eqref{eq:strichartz-first-count}, yields
			\begin{equation}
				\label{eq:strichartz-first-lower}
				C(N_1,N_2)^2
				\geq
				\frac{
					\displaystyle\int_{\Omega_1}
					|e^{it\Delta}\phi_1(x)|^2
					|e^{it\Delta}\phi_2(x)|^2\,dx\,dt
				}{(\#A_1)(\#A_2)}
				\gtrsim
				(\#A_1)(\#A_2)|\Omega_1|
				\sim \frac{N_2^{d-2}}{N_1}.
			\end{equation}
			
			\vspace{20pt}
			\noindent
			\textbf{Second example: \(N_1\geq N_2^2\).}
			For \(u\in\mathbb Z^{d-2}\) with \(|u|\leq cN_2\), set
			\begin{equation*}
				m_u:=
				\left(u,\left\lfloor\sqrt{N_2^2-|u|^2}\right\rfloor\right)
				\in\mathbb Z^{d-1},
			\end{equation*}
			\begin{equation*}
					A_1:=\{(N_1,m_u):|u|\leq cN_2\},
				\qquad
				A_2:=\{(0,m_u):|u|\leq cN_2\}.
			\end{equation*}
			As in the first example,
			\(N_2-1<|m_u|\leq N_2\) and
			\(0\leq N_2^2-|m_u|^2<2N_2\). Furthermore,
			\begin{equation*}
				0\leq |(N_1,m_u)|-N_1
				=\frac{|m_u|^2}{\sqrt{N_1^2+|m_u|^2}+N_1}
				\leq\frac{N_2^2}{2N_1}\leq\frac12,
				\qquad
				\#A_1=\#A_2\sim N_2^{d-2}.
			\end{equation*}
			Thus both shell conditions hold. Writing \(x=(x_1,x')\),
			the two evolutions satisfy the exact identities
			\begin{equation*}
				e^{it\Delta}\phi_1(x)
				=e^{2\pi iN_1x_1-itN_1^2}e^{it\Delta}\phi_2(x),
				\qquad
				e^{it\Delta}\phi_2(x)
				=\sum_{\substack{u\in\mathbb Z^{d-2}\\|u|\leq cN_2}}
				e^{2\pi i m_u\cdot x'-it|m_u|^2}.
			\end{equation*}
			In particular, their moduli coincide and are independent
			of \(x_1\). Consider
			\begin{equation}
				\label{eq:strichartz-second-region}
				\Omega_2:=
				\left\{
				(x,t)\in\mathbb T^d\times[0,1]:
				|x_1|\le 1,\
				|x_r|\leq\frac{c_0}{N_2},\
				0\leq t\leq\frac{c_0}{N_2}, \ 2\leq r\leq d
				\right\}.
			\end{equation}
			After removing \(e^{-itN_2^2}\) from the above exponential sum,
			the remaining phases obey
			\begin{equation*}
				\left|2\pi m_u\cdot x'-t(|m_u|^2-N_2^2)\right|
				\leq 2\pi |m_u||x'|+t\bigl||m_u|^2-N_2^2\bigr|
				\lesssim c_0.
			\end{equation*}
			Hence, for sufficiently small \(c_0\), both evolutions
			have modulus at least \(\#A_j/2\) on \(\Omega_2\).
			Since \(|\Omega_2|\sim N_2^{-d}\), we obtain
			\begin{equation}
				\label{eq:strichartz-second-lower}
				C(N_1,N_2)^2
				\geq
				\frac{
					\displaystyle\int_{\Omega_2}
					|e^{it\Delta}\phi_1(x)|^2
					|e^{it\Delta}\phi_2(x)|^2\,dx\,dt
				}{(\#A_1)(\#A_2)}
				\gtrsim
				(\#A_1)(\#A_2)|\Omega_2|
				\sim N_2^{d-4}.
			\end{equation}

		\end{proof}
		The preceding examples establish the lower bounds corresponding
		to the higher-dimensional coefficient. In dimension four,
		however, an additional example gives a larger lower bound when
		\(N_1\geq N_2^2\). We present it below to explain why the
		higher-dimensional coefficient cannot extend to this dimension.
		
		\medskip\noindent
		\textbf{The four-dimensional example: \(N_1\geq N_2^2\).}
		Let \(N_1,N_2\) be sufficiently large integers satisfying
		\(N_1\geq N_2^2\). Fix a small constant \(c>0\), and define
		\begin{equation*}
			A_j:=
			\left\{
			(N_j,u)\in\mathbb Z\times\mathbb Z^3:
			|u_r|\leq cN_2^{1/2},\quad r=1,2,3
			\right\},
			\end{equation*}
		\begin{equation*}
			\phi_j(x):=\sum_{k\in A_j}e^{2\pi i k\cdot x},
			\qquad j=1,2.
		\end{equation*}
		 For any \(k=(N_j,u)\in A_j\),
		\begin{equation*}
			0\leq |k|-N_j
			=\frac{|u|^2}{\sqrt{N_j^2+|u|^2}+N_j}
			\lesssim \frac{N_2}{N_j}\le 1.
		\end{equation*}
		Hence both unit-shell conditions hold. Moreover,
		\begin{equation*}
			\#A_1=\#A_2
			\sim N_2^{3/2},
			\qquad
			\|\phi_j\|_{L_x^2(\mathbb T^4)}^2=\#A_j.
		\end{equation*}
		
		For \(x=(x_1,x')\in\mathbb T\times\mathbb T^3\), we can write
		\begin{equation*}
			e^{it\Delta}\phi_j(x)
			=
			e^{2\pi iN_jx_1-itN_j^2}
			\sum_{\substack{u\in\mathbb Z^3\\
					|u_r|\leq cN_2^{1/2},\ r=1,2,3}}
			e^{2\pi i u\cdot x'-it|u|^2}.
		\end{equation*}
		Consider the following space-time region
		\begin{equation*}
			\Omega:=
			\left\{
			(x,t)\in\mathbb T^4\times[0,1]:
			|x_1|\le 1,\;
			|x'_r|\leq c_0N_2^{-1/2},\;
			0\leq t\leq c_0N_2^{-1}, \; r=1,2,3
			\right\},
		\end{equation*}
		where \(c_0>0\) is a sufficiently small fixed constant.
		Similarly, for every \((x,t)\in\Omega\),
		\begin{equation*}
			\left|2\pi u\cdot x'-t|u|^2\right|
			\lesssim c_0.
		\end{equation*}
		This gives
		\begin{equation*}
			|e^{it\Delta}\phi_j(x)|\geq\frac12\#A_j
			\quad\text{on }\Omega,
			\qquad
			|\Omega|\sim N_2^{-3/2}N_2^{-1}=N_2^{-5/2}.
		\end{equation*}
		Consequently, 
		\begin{equation*}
		\frac{
					\|e^{it\Delta}\phi_1\,e^{it\Delta}\phi_2\|_
					{L_t^2([0,1];L_x^2(\mathbb T^4))}^2
				}{
					\|\phi_1\|_{L_x^2(\mathbb T^4)}^2
					\|\phi_2\|_{L_x^2(\mathbb T^4)}^2
				}
				\geq
				\frac{
					\displaystyle\int_\Omega
					|e^{it\Delta}\phi_1(x)|^2
					|e^{it\Delta}\phi_2(x)|^2\,dx\,dt
				}{(\#A_1)(\#A_2)}
	\end{equation*}
	\begin{equation}\label{eq:strichartz-d4-lower}
	   \gtrsim
				(\#A_1)(\#A_2)|\Omega|
				\sim N_2^3N_2^{-5/2}
				=N_2^{1/2}.
	\end{equation}
		Thus the best constant in dimension four satisfies
		\begin{equation*}
			C(N_1,N_2)\gtrsim N_2^{1/4},
			\qquad N_1\geq N_2^2.
		\end{equation*}
		On the other hand, Corollary~\ref{bilinear Strichartz}
		gives
		\begin{equation*}
			C(N_1,N_2)
			\lesssim_\varepsilon
			N_2^{1/4+\varepsilon}
			\left(1+\frac{N_2^2}{N_1}\right)^{1/4}
			\sim N_2^{1/4+\varepsilon}
		\end{equation*}
		in this range. Hence \(N_2^{1/4}\) is the optimal coefficient
		up to \(N_2^\varepsilon\). 
		
		\vspace{10pt}
		\section{Applications}
		\label{sec:applications}
		
		We first give the applications of Corollary~\ref{bilinear Strichartz}
		on a fixed rectangular torus $\mathbb T_{\boldsymbol\ell}^d$, with no
		rationality assumption. We then state the stronger consequences of
		Theorem~\ref{thm:improved-bilinear-strichartz} on the square torus for
		$d\ge6$. The latter do not replace the former on irrational tori.
		Throughout this section spatial frequencies belong to
		$\Lambda_{\boldsymbol\ell}$, all spatial measures are normalized,
		and constants may depend on the fixed side lengths.

		\subsection{Bilinear estimates for eigenfunctions}
		
		The study of $L^p$ bounds for Laplace eigenfunctions on the torus goes
		back to Bourgain \cite{Bourgain1993Torus}. Bilinear eigenfunction
		estimates, which are particularly useful in nonlinear dispersive
		equations, were systematically developed by Burq, G\'erard and
		Tzvetkov \cite{BurqGerardTzvetkov2005}. Very recently, Pezzi proved a sharp eigenfunction bound for large $p$ by refining the circle method \cite{pezzi2026sharp}.
		
		Since toral eigenfunctions are Fourier supported on lattice spheres, the bilinear shell-type estimate immediately gives the following result.
		
		\begin{coro}[Bilinear toral eigenfunction estimate]
			Let $d\ge 4$ and let $\phi_1,\phi_2$ satisfy
			\begin{equation*}
				-\Delta\phi_j=N_j^2\phi_j,
				\qquad
				N_1\ge N_2\ge 1.
			\end{equation*}
			Then, for every $\epsilon>0$,
			\begin{equation*}
				\|\phi_1\phi_2\|_{L^2(\mathbb T_{\boldsymbol\ell}^d)}
				\lesssim_\epsilon
				N_2^\epsilon
				N_2^{\frac{2d-7}{4}}\left(1+\frac{N_2^2}{N_1}\right)^{1/4}
				\|\phi_1\|_{L^2(\mathbb T_{\boldsymbol\ell}^d)}
				\|\phi_2\|_{L^2(\mathbb T_{\boldsymbol\ell}^d)}.
			\end{equation*}
		\end{coro}
		
		\begin{proof}
			Since $\widehat{\phi_j}$ is supported on the lattice sphere
			$\{k\in \Lambda_{\boldsymbol\ell}: |k|=N_j\}$ and 
			\begin{equation*}
				e^{it\Delta}\phi_j=e^{-itN_j^2}\phi_j,
			\end{equation*}
			the bilinear Strichartz estimate (\ref{bilinear Strichartz true}) gives
			\begin{equation*}
				\|\phi_1\phi_2\|_{L^2(\mathbb T_{\boldsymbol\ell}^d)}=\left\|
				e^{it\Delta}\phi_1\,
				e^{it\Delta}\phi_2
				\right\|_{L^2_{t,x}([0,1]\times \mathbb T_{\boldsymbol\ell}^d)}
				\lessapprox 	N_2^{\frac{2d-7}{4}}\left(1+\frac{N_2^2}{N_1}\right)^{1/4}
				\|\phi_1\|_{L^2}
				\|\phi_2\|_{L^2}.
			\end{equation*}

		\end{proof}

		Similarly, for the square torus $\T^d$, we have the following improvement.
		\begin{coro}[Square-torus improvement]\label{coro:square-eigenfunctions}
			If $d\ge6$ and $-\Delta\phi_j=N_j^2\phi_j$ on the unit square torus,
			then
			\begin{equation*}
				\|\phi_1\phi_2\|_{L_x^2(\mathbb T^d)}
				\lesssim_\varepsilon N_2^\varepsilon
				\left(N_2^{d-4}+\frac{N_2^{d-2}}{N_1}\right)^{1/2}
				\|\phi_1\|_{L_x^2(\mathbb T^d)}\|\phi_2\|_{L_x^2(\mathbb T^d)}.
			\end{equation*}
		\end{coro}

		\subsection{Mixed additive energy on lattice spheres}
		
		Additive properties of lattice points on spheres are closely connected
		with discrete Fourier restriction. Bourgain and Demeter obtained
		important bounds for additive energies on lattice spheres in dimensions
		four and five \cite{BourgainDemeter2015Sphere}, and Mudgal subsequently
		obtained further improvements in dimensions three and four
		\cite{Mudgal2022}. The two-scale nature of our estimate gives a mixed
		additive-energy bound for lattice points lying on two different spheres.
		
		\begin{coro}[Mixed additive energy]
			Let $d\ge 4$, $N_1\ge N_2\ge 1$, and
			\begin{equation*}
				A\subset\{k\in\Lambda_{\boldsymbol\ell}:|k|=N_1\},
				\qquad
				B\subset\{k\in\Lambda_{\boldsymbol\ell}:|k|=N_2\}.
			\end{equation*}
			Define the mixed additive energy by
			\begin{equation*}
				E(A,B)
				:=
				\#\big\{
				(a_1,a_2,b_1,b_2)\in A\times A\times B\times B:
				a_1+b_1=a_2+b_2
				\big\}.
			\end{equation*}
			Then we have 
			\begin{equation}\label{x}
				E(A,B)
				\lesssim_{\epsilon}
				N_2^\epsilon
				N_2^{\frac{2d-7}{2}}\left(1+\frac{N_2^2}{N_1}\right)^{1/2}
				|A||B|.
			\end{equation}
		\end{coro}
		
		\begin{proof}
			Set
			\begin{equation*}
				\phi_A(x):=\sum_{a\in A}e^{2\pi ia\cdot x},
				\qquad
				\phi_B(x):=\sum_{b\in B}e^{2\pi ib\cdot x}.
			\end{equation*}
			By Plancherel,
			\begin{equation*}
				\|\phi_A\|_{L^2}^2=|A|,
				\qquad
				\|\phi_B\|_{L^2}^2=|B|,
			\end{equation*}
			and
			\begin{equation*}
				\|\phi_A\phi_B\|_{L^2}^2=E(A,B).
			\end{equation*}
			The conclusion then follows by applying the previous bilinear toral eigenfunction estimate.

		\end{proof}
		
		As an immediate consequence, Cauchy--Schwarz gives
		\begin{equation}\label{xx}
			|A+B|
			\gtrapprox
			N_2^{-\frac{2d-7}{2}}
			\left(1+\frac{N_2^2}{N_1}\right)^{-1/2}
			|A||B|.
		\end{equation}
		Indeed, if we write
		\begin{equation*}
			r_{A+B}(x)
			:=
			\#\{(a,b)\in A\times B:a+b=x\},
		\end{equation*}
		then
		\begin{equation*}
			|A|^2|B|^2
			=
			\left(
			\sum_x r_{A+B}(x)
			\right)^2
			\le
			|A+B|
			\sum_x r_{A+B}(x)^2
			=
			|A+B|E(A,B).
		\end{equation*}

		\vspace{10pt}
		
		On the square torus, the same Plancherel argument and
		Corollary~\ref{coro:square-eigenfunctions} give the stronger conclusions
		\begin{equation}\label{eq:square-mixed-energy}
			E(A,B)\lesssim_\varepsilon N_2^\varepsilon
			\left(N_2^{d-4}+\frac{N_2^{d-2}}{N_1}\right)|A||B|,
			\qquad d\ge6,
		\end{equation}
		and
		\begin{equation*}
			|A+B|\gtrsim_\varepsilon N_2^{-\varepsilon}
			\left(N_2^{d-4}+\frac{N_2^{d-2}}{N_1}\right)^{-1}|A||B|,
		\end{equation*}
		where now $A,B\subset\mathbb Z^d$ lie on the spheres of radii
		$N_1,N_2$. 
		
		\subsection{Separated nonlinear smoothing for the periodic Zakharov system}
		
		Nonlinear smoothing refers to the phenomenon that the nonlinear
		Duhamel part of a dispersive evolution enjoys more spatial regularity
		than the corresponding solution. For the one-dimensional periodic
		Zakharov system, nonlinear smoothing and its applications to higher
		Sobolev norms were established by Erdo\u{g}an and Tzirakis
		\cite{ErdoganTzirakis2013}. The local theory for the Zakharov system
		was developed extensively in earlier work, including
		\cite{GinibreTsutsumiVelo1997}. More recently, Kinoshita, Nakamura and
		Sanwal introduced shell-type decoupling into the periodic Zakharov
		system and obtained sharp local well-posedness
		\cite{kinoshita2023decoupling}. The key point in their argument is
		that the Schr\"odinger--wave interaction naturally localizes a
		Schr\"odinger factor to a thinner radial shell.
		
		We consider the periodic Zakharov system
		\begin{equation*}
			\left\{
			\begin{array}{l}
				i\partial_tu+\Delta u=nu,\\
				\partial_t^2n-\Delta n=\Delta (|u|^2),
			\end{array}
			\right.
			\qquad
			(t,x)\in\mathbb R\times\mathbb T_{\boldsymbol\ell}^d,
		\end{equation*}
		with initial data
		\begin{equation*}
			u(0)=u_0,\qquad
			n(0)=n_0,\qquad
			\partial_tn(0)=n_1.
		\end{equation*}
		Here $u$ is complex-valued and $n$ is real-valued. We also write $|D|:=(-\Delta)^{1/2}$.

		We assume that the spatial means of $n_0$ and $n_1$ vanish, so that
		$|D|^{-1}\partial_tn$ is well defined.
		Introducing
		\begin{equation*}
			w_\pm
			:=
			n\pm i|D|^{-1}\partial_tn,
		\end{equation*}
		the above Zakharov system can be written as the first-order system
		\begin{equation}\label{za}
			\left\{
			\begin{array}{ll}
				i\partial_tu+\Delta u
				=\frac12(w_++w_-)u,\\
				i\partial_tw_\pm\mp |D|w_\pm
				=
				\pm |D||u|^2,
			\end{array}
			\right.
			\qquad
			(t,x)\in\mathbb R\times\mathbb T_{\boldsymbol\ell}^d,
		\end{equation}
		with initial data
		\begin{equation*}
			u(0)=u_0,
			\qquad
			w_\pm(0)
			=
			n_0\pm i|D|^{-1}n_1.
		\end{equation*}
		
		We first introduce the Littlewood--Paley and modulation decompositions that will be used throughout this section. Let $\eta:\mathbb{R}\to[0,1]$ be a smooth cutoff such that $\eta(r)=1$ for $|r|\leq 1$ and
		$\operatorname{supp}\eta\subset (-2,2)$. For dyadic numbers $N\in 2^{\mathbb{N}_0}$,
		we set
		\begin{equation*}
			\eta_1(r):=\eta(r),\quad 	\eta_N(r)
			:=
			\eta\left(\frac{r}{N}\right)
			-
			\eta\left(\frac{2r}{N}\right), \quad N\ge 2.
		\end{equation*}
		The corresponding spatial Littlewood--Paley projection is defined by
		\begin{equation*}
			P_Nu
			:=
			\mathcal{F}_x^{-1}
			\left[
			\eta_N(|k|)\mathcal{F}_xu(k)
			\right].
		\end{equation*}
		
		We next introduce modulation projections adapted to the Schr\"odinger and
		wave dispersion relations. Writing
		\begin{equation*}
			\widetilde{u}(\tau,k)
			:=
			\mathcal{F}_{t,x}u(\tau,k),
		\end{equation*}
		for dyadic $L\in 2^{\mathbb{N}_0}$, we define
		\begin{equation*}
			Q_L^S u
			:=
			\mathcal{F}_{t,x}^{-1}
			\left[
			\eta_L(\tau+|k|^2)\widetilde{u}(\tau,k)
			\right],
		\end{equation*}
		and
		\begin{equation*}
			Q_L^{W_\pm}u
			:=
			\mathcal{F}_{t,x}^{-1}
			\left[
			\eta_L(\tau\pm |k|)\widetilde{u}(\tau,k)
			\right].
		\end{equation*}
		We further write
		\begin{equation*}
			P_{N,L}^S
			:=
			P_NQ_L^S,\quad 
			P_{N,L}^{W_\pm}
			:=
			P_NQ_L^{W_\pm}.
		\end{equation*}
		
		The Bourgain spaces associated with the Schr\"odinger and wave flows are
		then defined by
		\begin{equation*}
			X_S^{s,b}
			:=
			\left\{
			u\in\mathcal{S}'(\mathbb{R}\times\mathbb{T_{\boldsymbol\ell}}^d):
			\|u\|_{X_S^{s,b}}<\infty
			\right\},
		\end{equation*}
		where
		\begin{equation*}
			\|u\|_{X_S^{s,b}}
			:=
			\left(
			\sum_{N,L}
			N^{2s}L^{2b}
			\|P_{N,L}^Su\|_{L_{t,x}^2}^2
			\right)^{1/2},
		\end{equation*}
		and
		\begin{equation*}
			X_{W_\pm}^{s,b}
			:=
			\left\{
			u\in\mathcal{S}'(\mathbb{R}\times\mathbb{T_{\boldsymbol\ell}}^d):
			\|u\|_{X_{W_\pm}^{s,b}}<\infty
			\right\},
		\end{equation*}
		with
		\begin{equation*}
			\|u\|_{X_{W_\pm}^{s,b}}
			:=
			\left(
			\sum_{N,L}
			N^{2s}L^{2b}
			\|P_{N,L}^{W_\pm}u\|_{L_{t,x}^2}^2
			\right)^{1/2}.
		\end{equation*}
		
		For a finite time interval $[0,T]$ and
		$X=X_S^{s,b}$ or $X=X_{W_\pm}^{s,b}$, we use the standard restriction
		space
		\begin{equation*}
			X(T)
			:=
			\left\{
			u:[0,T]\times\mathbb{T_{\boldsymbol\ell}}^d\to\mathbb{C}:
			\|u\|_{X(T)}<\infty
			\right\},
		\end{equation*}
		equipped with the norm
		\begin{equation*}
			\|u\|_{X(T)}
			:=
			\inf
			\left\{
			\|U\|_X:
			U\in X,\;
			U(t)=u(t)
			\text{ for }t\in[0,T]
			\right\}.
		\end{equation*}
		
		Finally, when $b>\frac12$, the standard Sobolev embedding in the time
		variable yields
		\begin{equation*}
			\|u\|_{L_t^\infty H_x^s}
			\lesssim
			\|u\|_{X_S^{s,b}},
		\end{equation*}
		and similarly
		\begin{equation*}
			\|u\|_{L_t^\infty H_x^s}
			\lesssim
			\|u\|_{X_{W_\pm}^{s,b}}.
		\end{equation*}
		Consequently, the corresponding time-restricted spaces are continuously
		embedded into $C([0,T];H^s(\mathbb{T_{\boldsymbol\ell}}^d))$.
		
		On a time interval $I:=[0,T]$ with $0<T\le1$, we define
		\begin{equation*}
			\mathcal N_{\mathrm{sep}}^S(t)
			:=
			\sum_{\pm}
			\sum_{N_1\ge 2}
			\sum_{N_2\ll N_1}
			\int_0^t
			e^{i(t-t')\Delta}
			\big[
			(P_{N_1}w_\pm)(P_{N_2}u)
			\big](t')\,dt'
		\end{equation*}
		and
		\begin{equation*}
			\mathcal N_{\mathrm{sep}}^{W,\pm}(t)
			:=
			\sum_{N_1\ge 2}
			\sum_{N_2\ll N_1}
			\int_0^t
			e^{\mp i(t-t')|D|}
			|D|
			\big[
			(P_{N_1}u)\overline{P_{N_2}u}
			+
			(P_{N_2}u)\overline{P_{N_1}u}
			\big](t')\,dt'.
		\end{equation*}
		\begin{rem}
			The operators $\mathcal N_{\mathrm{sep}}^S(t)$ and $\mathcal N_{\mathrm{sep}}^{W,\pm}(t)$ are precisely the nonlinear Duhamel terms in (\ref{za}) restricted to the ``high-low'' interactions. In fact, we restrict our attention to the regime ``$N_1\gg N_2$", in which our bilinear estimate is strictly stronger than the linear one. We restrict to $N_1\ge 2$, since the case $N_1=N_2=1$ is much easier to handle.
			
			However, we should mention that the above ``high-low'' interaction need not be the worst case for the Zakharov system, since it automatically avoids the resonant set 
			\begin{equation*}
				\Gamma_{res}^{\pm}:=\lbrace (k_1,k_2,k_3): k_1-k_2+k_3=0,\; |k_1|^2-|k_2|^2\pm |k_3|=0\rbrace.
			\end{equation*}
			Thus, for the full nonlinear Duhamel term, the nonlinear smoothing result might be weaker than that in Theorem~\ref{thm:separated-smoothing}.
		\end{rem}
		
		We first recall a bilinear transference principle that allows us to
		apply the bilinear Strichartz estimate to general space-time functions. The argument is quite standard in Bourgain spaces, and we refer to \cite[Lemma 2.2]{kinoshita2023decoupling} for the linear analogue.
		
		\begin{lemma}\label{lem:bilinear-transference}
			Let $d\ge4$ and $N_1\ge N_2\ge1$. By (\ref{bilinear Strichartz true}), the following
			bilinear Strichartz estimate holds uniformly for
			$r_1\sim N_1$ and $r_2\sim N_2$:
			\begin{equation}\label{eq:free-bilinear-shell}
				\left\|
				e^{it\Delta}\phi_1
				e^{it\Delta}\phi_2
				\right\|_{L^2_{t,x}([0,1]\times\mathbb T_{\boldsymbol\ell}^d)}
				\lesssim_\epsilon
				N_2^\epsilon
				N_2^{\frac{2d-7}{4}}\left(1+\frac{N_2^2}{N_1}\right)^{1/4}
				\|\phi_1\|_{L^2_x}
				\|\phi_2\|_{L^2_x},
			\end{equation}
			whenever
			\begin{equation*}
				\operatorname{supp}\widehat{\phi_j}
				\subset
				\left\{
				k\in\Lambda_{\boldsymbol\ell}:
				\big||k|-r_j\big|\lesssim1
				\right\},
				\qquad j=1,2.
			\end{equation*}
			
			Let $c_1,c_2\in\mathbb R$, and suppose that
			$u_1,u_2\in L^2(\mathbb R\times\mathbb T_{\boldsymbol\ell}^d)$ satisfy
			\begin{equation*}
				\operatorname{supp}\widetilde{u_j}
				\subset
				\left\{
				(\tau,k)\in\mathbb R\times\Lambda_{\boldsymbol\ell}:
				|\tau+|k|^2-c_j|\lesssim1,\,
				\big||k|-r_j\big|\lesssim1
				\right\},
			\end{equation*}
			for $j=1,2$. Then
			\begin{equation}\label{eq:transferred-bilinear-shell}
				\|u_1u_2\|_{L^2_{t,x}}
				\lesssim_\epsilon
				N_2^\epsilon
				N_2^{\frac{2d-7}{4}}\left(1+\frac{N_2^2}{N_1}\right)^{1/4}
				\|u_1\|_{L^2_{t,x}}
				\|u_2\|_{L^2_{t,x}}.
			\end{equation}
			The implicit constant is uniform in $c_1$ and $c_2$.
			
			On the unit square torus, for $d\ge6$, the stronger conclusion is
			\begin{equation}\label{eq:transferred-square-shell}
				\|u_1u_2\|_{L_t^2(\mathbb R;L_x^2(\mathbb T^d))}
				\lesssim_\varepsilon N_2^\varepsilon
				\left(N_2^{d-4}+\frac{N_2^{d-2}}{N_1}\right)^{1/2}
				\|u_1\|_{L_t^2L_x^2}\|u_2\|_{L_t^2L_x^2}.
			\end{equation}
			The same bounds hold with either factor conjugated, since the modulus
			of the product is unchanged.
		\end{lemma}
		
		\begin{proof}
			Choose a nonnegative function $\psi\in\mathcal S(\mathbb R)$ whose
			temporal Fourier transform has compact support and such that
			\begin{equation}\label{pa}
				1
				\lesssim
				\sum_{m\in\mathbb Z}|\psi(t-m)|^8
				,\quad
				\sum_{m\in\mathbb Z}|\psi(t-m)|^2
				\lesssim1
			\end{equation}
			uniformly in $t\in\mathbb R$. Set
			\begin{equation*}
				\psi_m(t):=\psi(t-m).
			\end{equation*}
			It follows that
			\begin{equation}\label{eq:time-partition-product}
				\|u_1u_2\|_{L^2_{t,x}}^2
				\lesssim
				\sum_{m\in\mathbb Z}
				\|\psi_m^4u_1u_2\|_{L^2_{t,x}}^2.
			\end{equation}
			
			For $j=1,2$, define
			\begin{equation*}
				u_{j,m}:=\psi_m u_j.
			\end{equation*}
			Then
			\begin{equation*}
				\psi_m^4u_1u_2
				=
				\psi_m^2u_{1,m}u_{2,m}.
			\end{equation*}
			
			We now remove the Schr\"odinger evolution from $u_{j,m}$. Define
			$\phi_{j,m,\sigma}$ by
			\begin{equation*}
				\phi_{j,m,\sigma}
				:=
				\mathcal F_t
				\left[
				e^{-it\Delta}u_{j,m}
				\right](\sigma).
			\end{equation*}
			Since
			\begin{equation*}
				\operatorname{supp}\widetilde{u_j}
				\subset
				\left\{
				|\tau+|k|^2-c_j|\lesssim1
				\right\}
			\end{equation*}
			and $\widehat\psi$ has fixed compact support, there exists a fixed
			constant $C>0$, independent of $m$, $c_j$, $N_1$, and $N_2$, such
			that
			\begin{equation*}
				\phi_{j,m,\sigma}=0,
				\qquad
				\text{unless }
				|\sigma-c_j|\le C.
			\end{equation*}
			Moreover,
			\begin{equation}\label{eq:representation-modulation-piece}
				u_{j,m}(t)
				=
				\frac1{2\pi}\int_{|\sigma-c_j|\le C}
				e^{it\sigma}e^{it\Delta}
				\phi_{j,m,\sigma}\,d\sigma.
			\end{equation}
			The spatial Fourier support of $\phi_{j,m,\sigma}$ is contained in
			the same radial shell as that of $u_j$.
			
			Substituting \eqref{eq:representation-modulation-piece}, we obtain
			\begin{equation*}
				\psi_m^2u_{1,m}u_{2,m}
				=
				\frac1{(2\pi)^2}\int_{|\sigma_1-c_1|\le C}
				\int_{|\sigma_2-c_2|\le C}
				e^{it(\sigma_1+\sigma_2)}
				\psi_m^2
				e^{it\Delta}\phi_{1,m,\sigma_1}
				e^{it\Delta}\phi_{2,m,\sigma_2}
				\,d\sigma_1d\sigma_2.
			\end{equation*}
			By Minkowski's inequality and the time-translated version of
			\eqref{eq:free-bilinear-shell},
			\begin{equation*}
                \|\psi_m^2u_{1,m}u_{2,m}\|_{L^2_{t,x}}
                \lesssim_\epsilon
                N_2^{\frac{2d-7}{4}+\epsilon}
                \left(1+\frac{N_2^2}{N_1}\right)^{1/4}
                \prod_{j=1}^2
                \int_{|\sigma_j-c_j|\le C}
                \|\phi_{j,m,\sigma_j}\|_{L_x^2}\,d\sigma_j.
            \end{equation*}
			Here the estimate is uniform in $m$, since after the change of
			variable $t\mapsto t-m$, one has
			\begin{equation*}
				e^{it\Delta}\phi
				=
				e^{i(t-m)\Delta}e^{im\Delta}\phi,
			\end{equation*}
			and $e^{im\Delta}$ is unitary on $L^2(\mathbb T_{\boldsymbol\ell}^d)$ and preserves
			the spatial Fourier support.
			
			The two $\sigma_j$-intervals have bounded length. Hence,
			Cauchy--Schwarz in $\sigma_1$ and $\sigma_2$ gives
			\begin{equation*}
				\|\psi_m^2u_{1,m}u_{2,m}\|_{L^2_{t,x}}
				\lesssim_\epsilon
				N_2^\epsilon
				N_2^{\frac{2d-7}{4}}\left(1+\frac{N_2^2}{N_1}\right)^{1/4}
				\prod_{j=1}^2\left(\int_{\R}\|\phi_{j,m,\sigma_j}\|_{L_x^2}^2\,d\sigma_j\right)^{1/2}.
			\end{equation*}
			By Plancherel's theorem and the unitarity of $e^{-it\Delta}$,
			\begin{equation*}
				\frac1{2\pi}\int_{\R}
				\|\phi_{j,m,\sigma}\|_{L^2_x}^2\,d\sigma
				=
				\|u_{j,m}\|_{L^2_{t,x}}^2.
			\end{equation*}
			Therefore,
			\begin{equation}\label{eq:localized-transfer}
				\|\psi_m^4u_1u_2\|_{L^2_{t,x}}
				\lesssim_\epsilon
				N_2^\epsilon
				N_2^{\frac{2d-7}{4}}\left(1+\frac{N_2^2}{N_1}\right)^{1/4}
				\|\psi_m u_1\|_{L^2_{t,x}}
				\|\psi_m u_2\|_{L^2_{t,x}}.
			\end{equation}
			
			Squaring \eqref{eq:localized-transfer}, summing over $m$, and using
			\eqref{eq:time-partition-product}, we obtain
			\begin{equation*}
				\|u_1u_2\|_{L^2_{t,x}}^2
				\lesssim_\epsilon
				N_2^{2\epsilon}
				N_2^{\frac{2d-7}{2}}\left(1+\frac{N_2^2}{N_1}\right)^{1/2}
				\sum_m
				\|\psi_m u_1\|_{L^2_{t,x}}^2
				\|\psi_m u_2\|_{L^2_{t,x}}^2.
			\end{equation*}
			Since all terms are nonnegative,
			\begin{equation*}
				\sum_m
				\|\psi_m u_1\|_{L^2}^2
				\|\psi_m u_2\|_{L^2}^2
				\le
				\left(
				\sum_m\|\psi_m u_1\|_{L^2}^2
				\right)
				\left(
				\sum_m\|\psi_m u_2\|_{L^2}^2
				\right).
			\end{equation*}
			Finally, by (\ref{pa}),
			\begin{equation*}
				\sum_m\|\psi_m u_j\|_{L^2_{t,x}}^2
				\lesssim
				\|u_j\|_{L^2_{t,x}}^2,
			\end{equation*}
			which proves \eqref{eq:transferred-bilinear-shell}.
			
			For \eqref{eq:transferred-square-shell}, use
			Theorem~\ref{thm:improved-bilinear-strichartz} in place of
			\eqref{eq:free-bilinear-shell} at each occurrence above. All remaining
			steps concern time localization, unitarity, and coefficient
			Cauchy--Schwarz, so the improved coefficient is preserved exactly.

		\end{proof}

		The main ingredient is the following high--low trilinear estimate.
		The shell localization needed to apply the bilinear shell-type
		estimate will be established as part of its proof. The localization argument is adapted from \cite{kinoshita2023decoupling}.
		
		\begin{lemma}\label{lem:high-low-trilinear}
			Let $d\ge4$, $N_1\ge2$, $1\le N_2\ll N_1$, and let $L_1,L_2,L_3\ge1$ be dyadic. Suppose that
			$u_1,v_2,w_3^\pm\in L^2(\mathbb R\times\mathbb T_{\boldsymbol\ell}^d)$ satisfy
			\begin{equation*}
				\begin{gathered}
					\operatorname{supp}\widetilde u_1\subset\operatorname{supp}(P_{N_1}Q_{L_1}^S),\qquad
					\operatorname{supp}\widetilde v_2\subset\operatorname{supp}(P_{N_2}Q_{L_2}^S),\\
					\operatorname{supp}\widetilde w_3^\pm\subset\operatorname{supp}(P_{N_1}Q_{L_3}^{W_\pm}),
				\end{gathered}
			\end{equation*}
			Set
			\begin{equation*}
				L_{\min}:=\min\{L_1,L_2,L_3\}.
			\end{equation*}
			Then, for every $0\le\theta<1$ and every $\epsilon>0$,
			\begin{equation}\label{eq:high-low-trilinear}
				\left|
				\iint_{\mathbb R\times\mathbb T_{\boldsymbol\ell}^d}
				u_1\overline{v_2}w_3^\pm\,dt\,dx
				\right|
				\lesssim_\epsilon
				N_2^{\frac d2-\frac{7\theta}{4}+\epsilon}
				\left(
				1+\frac{N_2^2}{N_1}
				\right)^{\theta/4}
				L_{\min}^{\frac{1-\theta}{2}}
				(L_1L_2L_3)^{\theta/2}
				\|u_1\|_{L^2}
				\|v_2\|_{L^2}
				\|w_3^\pm\|_{L^2}.
			\end{equation}
			
			On the unit square torus and for $d\ge6$, one additionally has
			\begin{equation}\label{eq:square-high-low-trilinear}
                \left|\int_{\mathbb R}\int_{\mathbb T^d}
                u_1\overline{v_2}w_3^\pm\,dx\,dt\right|
                \lesssim_\varepsilon
                N_2^{d/2-2\theta+\varepsilon}
                \left(1+\frac{N_2^2}{N_1}\right)^{\theta/2}
                L_{\min}^{(1-\theta)/2}(L_1L_2L_3)^{\theta/2} \|u_1\|_{L^2}\|v_2\|_{L^2}
                \|w_3^\pm\|_{L^2}.
            \end{equation}
		\end{lemma}
		
		\begin{proof}
			We prove two estimates and then interpolate between them.
			
			We first use only the sizes of the Fourier supports.
			By Plancherel and Cauchy--Schwarz in the convolution,
			\begin{equation*}
				\|u_1\overline{v_2}\|_{L^2_{t,x}}
				\lesssim N_2^{d/2}\min\{L_1,L_2\}^{1/2}
				\|u_1\|_{L^2_{t,x}}\|v_2\|_{L^2_{t,x}},
			\end{equation*}
			\begin{equation*}
				\|\overline{v_2}w_3^\pm\|_{L^2_{t,x}}
				\lesssim N_2^{d/2}\min\{L_2,L_3\}^{1/2}
				\|v_2\|_{L^2_{t,x}}\|w_3^\pm\|_{L^2_{t,x}}.
			\end{equation*}
			Pairing each product with the remaining factor and taking
			the smaller bound gives
			\begin{equation}\label{eq:crude-trilinear}
				\left|
				\iint
				u_1\overline{v_2}w_3^\pm\,dt\,dx
				\right|
				\lesssim
				N_2^{d/2}
				L_{\min}^{1/2}
				\|u_1\|_{L^2}
				\|v_2\|_{L^2}
				\|w_3^\pm\|_{L^2}.
			\end{equation}
			
			We next prove an estimate that exploits the shell structure.
			
			For each of the three factors, decompose the modulation region into
			translated intervals of length $O(1)$. Thus,
			\begin{equation*}
				u_1=\sum_{c_1\in\mathcal C_1}u_{1,c_1},
				\qquad
				v_2=\sum_{c_2\in\mathcal C_2}v_{2,c_2},
				\qquad
				w_3^\pm=\sum_{c_3\in\mathcal C_3}w_{3,c_3}^\pm,
			\end{equation*}
			where
			\begin{equation*}
				\#\mathcal C_j\lesssim L_j,
				\qquad j=1,2,3,
			\end{equation*}
			and
			\begin{equation*}
				\operatorname{supp}\widetilde{u_{1,c_1}}
				\subset
				\left\{
				(\tau_1,k_1):
				|\tau_1+|k_1|^2-c_1|\lesssim1,\,
				|k_1|\sim N_1
				\right\},
			\end{equation*}
			\begin{equation*}
				\operatorname{supp}\widetilde{v_{2,c_2}}
				\subset
				\left\{
				(\tau_2,k_2):
				|\tau_2+|k_2|^2-c_2|\lesssim1,\,
				|k_2|\sim N_2
				\right\},
			\end{equation*}
			and
			\begin{equation*}
				\operatorname{supp}\widetilde{w_{3,c_3}^\pm}
				\subset
				\left\{
				(\tau_3,k_3):
				|\tau_3\pm|k_3|-c_3|\lesssim1,\,
				|k_3|\sim N_1
				\right\}.
			\end{equation*}
			These decompositions may be chosen with bounded overlap, and hence
			\begin{equation}\label{eq:unit-mod-orthogonality}
				\sum_{c_1\in\mathcal C_1}
				\|u_{1,c_1}\|_{L^2}^2
				\lesssim
				\|u_1\|_{L^2}^2,\quad 
				\sum_{c_2\in\mathcal C_2}
				\|v_{2,c_2}\|_{L^2}^2
				\lesssim
				\|v_2\|_{L^2}^2,
			\end{equation}
			and
			\begin{equation*}
				\sum_{c_3\in\mathcal C_3}
				\|w_{3,c_3}^\pm\|_{L^2}^2
				\lesssim
				\|w_3^\pm\|_{L^2}^2.
			\end{equation*}
			
			Fix $(c_1,c_2,c_3)$. For brevity, write
			\begin{equation*}
				U:=u_{1,c_1},
				\qquad
				V:=v_{2,c_2},
				\qquad
				W:=w_{3,c_3}^\pm.
			\end{equation*}
			
			We first decompose the spatial Fourier support of the wave factor.
			Let $\{B_\alpha\}_{\alpha\in\mathcal A}$ be a finitely overlapping
			cover of the annulus $\{|k_3|\sim N_1\}$ by balls
			\begin{equation*}
				B_\alpha
				:=
				B(\xi_\alpha,CN_2),
				\qquad
				|\xi_\alpha|\sim N_1.
			\end{equation*}
			Using a smooth partition of unity subordinate to this cover, write
			\begin{equation*}
				W=\sum_{\alpha\in\mathcal A}W_\alpha,\quad 
				\operatorname{supp}_k\widehat W_\alpha
				\subset
				B_\alpha.
			\end{equation*}
			By Plancherel, we have
			\begin{equation}\label{eq:wave-spatial-orthogonality}
				\sum_{\alpha\in\mathcal A}
				\|W_\alpha\|_{L^2}^2
				\lesssim
				\|W\|_{L^2}^2.
			\end{equation}
			
			For $k_3\in B(\xi_\alpha,CN_2)$, we have
			\begin{equation*}
				\big||k_3|-|\xi_\alpha|\big|
				\lesssim N_2.
			\end{equation*}
			Since $W$ is restricted to a translated unit wave-modulation
			region,
			\begin{equation*}
				|\tau_3\pm|k_3|-c_3|\lesssim1,
			\end{equation*}
			there exists an interval $J_\alpha$ of length $O(N_2)$ such that
			\begin{equation}\label{eq:wave-time-localization}
				\operatorname{supp}_\tau\widehat W_\alpha
				\subset J_\alpha.
			\end{equation}
			
			We next decompose the temporal Fourier support of $V$. Let
			$\{I_m\}_{m\in \Z}$ be a finitely overlapping cover of
			$\mathbb R$ by intervals of length $O(N_2)$, and choose a smooth
			partition of unity subordinate to this cover. 
			
			Write
			\begin{equation*}
				V=\sum_{m\in \Z}V_m,\quad 
				\operatorname{supp}_\tau\widehat V_m
				\subset I_m.
			\end{equation*}
			Plancherel also gives
			\begin{equation}\label{eq:low-time-orthogonality}
				\sum_{m\in\mathbb Z}
				\|V_m\|_{L^2}^2
				\lesssim
				\|V\|_{L^2}^2.
			\end{equation}
			
			Consider a nonzero contribution
			\begin{equation*}
				\iint U\overline{V_m}W_\alpha\,dt\,dx,
			\end{equation*}
			where the contributing Fourier frequencies satisfy
			\begin{equation}\label{eq:fourier-conservation}
				k_1-k_2+k_3=0,
				\qquad
				\tau_1-\tau_2+\tau_3=0.
			\end{equation}
			Since $|k_2|\sim N_2$ and $k_3\in B(\xi_\alpha,CN_2)$, the first relation in \eqref{eq:fourier-conservation} implies
			\begin{equation}\label{eq:high-spatial-ball}
				k_1\in B(-\xi_\alpha,C'N_2)
			\end{equation}
			for a fixed sufficiently large $C'$.
			
			On the other hand,
			\begin{equation*}
				\tau_2\in I_m,
				\qquad
				\tau_3\in J_\alpha.
			\end{equation*}
			The second relation in \eqref{eq:fourier-conservation} therefore
			forces
			\begin{equation*}
				\tau_1\in K_{\alpha,m},
			\end{equation*}
			where $K_{\alpha,m}$ is an interval of length $O(N_2)$.
			
			Let $U_{\alpha,m}$ denote a smooth Fourier restriction of $U$ to
			a fixed enlargement of
			\begin{equation*}
				\left\{
				(\tau_1,k_1):
				k_1\in B(-\xi_\alpha,C'N_2),\,
				\tau_1\in K_{\alpha,m}
				\right\}.
			\end{equation*}
			Then
			\begin{equation*}
				\iint U\overline{V_m}W_\alpha\,dt\,dx
				=
				\iint U_{\alpha,m}\overline{V_m}W_\alpha\,dt\,dx.
			\end{equation*}
			Hence
			\begin{equation}\label{eq:piece-decomposition}
				\iint U\overline V W\,dt\,dx
				=
				\sum_{\alpha\in\mathcal A}
				\sum_{m\in \Z}
				\iint U_{\alpha,m}\overline{V_m}W_\alpha\,dt\,dx.
			\end{equation}
			
			Note that the enlarged balls
			\begin{equation*}
				B(-\xi_\alpha,C'N_2)
			\end{equation*}
			have uniformly bounded overlap. For each fixed $\alpha$, the
			intervals $K_{\alpha,m}$ also have uniformly bounded overlap,
			since the intervals $I_m$ have length $O(N_2)$ and their centers
			are separated by a constant multiple of $N_2$. Thus, we obtain
			\begin{equation}\label{eq:high-orthogonality}
				\sum_{\alpha\in\mathcal A}
				\sum_{m\in \Z}
				\|U_{\alpha,m}\|_{L^2}^2
				\lesssim
				\|U\|_{L^2}^2.
			\end{equation}
			
			We now obtain the radial shell localization. For $(\tau_2,k_2)
			\in \operatorname{supp}\widetilde V_m$, we have
			\begin{equation*}
				|\tau_2+|k_2|^2-c_2|\lesssim1
			\end{equation*}
			and $\tau_2$ belongs to an interval of length $O(N_2)$. Hence,
			$|k_2|^2$ belongs to an interval of length $O(N_2)$. Since
			\begin{equation*}
				|k_2|\sim N_2,
			\end{equation*}
			the mean value theorem implies that there exists
			$r_{2,m}\sim N_2$ such that
			\begin{equation*}
				\operatorname{supp}_k\widehat V_m
				\subset
				\left\{
				k_2\in\Lambda_{\boldsymbol\ell}:
				\big||k_2|-r_{2,m}\big|\lesssim1
				\right\}.
			\end{equation*}
			
			Similarly,
			\begin{equation*}
				|\tau_1+|k_1|^2-c_1|\lesssim1
			\end{equation*}
			on the support of $U_{\alpha,m}$, while $\tau_1$ belongs to an
			interval of length $O(N_2)$. Hence, $|k_1|^2$ belongs to an
			interval of length $O(N_2)$. Since
			\begin{equation*}
				|k_1|\sim N_1,
			\end{equation*}
			there exists $r_{1,\alpha,m}\sim N_1$ such that
			\begin{equation*}
				\operatorname{supp}_k\widehat U_{\alpha,m}
				\subset
				\left\{
				k_1\in\Lambda_{\boldsymbol\ell}:
				\big||k_1|-r_{1,\alpha,m}\big|
				\lesssim
				\frac{N_2}{N_1}\le 1
				\right\}.
			\end{equation*}
			
			Therefore, both $U_{\alpha,m}$ and $V_m$ satisfy the hypotheses of
			Lemma \ref{lem:bilinear-transference}. The additional localization
			of $U_{\alpha,m}$ to a ball of radius $O(N_2)$ only further restricts its
			Fourier support and is harmless. Hence,
			\begin{equation}\label{eq:bilinear-piece}
				\|U_{\alpha,m}V_m\|_{L^2}
				\lesssim_\epsilon
				N_2^\epsilon
				N_2^{\frac{2d-7}{4}}\left(1+\frac{N_2^2}{N_1}\right)^{1/4}
				\|U_{\alpha,m}\|_{L^2}
				\|V_m\|_{L^2}.
			\end{equation}
			By Cauchy--Schwarz in $(t,x)$,
			\begin{equation*}
				\left|
				\iint
				U_{\alpha,m}\overline{V_m}W_\alpha\,dt\,dx
				\right|
				\lesssim_\epsilon
				N_2^\epsilon
				N_2^{\frac{2d-7}{4}}\left(1+\frac{N_2^2}{N_1}\right)^{1/4}
				\|U_{\alpha,m}\|_{L^2}
				\|V_m\|_{L^2}
				\|W_\alpha\|_{L^2}.
			\end{equation*}
			
			Summing over $\alpha$ and $m$ in
			\eqref{eq:piece-decomposition}, we obtain
			\begin{equation*}
				\left|
				\iint U\overline V W\,dt\,dx
				\right|
				\lesssim_\epsilon
				N_2^\epsilon
				N_2^{\frac{2d-7}{4}}\left(1+\frac{N_2^2}{N_1}\right)^{1/4}
				\sum_{\alpha,m}
				\|U_{\alpha,m}\|_{L^2}
				\|V_m\|_{L^2}
				\|W_\alpha\|_{L^2}.
			\end{equation*}
			Applying Cauchy--Schwarz in the pair $(\alpha,m)$ gives
			\begin{equation*}
				\sum_{\alpha,m}
				\|U_{\alpha,m}\|_{L^2}
				\|V_m\|_{L^2}
				\|W_\alpha\|_{L^2}
				\le
				\left(
				\sum_{\alpha,m}
				\|U_{\alpha,m}\|_{L^2}^2
				\right)^{1/2}
				\left(
				\sum_{\alpha,m}
				\|V_m\|_{L^2}^2
				\|W_\alpha\|_{L^2}^2
				\right)^{1/2}.
			\end{equation*}
			The second sum factorizes as
			\begin{equation*}
				\sum_{\alpha,m}
				\|V_m\|_{L^2}^2
				\|W_\alpha\|_{L^2}^2
				=
				\left(
				\sum_m\|V_m\|_{L^2}^2
				\right)
				\left(
				\sum_\alpha\|W_\alpha\|_{L^2}^2
				\right).
			\end{equation*}
			Combining
			\eqref{eq:wave-spatial-orthogonality},
			\eqref{eq:low-time-orthogonality}, and
			\eqref{eq:high-orthogonality}, we conclude that
			\begin{equation*}
				\sum_{\alpha,m}
				\|U_{\alpha,m}\|_{L^2}
				\|V_m\|_{L^2}
				\|W_\alpha\|_{L^2}
				\lesssim
				\|U\|_{L^2}
				\|V\|_{L^2}
				\|W\|_{L^2}.
			\end{equation*}
			Thus, for every fixed $(c_1,c_2,c_3)$,
			\begin{equation*}
				\left|
				\int
				u_{1,c_1}\overline{v_{2,c_2}}w_{3,c_3}^\pm
				\,dt\,dx
				\right|
				\lesssim_\epsilon
				N_2^\epsilon
				N_2^{\frac{2d-7}{4}}\left(1+\frac{N_2^2}{N_1}\right)^{1/4}
				\|u_{1,c_1}\|_{L^2}
				\|v_{2,c_2}\|_{L^2}
				\|w_{3,c_3}^\pm\|_{L^2}.
			\end{equation*}
			
			Finally, applying Cauchy--Schwarz separately in $c_1,c_2,c_3$ yields
			\begin{equation}\label{eq:shell-trilinear}
				\left|
				\int
				u_1\overline{v_2}w_3^\pm\,dt\,dx
				\right|
				\lesssim_\epsilon
				N_2^{\frac{2d-7}{4}+\epsilon}
				\left(
				1+\frac{N_2^2}{N_1}
				\right)^{1/4}
				(L_1L_2L_3)^{1/2}
				\|u_1\|_{L^2}
				\|v_2\|_{L^2}
				\|w_3^\pm\|_{L^2}.
			\end{equation}
			
			Interpolating \eqref{eq:crude-trilinear} and
			\eqref{eq:shell-trilinear} with interpolation parameter
			$\theta\in[0,1)$ gives \eqref{eq:high-low-trilinear}.

			For the square-torus improvement, the very same frequency and
			modulation decompositions place $U_{\alpha,m}$ and $V_m$ in the
			shells required by \eqref{eq:transferred-square-shell}.
			Replace the coefficient in \eqref{eq:bilinear-piece} by
			$(N_2^{d-4}+N_2^{d-2}/N_1)^{1/2}$.
			The orthogonality and Cauchy--Schwarz summations give
			\begin{equation*}
				\left|\int_{\mathbb R}\int_{\mathbb T^d}
				u_1\overline{v_2}w_3^\pm\,dx\,dt\right|
				\lesssim_\varepsilon N_2^\varepsilon
				\left(N_2^{d-4}+\frac{N_2^{d-2}}{N_1}\right)^{1/2}
				(L_1L_2L_3)^{1/2}
				\|u_1\|_{L_t^2L_x^2}\|v_2\|_{L_t^2L_x^2}
				\|w_3^\pm\|_{L_t^2L_x^2}.
			\end{equation*}
			Taking the geometric mean of this estimate and
			\eqref{eq:crude-trilinear} with weights $\theta$ and $1-\theta$
			proves \eqref{eq:square-high-low-trilinear}, since
			\begin{equation*}
				N_2^{d(1-\theta)/2}
				\left(N_2^{d-4}+\frac{N_2^{d-2}}{N_1}\right)^{\theta/2}
				=N_2^{d/2-2\theta}
				\left(1+\frac{N_2^2}{N_1}\right)^{\theta/2}.
			\end{equation*}

		\end{proof}
		
		Before proving our desired nonlinear smoothing estimate, we briefly recall the following standard estimate for the Duhamel terms (see, e.g., \cite{ginibre1995probleme}):
		\begin{lemma}\label{Gini}
			Let $s\in\mathbb R$, $\frac12<b<1$, $0<T<1$, and $0<\delta<1-b$, and let $\varphi\in C_c^{\infty}(\R)$ satisfy $\varphi=1$ on $[-1,1]$ and $\supp\; \varphi\subseteq (-2,2)$. Then we have
			\begin{equation*}
				\left\|\varphi\left(\frac{t}{T}\right)\int_0^t e^{i(t-t')\Delta}F(t')dt'\right\|_{X_S^{s,b}}
				\lesssim T^{\delta}\|F\|_{X_S^{s,b-1+\delta}},
			\end{equation*}
			and
			\begin{equation*}
				\left\|\varphi\left(\frac{t}{T}\right)\int_0^t e^{\mp i(t-t')|D|}F(t')dt'\right\|_{X_{W_{\pm}}^{s,b}}
				\lesssim T^{\delta}\|F\|_{X_{W_{\pm}}^{s,b-1+\delta}}.
			\end{equation*}
		\end{lemma}
		
		Applying the trilinear estimate (\ref{eq:high-low-trilinear}), we obtain the following nonlinear smoothing result, which yields a ``$\rho$-gain'' in regularity.
		\begin{thm}[Separated nonlinear smoothing]
			\label{thm:separated-smoothing}
			Let $d\ge4$ and $s>\frac{3d-5}{6}$. For
			\begin{equation*}
				0<\rho<
				\min\left\{
				\frac12,
				\frac35\left(
				s-\frac{3d-5}{6}
				\right)
				\right\},
			\end{equation*}
			there exists $b>\frac12$ such that, for $I=[0,T]$ with $0<T\le1$,
			\begin{equation*}
				\|\mathcal N_{\mathrm{sep}}^S\|_{X_S^{s+\rho,b}(I)}
				\lesssim
				\|u\|_{X_S^{s,b}(I)}
				\sum_\pm
				\|w_\pm\|_{X_{W_\pm}^{s-\frac12,b}(I)},
			\end{equation*}
			and
			\begin{equation*}
				\|\mathcal N_{\mathrm{sep}}^{W,\pm}\|_
				{X_{W_\pm}^{s-\frac12+\rho,b}(I)}
				\lesssim
				\|u\|_{X_S^{s,b}(I)}^2.
			\end{equation*}
			In particular,
			\begin{equation*}
				\mathcal N_{\mathrm{sep}}^S
				\in
				C_tH_x^{s+\rho}(I),
				\qquad
				\mathcal N_{\mathrm{sep}}^{W,\pm}
				\in
				C_tH_x^{s-\frac12+\rho}(I).
			\end{equation*}
		\end{thm}
		
		\begin{proof}
			Take $b:=\frac12+\beta$ with $\beta>0$ sufficiently small. For the interpolation parameter $\theta$ in Lemma \ref{lem:high-low-trilinear}, set
			\begin{equation*}
				q:=\frac{1-\theta}{2}.
			\end{equation*}
			We also take $\delta$ sufficiently small so that $q>\beta+\delta$ and $\delta<1-b$.
			
			Recall that under our ``high--low'' restriction, the spatial frequencies satisfy
			\begin{equation*}
				|k_1|\sim|k_3|\sim N_1,
				\qquad
				|k_2|\sim N_2,
				\qquad
				N_2\ll N_1.
			\end{equation*}
			Hence,
			\begin{equation*}
				\left|
				|k_1|^2-|k_2|^2\pm|k_3|
				\right|
				\gtrsim N_1^2.
			\end{equation*}
			Using the temporal frequency relation in \eqref{eq:fourier-conservation}, it follows that
			\begin{equation*}
				L_{\max}
				:=
				\max\{L_1,L_2,L_3\}
				\gtrsim |(\tau_1+|k_1|^2)-(\tau_2+|k_2|^2)+(\tau_3\pm|k_3|)|
			\end{equation*}
			\begin{equation*}
                |(\tau_1+|k_1|^2)-(\tau_2+|k_2|^2)+(\tau_3\pm|k_3|)|
                =\left||k_1|^2-|k_2|^2\pm|k_3|\right|.
            \end{equation*}
            Hence
            \begin{equation}\label{eq:max-modulation}
                L_{\max}\gtrsim N_1^2.
            \end{equation}
			
			We now apply Lemma \ref{Gini} and sum over the modulations. We first consider the
			Schr\"odinger nonlinearity. By duality, it suffices to test
			\begin{equation*}
				\big[(P_{N_1}w^\pm)(P_{N_2}u)\big]
			\end{equation*}
			against a function $h$ with
			\begin{equation*}
				\|h\|_{X_S^{0,1-b-\delta}}=1.
			\end{equation*}
			Since the product $\big[(P_{N_1}w^\pm)(P_{N_2}u)\big]$ has frequency $\sim N_1$, we may assume that
			$h=P_{N_1}h$. Decomposing the three factors into dyadic modulations,
			the trilinear form is bounded, from Lemma \ref{lem:high-low-trilinear}, by
			\begin{equation*}
		N_2^{\frac d2-\frac{7\theta}{4}+\epsilon}
					\left(1+\frac{N_2^2}{N_1}\right)^{\theta/4}
					\sum_{L_1,L_2,L_3}
					L_{\min}^{q}(L_1L_2L_3)^{\theta/2}
					\|Q^S_{L_1}h\|_{L^2}
					\|Q^S_{L_2}P_{N_2}u\|_{L^2}
					\|Q^{W_\pm}_{L_3}P_{N_1}w^\pm\|_{L^2}.
	\end{equation*}
			After inserting the Bourgain-space weights, the modulation factor to be
			summed is
			\begin{equation*}
				L_{\min}^{q}
				L_1^{-q+\beta+\delta}
				L_2^{-q-\beta}
				L_3^{-q-\beta}.
			\end{equation*}
			
			If $L_1=L_{\max}$, then
			\begin{equation*}
				L_1^{-q+\beta+\delta}
				\lesssim
				N_1^{-2q+2\beta+2\delta}
				=
				N_1^{-1+\theta+2\beta+2\delta}.
			\end{equation*}
			The remaining dyadic sums converge since $q>\beta+\delta$. For instance,
			in the subregion $L_2\leq L_3\leq L_1$, the modulation factor becomes
			\begin{equation*}
				L_1^{-q+\beta+\delta}L_2^{-\beta}L_3^{-q-\beta},
			\end{equation*}
			which is summable by Cauchy--Schwarz.
			
			If either $L_2$ or $L_3$ is maximal, the maximal modulation instead carries
			the stronger factor
			\begin{equation*}
				L_{\max}^{-q-\beta}
				\lesssim
				N_1^{-2q-2\beta}
				=
				N_1^{-1+\theta-2\beta}.
			\end{equation*}
			Similarly, all the remaining powers are summable because
			$q>\beta+\delta$. Therefore, the worst contribution arises when
			$L_1=L_{\max}$. Taking the supremum over $h$ gives
			\begin{equation}\label{5.23}
				\left\|
				(P_{N_1}w^\pm)(P_{N_2}u)
				\right\|_{X_S^{0,b-1+\delta}}
				\lesssim_\epsilon
				N_1^{-1+\theta+2\beta+2\delta}
				N_2^{\frac d2-\frac{7\theta}{4}+\epsilon}
				\left(1+\frac{N_2^2}{N_1}\right)^{\theta/4}
				\|P_{N_1}w^\pm\|_{X_{W_\pm}^{0,b}}
				\|P_{N_2}u\|_{X_S^{0,b}}.
			\end{equation}
			
			The same duality argument, now testing against a wave function in
			$X_{W_\pm}^{0,1-b-\delta}$, gives
			\begin{equation}\label{5.24}
				\left\|
				(P_{N_1}u)\overline{P_{N_2}u}
				\right\|_{X_{W_\pm}^{0,b-1+\delta}}
				\lesssim_\epsilon
				N_1^{-1+\theta+2\beta+2\delta}
				N_2^{\frac d2-\frac{7\theta}{4}+\epsilon}
				\left(1+\frac{N_2^2}{N_1}\right)^{\theta/4}
				\|P_{N_1}u\|_{X_S^{0,b}}
				\|P_{N_2}u\|_{X_S^{0,b}}.
			\end{equation}
			
			We next transfer the spatial Sobolev weights. Multiplying
			\eqref{5.23} by the Sobolev weight $N_1^{s+\rho}$ and extracting the input
			weights $N_1^{s-\frac12}$ and $N_2^s$, we are left with the dyadic coefficient
			\begin{equation}\label{5.25}
				N_1^{-\frac12+\rho+\theta+2\beta+2\delta}
				N_2^{\frac d2-s-\frac{7\theta}{4}+\epsilon}
				\left(1+\frac{N_2^2}{N_1}\right)^{\theta/4}.
			\end{equation}
			More precisely,
			\begin{equation*}
				N_1^{s+\rho}
				N_1^{-1+\theta+2\beta+2\delta}
				N_1^{-s+\frac12}
				=
				N_1^{-\frac12+\rho+\theta+2\beta+2\delta},
			\end{equation*}
			while the low-frequency input contributes the factor $N_2^{-s}$.
			
			For the wave nonlinearity, the Duhamel term contains one derivative $|D|$,
			which contributes a factor $N_1$. The wave Duhamel term carries the weight
			$N_1^{s-\frac12+\rho}$, while the two Schr\"odinger inputs carry the weights
			$N_1^s$ and $N_2^s$. Hence,
			\begin{equation*}
				N_1^{s-\frac12+\rho}
				N_1
				N_1^{-1+\theta+2\beta+2\delta}
				N_1^{-s}
				=
				N_1^{-\frac12+\rho+\theta+2\beta+2\delta},
			\end{equation*}
			so the same coefficient \eqref{5.25} appears. Thus, both nonlinearities
			are reduced to the same dyadic summation.
			
			We first consider the case $N_1\ge N_2^2$. Then \eqref{5.25} is bounded by
			\begin{equation*}
				N_1^{-\frac12+\rho+\theta+2\beta+2\delta}
				N_2^{\frac d2-s-\frac{7\theta}{4}+\epsilon}.
			\end{equation*}
			Choose
			\begin{equation}\label{5.26}
				\max\left\{0,\frac47\left(\frac d2-s\right)\right\}
				<\theta<\frac12-\rho.
			\end{equation}
			After fixing $\theta$, we take $\beta,\delta,\epsilon>0$ sufficiently
			small. Then both powers of $N_1$ and $N_2$ are strictly negative.
			
			We next consider $N_2\ll N_1\le N_2^2$. In this case, \eqref{5.25} is bounded by
			\begin{equation*}
				N_1^{-\frac12+\rho+\frac{3\theta}{4}+2\beta+2\delta}
				N_2^{\frac d2-s-\frac{5\theta}{4}+\epsilon}.
			\end{equation*}
			We choose
			\begin{equation}\label{5.27}
				\max\left\{0,\frac45\left(\frac d2-s\right)\right\}
				<\theta<\frac23-\frac{4\rho}{3}.
			\end{equation}
			Taking $\beta,\delta,\epsilon>0$ sufficiently small again makes both
			frequency powers strictly negative.
			
			Consequently, in both cases, the dyadic coefficient is summable in
			$N_1$ and $N_2$, which yields
			\begin{equation*}
				\left\|
				\sum_{N_1\geq 2}\sum_{N_2\ll N_1}
				\big[(P_{N_1}w^\pm)(P_{N_2}u)\big]
				\right\|_{X_S^{s+\rho,b-1+\delta}}
				\lesssim
				\|w^\pm\|_{X_{W_\pm}^{s-\frac12,b}}
				\|u\|_{X_S^{s,b}},
			\end{equation*}
			and
			\begin{equation*}
				\left\|
				\sum_{N_1\geq 2}\sum_{N_2\ll N_1}
				|D|
				\big[
				(P_{N_1}u)\overline{P_{N_2}u}
				+
				(P_{N_2}u)\overline{P_{N_1}u}
				\big]
				\right\|_{X_{W_\pm}^{s-\frac12+\rho,b-1+\delta}}
				\lesssim
				\|u\|_{X_S^{s,b}}^2.
			\end{equation*}
			
			Since $|I|\leq 1$, the standard Duhamel estimates in Lemma \ref{Gini} imply
			\begin{equation*}
				\|\mathcal N^S_{\mathrm{sep}}\|_{X_S^{s+\rho,b}(I)}
				\lesssim
				\|u\|_{X_S^{s,b}(I)}
				\sum_\pm
				\|w^\pm\|_{X_{W_\pm}^{s-\frac12,b}(I)},
			\end{equation*}
			and
			\begin{equation*}
				\|\mathcal N^{W,\pm}_{\mathrm{sep}}\|_{X_{W_\pm}^{s-\frac12+\rho,b}(I)}
				\lesssim
				\|u\|_{X_S^{s,b}(I)}^2.
			\end{equation*}

		\end{proof}
		
		Similarly, we can also obtain the following improvement on $\T^d$.
		\begin{thm}[Separated nonlinear smoothing on the square torus]
			\label{thm:square-separated-smoothing}
			Let $d\ge6$, $s>d/2-1$, and
			\begin{equation*}
				0<\rho<\min\left\{\frac12,\frac{s-d/2+1}{2}\right\}.
			\end{equation*}
			On the unit square torus there exists $b>1/2$ such that, for
			$I=[0,T]$, $0<T\le1$,
			\begin{equation*}
				\|\mathcal N_{\mathrm{sep}}^S\|_{X_S^{s+\rho,b}(I)}
				\lesssim\|u\|_{X_S^{s,b}(I)}
				\sum_\pm\|w_\pm\|_{X_{W_\pm}^{s-1/2,b}(I)},
			\end{equation*}
			and
			\begin{equation*}
				\|\mathcal N_{\mathrm{sep}}^{W,\pm}\|_{X_{W_\pm}^{s-1/2+\rho,b}(I)}
				\lesssim\|u\|_{X_S^{s,b}(I)}^2.
			\end{equation*}
		\end{thm}
		\begin{proof}
		The proof is identical to that of
		Theorem~\ref{thm:separated-smoothing}, with
		\eqref{eq:high-low-trilinear} replaced by
		\eqref{eq:square-high-low-trilinear}.

		\end{proof}

		\newpage
		\appendix
		\section{General position}
		\label{sec:general-position}
		
		In this appendix, we consider $\tau_2$ in general position and complete
		the proof of Conjecture~\ref{8211}. We keep the notation
		\begin{equation*}
			\delta:=\frac{1}{N_1},
			\qquad
			v:=\frac{N_2}{N_1}=N_2\delta.
		\end{equation*}
		As in Section~\ref{sec:decoupling-proof}, the short direction of all the plates is taken to be
		the $e_d$-direction.
		
		Before turning to the general-position geometry, we first recall the
		localized version of the shell-type linear decoupling estimate \cite[Corollary 3.3]{kinoshita2023decoupling}.
		
		\begin{thm}\label{variant}
			Let $R\gg 1$, $d\ge 2$, and $R^{1/2}\ge Q\ge 1$. Assume that
			$f:\R^d\to\C$ is supported on
			\begin{equation*}
				\operatorname{supp}f
				\subseteq
				\left\{
				\xi\in\R^d:
				1-\frac{1}{R^{1/2}}
				\le |\xi|
				\le
				1+\frac{1}{R^{1/2}}
				\right\}
				\cap B_{QR^{-1/2}},
			\end{equation*}
			where $B_{QR^{-1/2}}\subseteq\R^d$ is a ball of radius
			$QR^{-1/2}$ with arbitrary center. Then, for $2\le p\le\infty$ and any $\epsilon>0$, we have 
			\begin{equation*}
				\|E_{\mathbb P^d}f\|_{L^p(B_R^{d+1})}
				\lesssim_{\epsilon} Q^{\epsilon}
				C_d'(Q)
				\left(
				\sum_{|\theta|=R^{-1/2}}
				\|E_{\mathbb P^d}f_\theta\|_{L^p(B_R^{d+1})}^2
				\right)^{1/2},
			\end{equation*}
			where
			\begin{equation*}
				C_d'(Q)
				:=
				\begin{cases}
					Q^{\frac{d-1}{2}-\frac{d+1}{p}},
					&
					\dfrac{2(d+1)}{d-1}\le p\le\infty,
					\\
					1,
					&
					2\le p<\dfrac{2(d+1)}{d-1}.
				\end{cases}
			\end{equation*}
		\end{thm}
		
		\begin{rem}
			The main point of Theorem~\ref{variant} is that, after localizing the
			frequency support to a ball of radius $QR^{-1/2}$, the decoupling
			constant is determined by the local parameter $Q$, rather than by
			the ambient scale $R^{1/2}$. In particular, the estimate is uniform
			in $R$ and in the center of the frequency ball.
		\end{rem}

		Recall that we write 
		\begin{equation*}
			D_\delta(f)
			:=
			\left(
			\sum_{|\theta|\sim\delta}
			\|Ef_\theta\|_{L^4(B_{\delta^{-2}})}^2
			\right)^{1/2}
		\end{equation*}
		and will repeatedly use the elementary inequality
		\begin{equation}\label{A-recombination}
			\left\|
			\sum_{\nu=1}^m G_\nu
			\right\|_{L^4}
			\lesssim
			m^{1/4}
			\left(
			\sum_{\nu=1}^m
			\|G_\nu\|_{L^4}^2
			\right)^{1/2},
		\end{equation}
		see also in (\ref{eq:strip-interpolation}).
		
		\subsection{Geometry of a general-position plate}
		
		We first describe precisely the parameters shown in Figure
		\ref{general position}. By symmetry, it is enough to consider the
		upper half of the shell. Let $\tau_2$ be a $(v,\delta)$-plate whose
		short direction is parallel to $e_d$. We write the $d$-th coordinate
		of its center as
		\begin{equation*}
			v-t,
		\end{equation*}
		where
		\begin{equation*}
			\delta\lesssim t\le v.
		\end{equation*}
		Thus, up to an $O(\delta)$ ambiguity coming from the thickness of the
		plate, $t$ measures the distance from the polar point $ve_d$ in the
		short direction.
		
		\begin{figure}[t]
			\centering
			\includegraphics[width=0.7\linewidth]{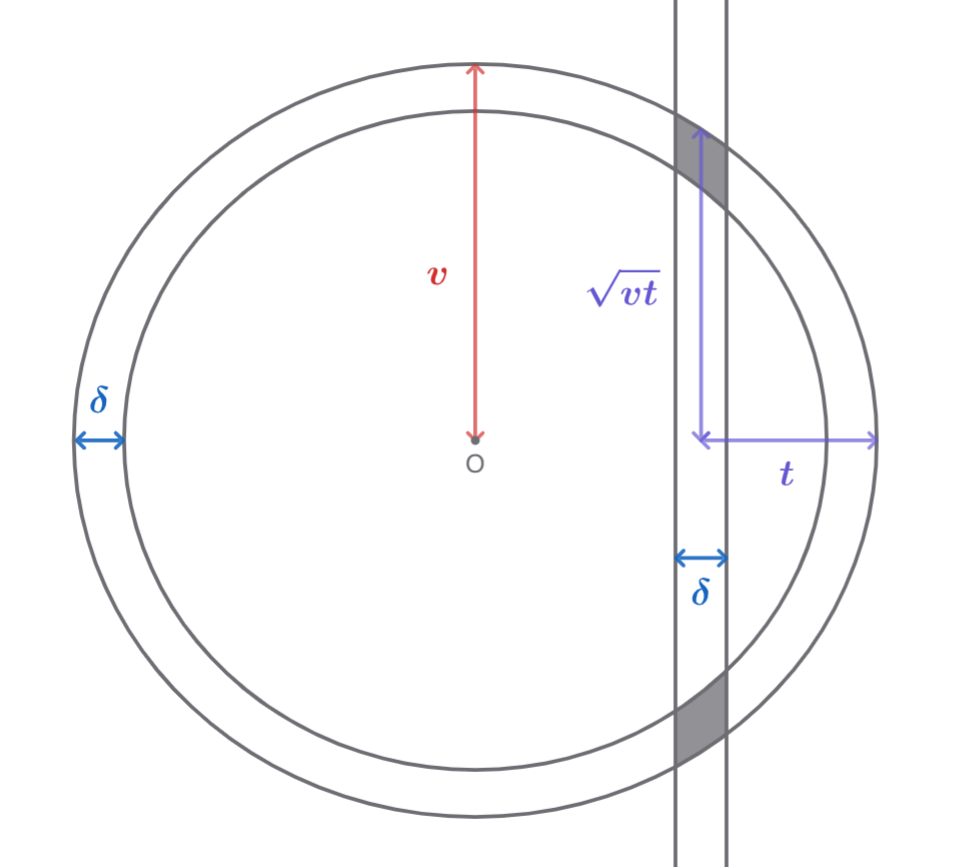}
			\caption{General position. The parameter $t$ measures the
				displacement from the polar position in the short direction, and
				the projected radius is comparable to $\sqrt{vt}$.}
			\label{general position}
		\end{figure}
		
		Suppose that
		\begin{equation*}
			\operatorname{supp}f_2
			\subseteq
			\left\{
			\xi\in\R^d:
			\big||\xi|-v\big|\lesssim\delta
			\right\}
			\cap\tau_2.
		\end{equation*}
		Writing $\xi=(\xi',\xi_d)$, we have
		\begin{equation*}
			|\xi|^2=v^2+O(v\delta),
		\end{equation*}
		and
		\begin{equation*}
			\xi_d=v-t+O(\delta).
		\end{equation*}
		Consequently,
		\begin{equation}\label{general-position-projection}
			|\xi'|^2
			=
			v^2-(v-t)^2+O(v\delta)
			=
			2vt-t^2+O(v\delta).
		\end{equation}
		Since $0<t\le v$,
		\begin{equation*}
			(2vt-t^2)^{1/2}\sim\sqrt{vt}.
		\end{equation*}
		Moreover, since $t\gtrsim\delta$, equation
		\eqref{general-position-projection} implies
		\begin{equation}\label{general-position-shell}
			\operatorname{supp}\Pi_d f_2
			\subseteq
			\left\{
			\xi'\in\R^{d-1}:
			\left|
			|\xi'|-\sqrt{vt}
			\right|
			\lesssim
			\sqrt{\frac vt}\,\delta
			\right\},
		\end{equation}
		where $\Pi_d$ denotes the projection onto the first $d-1$
		coordinates. Notice that all quantities appearing in
		\eqref{general-position-shell} agree with the parameters indicated in
		Figure~\ref{general position}.
		
		We now set
		\begin{equation}\label{general-KL}
			K:=N_1\sqrt{vt},
			\qquad
			L:=\sqrt{\frac vt}.
		\end{equation}
		Thus, the projected radius and radial thickness in
		\eqref{general-position-shell} are comparable to $K\delta$ and
		$L\delta$, respectively. Furthermore,
		\begin{equation}\label{general-KL-relations}
			KL=N_2,
			\qquad
			\frac KL=N_1t.
		\end{equation}
		Since $\delta\lesssim t\le v$, we have
		\begin{equation}\label{general-KL-range}
			1\lesssim L\le K\lesssim N_2.
		\end{equation}
		
		We next record the linear estimate associated with this geometry.
		
		\begin{lemma}\label{general slice linear}
			Let $d\ge4$, and suppose that $g$ is supported in a
			$\delta$-thin general-position plate as above. Then
			\begin{equation}\label{general-slice-L4}
				\|Eg\|_{L^4(B_{\delta^{-2}}^{d+1})}
				\lesssim_\epsilon
				N_2^\epsilon
				L^{1/4}
				K^{\frac{d-4}{4}}
				D_\delta(g).
			\end{equation}
		\end{lemma}
		
		\begin{proof}
			By Lemma~\ref{lifting}, it suffices to estimate the projection
			of the support onto $\R^{d-1}$. By
			\eqref{general-position-shell}, this projection is contained in a
			radial shell of radius comparable to $K\delta$ and radial thickness
			comparable to $L\delta$.
			
			Decompose this region into $O(L)$ radial shells of thickness
			comparable to $\delta$:
			\begin{equation*}
				g=\sum_\alpha g_\alpha.
			\end{equation*}
			For each $\alpha$, the projected support is contained in a frequency
			ball of radius $O(K\delta)$. Applying Theorem~\ref{variant} in
			dimension $d-1$, with $p=4$, gives
			\begin{equation*}
				\|Eg_\alpha\|_{L^4}
				\lesssim_\epsilon
				N_2^\epsilon
				K^{\frac{d-4}{4}}
				D_\delta(g_\alpha).
			\end{equation*}
			The corresponding energy intervals have bounded overlap. Hence,
			by \eqref{A-recombination},
			\begin{equation*}
				\|Eg\|_{L^4}
				\lesssim
				L^{1/4}
				\left(
				\sum_\alpha
				\|Eg_\alpha\|_{L^4}^2
				\right)^{1/2}.
			\end{equation*}
			Since every $\delta$-cap intersects only $O(1)$ of the radial
			pieces,
			\begin{equation*}
				\sum_\alpha D_\delta(g_\alpha)^2
				\lesssim
				D_\delta(g)^2.
			\end{equation*}
			This proves \eqref{general-slice-L4}.

		\end{proof}
		
		\vspace{7pt}
		\subsection{$\boldsymbol{N_1\ge N_2^2}$}
		
		We now prove the desired estimate for a fixed
		$\left(v,\delta\right)$-plate pair $(\tau_1,\tau_2)$, with $\tau_2$
		in general position.
		
		Since
		\begin{equation*}
			\operatorname{supp}f_{1,\tau_1}
			\subseteq
			\left\{
			\xi:
			\big||\xi|-1\big|\lesssim\delta
			\right\}
			\cap B(e_d,Cv),
		\end{equation*}
		Theorem~\ref{variant}, with local parameter
		\begin{equation*}
			Q:=\frac v\delta=N_2,
		\end{equation*}
		gives
		\begin{equation}\label{general-f1-global}
			\|Ef_{1,\tau_1}\|_{L^4}
			\lesssim_\epsilon
			N_2^{\frac{d-3}{4}+\epsilon}
			D_\delta(f_{1,\tau_1}).
		\end{equation}
		
		Suppose first that $d\ge5$. From
		\eqref{general-slice-L4}, \eqref{general-KL-relations}, and
		\eqref{general-KL-range},
		\begin{equation*}
			L^{1/4}K^{\frac{d-4}{4}}
			=
			N_2^{1/4}K^{\frac{d-5}{4}}
			\le
			N_2^{\frac{d-4}{4}}.
		\end{equation*}
		Therefore,
		\begin{equation*}
			\|Ef_{2,\tau_2}\|_{L^4}
			\lesssim_\epsilon
			N_2^{\frac{d-4}{4}+\epsilon}
			D_\delta(f_{2,\tau_2}).
		\end{equation*}
		By H\"older's inequality and \eqref{general-f1-global},
		\begin{equation}\label{general-position-high-d}
			\|Ef_{1,\tau_1}Ef_{2,\tau_2}\|_{L^2}
			\lesssim_\epsilon
			N_2^{\frac{2d-7}{4}+\epsilon}
			D_\delta(f_{1,\tau_1})
			D_\delta(f_{2,\tau_2}).
		\end{equation}
		
		We next consider $d=4$. In this case,
		\eqref{general-slice-L4} becomes
		\begin{equation}\label{general-slice-d4}
			\|Ef_{2,\tau_2}\|_{L^4}
			\lesssim_\epsilon
			N_2^\epsilon
			L^{1/4}
			D_\delta(f_{2,\tau_2}).
		\end{equation}
		To compensate for the factor $L^{1/4}$, we further localize
		$f_{1,\tau_1}$ at the scale $K\delta$.
		
		Cover the tangential projection of
		$\operatorname{supp}f_{1,\tau_1}$ by finitely overlapping balls
		$q$ of radius comparable to $K\delta$, and write
		\begin{equation*}
			f_{1,\tau_1}
			=
			\sum_q f_{1,q}.
		\end{equation*}
		Since $f_1$ is supported in the unit shell and
		$|\xi'|\lesssim v$ on its support, each $f_{1,q}$ is contained in a
		full 4-dimensional frequency ball of radius $O(K\delta)$. Indeed,
		if $\xi,\eta$ belong to the same $q$, then
		\begin{equation*}
			|\xi'-\eta'|
			\lesssim K\delta,
		\end{equation*}
		and the shell relation gives
		\begin{equation*}
			|\xi_d-\eta_d|
			\lesssim
			vK\delta+\delta
			\lesssim
			K\delta.
		\end{equation*}
		Hence Theorem~\ref{variant} in dimension $4$ gives
		\begin{equation}\label{general-f1-d4-local}
			\|Ef_{1,q}\|_{L^4}
			\lesssim_\epsilon
			K^{1/4+\epsilon}
			D_\delta(f_{1,q}).
		\end{equation}
		
		On the other hand, the support of $f_{2,\tau_2}$ has diameter
		$O(K\delta)$. Therefore the Minkowski sums
		\begin{equation*}
			\operatorname{supp}f_{1,q}
			+
			\operatorname{supp}f_{2,\tau_2}
		\end{equation*}
		have bounded overlap as $q$ varies. Standard local
		$L^2$-orthogonality yields
		\begin{equation*}
			\|Ef_{1,\tau_1}Ef_{2,\tau_2}\|_{L^2}^2
			\lesssim
			\sum_q
			\|Ef_{1,q}Ef_{2,\tau_2}\|_{L^2}^2.
		\end{equation*}
		Combining H\"older's inequality,
		\eqref{general-slice-d4}, and
		\eqref{general-f1-d4-local}, we obtain
		\begin{equation*}
			\|Ef_{1,q}Ef_{2,\tau_2}\|_{L^2}
			\lesssim_\epsilon
			N_2^\epsilon
			K^{1/4}L^{1/4}
			D_\delta(f_{1,q})
			D_\delta(f_{2,\tau_2}).
		\end{equation*}
		By \eqref{general-KL-relations},
		\begin{equation*}
			K^{1/4}L^{1/4}
			=
			N_2^{1/4}.
		\end{equation*}
		After square-summing over $q$, we conclude that
		\begin{equation}\label{general-position-d4-high}
			\|Ef_{1,\tau_1}Ef_{2,\tau_2}\|_{L^2}
			\lesssim_\epsilon
			N_2^{1/4+\epsilon}
			D_\delta(f_{1,\tau_1})
			D_\delta(f_{2,\tau_2}).
		\end{equation}
		
		Thus, for every $d\ge4$, the general-position case satisfies
		\begin{equation}\label{general-position-case1}
			\|Ef_{1,\tau_1}Ef_{2,\tau_2}\|_{L^2}
			\lesssim_\epsilon
			N_2^{\frac{2d-7}{4}+\epsilon}
			D_\delta(f_{1,\tau_1})
			D_\delta(f_{2,\tau_2}),
			\qquad
			N_1\ge N_2^2.
		\end{equation}
		
		\vspace{7pt}
		\subsection{$\boldsymbol{N_1\le N_2^2}$}
		
		We now fix a $(v,v^2)$-plate pair $(\tau_1,\tau_2)$ as in Conjecture~\ref{8211}.
		Only $\tau_2$ will be further decomposed in the short direction.
		
		Write
		\begin{equation*}
			f_{2,\tau_2}
			=
			\sum_{j=1}^{H}f_{2,j},
		\end{equation*}
		where each $f_{2,j}$ is supported in a $(v,\delta)$-plate and
		\begin{equation}\label{general-H}
			H
			\sim
			\frac{v^2}{\delta}
			=
			\frac{N_2^2}{N_1}.
		\end{equation}
		For every $j$, let $t_j$ denote the corresponding general-position
		parameter in Figure~\ref{general position}. The slices with
		$t_j\lesssim\delta$ are included in the polar range; their projection
		is contained in a ball of radius $O((v\delta)^{1/2})$, and the same
		bounds below continue to hold.
		
		Suppose first that $d\ge5$. By Lemma~\ref{general slice linear},
		together with the polar endpoint estimate (\ref{polar-low-linear}),
		\begin{equation*}
			\|Ef_{2,j}\|_{L^4}
			\lesssim_\epsilon
			N_2^{\frac{d-4}{4}+\epsilon}
			D_\delta(f_{2,j})
		\end{equation*}
		uniformly in $j$. Since the $\xi_d$-supports of the $f_{2,j}$'s have
		bounded overlap at scale $\delta$, \eqref{A-recombination} gives
		\begin{equation*}
			\|Ef_{2,\tau_2}\|_{L^4}
			\lesssim_\epsilon
			H^{1/4}
			N_2^{\frac{d-4}{4}+\epsilon}
			D_\delta(f_{2,\tau_2}).
		\end{equation*}
		Combining this with \eqref{general-f1-global} and H\"older's
		inequality yields
		\begin{equation*}
			\|Ef_{1,\tau_1}Ef_{2,\tau_2}\|_{L^2}
			\lesssim_\epsilon
			H^{1/4}
			N_2^{\frac{2d-7}{4}+\epsilon}
			D_\delta(f_{1,\tau_1})
			D_\delta(f_{2,\tau_2}).
		\end{equation*}
		By \eqref{general-H},
		\begin{equation*}
			H^{1/4}
			N_2^{\frac{2d-7}{4}}
			\sim
			N_2^{\frac{2d-5}{4}}
			N_1^{-1/4}.
		\end{equation*}
		Thus,
		\begin{equation}\label{general-position-case2-high-d}
			\|Ef_{1,\tau_1}Ef_{2,\tau_2}\|_{L^2}
			\lesssim_\epsilon
			N_2^{\frac{2d-5}{4}+\epsilon}
			N_1^{-1/4}
			D_\delta(f_{1,\tau_1})
			D_\delta(f_{2,\tau_2}),
			\qquad d\ge5.
		\end{equation}
		
		It remains to consider $d=4$. Here one has to retain the
		general-position parameter instead of taking the uniform bound above.
		
		For each $\delta$-slice $f_{2,j}$ with $t_j\gtrsim\delta$, define
		\begin{equation*}
			K_j:=N_1\sqrt{vt_j},
			\qquad
			L_j:=\sqrt{\frac{v}{t_j}},
		\end{equation*}
		so that
		\begin{equation*}
			K_jL_j=N_2.
		\end{equation*}
		The polar slices $t_j\lesssim\delta$ are placed in the block
		$K\sim N_2^{1/2}$. Group the remaining slices dyadically according to
		$K_j$:
		\begin{equation*}
			f_{2,\tau_2}
			=
			\sum_{\substack{N_2^{1/2}\lesssim K\lesssim N_2\\
					K\ {\rm dyadic}}}
			f_{2,K}.
		\end{equation*}
		Let $m_K$ denote the number of $\delta$-slices contained in the
		$K$-block. Then
		\begin{equation}\label{general-mK}
			m_K\le H.
		\end{equation}
		For a fixed $K$, put
		\begin{equation*}
			L:=\frac{N_2}{K}.
		\end{equation*}
		Lemma~\ref{general slice linear} and
		\eqref{A-recombination} give
		\begin{equation}\label{general-f2-K}
			\|Ef_{2,K}\|_{L^4}
			\lesssim_\epsilon
			N_2^\epsilon
			m_K^{1/4}
			L^{1/4}
			D_\delta(f_{2,K}).
		\end{equation}
		
		We also note that, after splitting into the two signs of $\xi_d$, the
		support of each $f_{2,K}$ has diameter $O(K\delta)$. Indeed, if $\xi,\eta$ belong to the same $K$-block, then
		\begin{equation*}
			|\xi_d-\eta_d|
			\lesssim
			\frac{
				\big||\xi'|^2-|\eta'|^2\big|+v\delta
			}{v}
			\lesssim
			\frac{K^2}{N_2}\delta+\delta
			\lesssim
			K\delta.
		\end{equation*}
		
		For this fixed $K$, decompose
		\begin{equation*}
			f_{1,\tau_1}
			=
			\sum_qf_{1,q}
		\end{equation*}
		into tangential frequency balls of radius $K\delta$. As in the
		previous subsection, each $f_{1,q}$ is contained in a full
		four-dimensional frequency ball of radius $O(K\delta)$, and hence
		\begin{equation*}
			\|Ef_{1,q}\|_{L^4}
			\lesssim_\epsilon
			K^{1/4+\epsilon}
			D_\delta(f_{1,q}).
		\end{equation*}
		Since $\operatorname{diam}(\operatorname{supp}f_{2,K})
		\lesssim K\delta$, the corresponding Minkowski sums have bounded
		overlap, and therefore
		\begin{equation*}
			\|Ef_{1,\tau_1}Ef_{2,K}\|_{L^2}^2
			\lesssim
			\sum_q
			\|Ef_{1,q}Ef_{2,K}\|_{L^2}^2.
		\end{equation*}
		By H\"older's inequality and \eqref{general-f2-K},
		\begin{equation*}
			\|Ef_{1,q}Ef_{2,K}\|_{L^2}
			\lesssim_\epsilon
			N_2^\epsilon
			K^{1/4}L^{1/4}m_K^{1/4}
			D_\delta(f_{1,q})
			D_\delta(f_{2,K}).
		\end{equation*}
		Since $KL=N_2$, we have
		\begin{equation*}
			\|Ef_{1,\tau_1}Ef_{2,K}\|_{L^2}
			\lesssim_\epsilon
			N_2^{1/4+\epsilon}
			m_K^{1/4}
			D_\delta(f_{1,\tau_1})
			D_\delta(f_{2,K}).
		\end{equation*}
		Using \eqref{general-mK} and summing over the
		$O(\log N_2)$ dyadic values of $K$, one can finally obtain
		\begin{equation}\label{general-position-case2-d4}
			\|Ef_{1,\tau_1}Ef_{2,\tau_2}\|_{L^2}
			\lesssim_\epsilon
			N_2^{3/4+\epsilon}
			N_1^{-1/4}
			D_\delta(f_{1,\tau_1})
			D_\delta(f_{2,\tau_2}).
		\end{equation}
		
		Combining \eqref{general-position-case2-high-d} and
		\eqref{general-position-case2-d4}, we conclude that, for every
		$d\ge4$,
		\begin{equation}\label{general-position-case2}
			\|Ef_{1,\tau_1}Ef_{2,\tau_2}\|_{L^2}
			\lesssim_\epsilon
			N_2^{\frac{2d-5}{4}+\epsilon}
			N_1^{-1/4}
			D_\delta(f_{1,\tau_1})
			D_\delta(f_{2,\tau_2}),
			\qquad
			N_1\le N_2^2.
		\end{equation}
		
		The estimates \eqref{general-position-case1} and
		\eqref{general-position-case2} are uniform in the position of
		$\tau_2$. Therefore, we complete the proof of Conjecture \ref{8211}.
		
		\vspace{10pt}
		\section{The bilinear Strichartz estimate on rescaled square tori}
		\label{sec:lambda-strichartz}
		
		We extend Theorem~\ref{thm:improved-bilinear-strichartz}
		to the rescaled square torus
		\begin{equation*}
			\mathbb T_\lambda^d
			:=(\mathbb R/\lambda\mathbb Z)^d,
			\qquad \lambda\geq1.
		\end{equation*}
		Throughout this appendix, spatial integration on
		\(\mathbb T_\lambda^d\) uses ordinary Lebesgue measure,
		so that \(|\mathbb T_\lambda^d|=\lambda^d\).
		We normalize the Laplacian so that
		\begin{equation*}
			\Delta e^{2\pi ix\cdot\xi}
			=-|\xi|^2e^{2\pi ix\cdot\xi},
			\qquad \xi\in\lambda^{-1}\mathbb Z^d.
		\end{equation*}

		\begin{thm}
			\label{thm:lambda-shell-strichartz}
			Let \(d\geq6\), \(\lambda\geq1\), and
			\(N_1\geq N_2\geq2\). Suppose that
			\begin{equation*}
				\operatorname{supp}\widehat{\phi_j}
				\subset
				\left\{
				\xi\in\lambda^{-1}\mathbb Z^d:
				\bigl||\xi|-N_j\bigr|\leq1
				\right\},
				\qquad j=1,2.
			\end{equation*}
			Then, for every \(\varepsilon>0\),
			\begin{equation}
				\label{eq:lambda-shell-strichartz}
				\left\|
				e^{it\Delta}\phi_1\,e^{it\Delta}\phi_2
				\right\|_{L_t^2([0,1];L_x^2(\mathbb T_\lambda^d))}
				\lesssim_\varepsilon
				N_2^\varepsilon
				\left(
				\frac{N_2^{d-4}}{\lambda}
				+\frac{N_2^{d-2}}{N_1}
				\right)^{1/2}
				\|\phi_1\|_{L_x^2(\mathbb T_\lambda^d)}\|\phi_2\|_{L_x^2(\mathbb T_\lambda^d)}.
			\end{equation}
			The implicit constant is independent of
			\(\lambda,N_1,N_2\).
		\end{thm}
		
		The spatial localization, decomposition into energy
		intervals, and final summation are the same as in the
		proof for \(\lambda=1\). Two additional estimates are
		needed. First, the high-frequency estimate must retain
		the short time interval when \(\lambda\) is large. Second, the low-frequency
		projection argument must account for the wider integer
		shell and the shorter time interval. We establish these
		two estimates before completing the bilinear argument.
		
		Write
		\begin{equation*}
			\phi(x)
			=\lambda^{-d/2}
			\sum_{k\in\mathbb Z^d}a(k)e^{2\pi ik\cdot x/\lambda},
			\qquad
			\|\phi\|_{L_x^2(\mathbb T_\lambda^d)}
			=\|a\|_{\ell^2(\mathbb Z^d)}.
		\end{equation*}
		Using the same evolution as in the main text,
		\begin{equation*}
			U_a(y,s)
			=\sum_{k\in\mathbb Z^d}
			a(k)e^{2\pi i(y\cdot k+s|k|^2)},
		\end{equation*}
		we have
		\begin{equation}
			\label{eq:lambda-evolution-scaling}
			e^{it\Delta}\phi(x)
			=
			\lambda^{-d/2}
			U_a\left(\frac{x}{\lambda},
			-\frac{t}{2\pi\lambda^2}\right).
		\end{equation}
		Set \(T:=\lambda^{-2}\).
		The integer frequencies lie in shells of radius
		\(\lambda N_j\) and thickness \(O(\lambda)\), while
		the relevant time interval has length comparable
		to \(T\). 
		
		\begin{lemma}
			\label{lem:lambda-high-fourth}
			Let \(d\geq6\), and suppose that
			\(|k_0|\sim\lambda N_1\) and
			\begin{equation*}
				\operatorname{supp}a
				\subset
				\left\{
				k\in\mathbb Z^d:
				\begin{array}{l}
					|k-k_0|\leq C\lambda N_2,\\
					\bigl||k|-\lambda N_1\bigr|\leq C\lambda,\\
					\bigl||k|^2-E_0\bigr|\leq C\lambda^2N_2
				\end{array}
				\right\},
			\end{equation*}
			where \(C\) is fixed.
			For every interval \(I\) of length \(T=\lambda^{-2}\),
			\begin{equation}
				\label{eq:lambda-high-fourth}
				\|U_a\|_{L_s^4(I;L_y^4(\mathbb T^d))}^4
				\lesssim_\varepsilon
				N_2^\varepsilon\lambda
				\left(
				(\lambda N_2)^{d-4}
				+\frac{(\lambda N_2)^{d-2}}{\lambda N_1}
				\right)
				\|a\|_{\ell^2}^4.
			\end{equation}
		\end{lemma}
		
		\begin{proof}
			Time translation changes only the phases of the
			coefficients, so it suffices to consider a fixed interval
			of length \(T\). We normalize \(\|a\|_{\ell^2}=1\).
			The energy center \(E_0\) may be taken to be an integer,
			after enlarging the fixed support constant.
			
		\medskip\noindent
		\textbf{The range \(1\leq\lambda\leq64N_2\).}
		We follow the proof of \eqref{eq:framework-high-fourth}
		and record the estimates that change.
		Choose a nonnegative \(\chi\in C_c^\infty(\mathbb R)\)
		equal to one on \([-C,C]\), where \(C\) is the constant
		in the assumption
		\(\bigl||k|^2-E_0\bigr|\leq C\lambda^2N_2\), and define
		\begin{equation*}
			G(y,s)
			:=
			\sum_{k\in\mathbb Z^d}
			e^{-\pi|k-k_0|^2/(\lambda N_2)^2}
			e^{2\pi i(y\cdot k+s|k|^2)},
			\qquad
			D_{\lambda^2N_2}(s)
			:=
			\sum_{\omega\in\mathbb Z}
			\chi\left(\frac{\omega-E_0}{\lambda^2N_2}\right)
			e^{2\pi i\omega s}.
		\end{equation*}
		Set \(K:=D_{\lambda^2N_2}*_sG\).
		Use the time cutoffs from the proof for \(\lambda=1\),
		with their scale parameter \(N_2\) replaced by a dyadic
		number comparable to \(\lambda N_2\), and define
		\begin{equation*}
			K_0:=D_{\lambda^2N_2}*_s(\eta_0G),
			\qquad
			K_{Q,S}:=D_{\lambda^2N_2}*_s(\eta_{Q,S}G),
			\qquad
			K_\rho:=D_{\lambda^2N_2}*_s(\rho G),
		\end{equation*}
		where \(Q\leq S\lesssim\lambda N_2\).
		
		Poisson summation and the calculation in
		\eqref{eq:energy-rational-sum} give
		\begin{equation*}
			|D_{\lambda^2N_2}(s)|
			\lesssim_A
			\lambda^2N_2
			\left(
			1+\lambda^2N_2\operatorname{dist}(s,\mathbb Z)
			\right)^{-A},
			\qquad
			\|D_{\lambda^2N_2}\|_{L_s^1(\mathbb T)}
			\lesssim1,
		\end{equation*}
		and
		\begin{equation*}
			\sum_{\substack{0\leq r<q\\(r,q)=1}}
			|D_{\lambda^2N_2}(s-r/q-\varphi)|
			\lesssim\lambda^2N_2,
			\qquad q\leq\lambda N_2.
		\end{equation*}
		Substituting these bounds into the Gaussian calculations
		for \(K_0\), \(K_{Q,S}\), and \(K_\rho\), respectively,
		yields
		\begin{equation*}
			\|K_0\|_{L_s^\infty L_y^\infty}
			\lesssim
			\lambda^2N_2
			\left(
			(\lambda N_2)^{d-3}
			+\frac{(\lambda N_2)^{d-1}}{|k_0|}
			\right),
			\qquad
			\|K_0*_{y,s}\|_{L_s^2L_y^2\to L_s^2L_y^2}
			\lesssim(\lambda N_2)^{-2},
		\end{equation*}
		\begin{equation*}
			\|K_{Q,S}\|_{L_s^\infty L_y^\infty}
			\lesssim
			\lambda^2N_2(\lambda N_2S)^{d/2-1}
			Q^{1-d/2}
			\left(\frac{\lambda N_2}{|k_0|}+\frac QS\right),
			\qquad
			\|K_\rho\|_{L_s^\infty L_y^\infty}
			\lesssim(\lambda N_2)^{d/2}.
		\end{equation*}
		
		For the \(L_s^2L_y^2\) estimate, set, as before,
		\begin{equation*}
			b_{Q,S}
			:=\frac{\widehat\eta_{Q,S}(0)}{\widehat\rho(0)},
			\qquad
			\widetilde K_{Q,S}:=K_{Q,S}-b_{Q,S}K_\rho.
		\end{equation*}
		Its Fourier multiplier is
		\begin{equation*}
			\widehat{\widetilde K}_{Q,S}(k,\omega)
			=
			e^{-\pi|k-k_0|^2/(\lambda N_2)^2}
			\chi\left(\frac{\omega-E_0}{\lambda^2N_2}\right)
			\left[
			\widehat\eta_{Q,S}(\omega-|k|^2)
			-b_{Q,S}\widehat\rho(\omega-|k|^2)
			\right].
		\end{equation*}
		The expression in brackets vanishes when
		\(\omega=|k|^2\). For the other integer values of
		\(\omega-|k|^2\), the Fourier coefficient estimates
		used in \eqref{eq:arc-fourier} give
		\begin{equation*}
			\|\widetilde K_{Q,S}*_{y,s}\|_
			{L_s^2L_y^2\to L_s^2L_y^2}
			\lesssim_\delta
			(\lambda N_2)^\delta\frac{Q}{\lambda N_2S}.
		\end{equation*}
		Moreover,
		\begin{equation*}
			b_{Q,S}\|K_\rho\|_{L_s^\infty L_y^\infty}
			\lesssim
			\frac{Q^2}{\lambda N_2S}(\lambda N_2)^{d/2}
			\leq
			\lambda^2N_2(\lambda N_2S)^{d/2-1}
			Q^{1-d/2}\frac QS.
		\end{equation*}
		Thus the displayed pointwise bound for \(K_{Q,S}\)
		also holds for \(\widetilde K_{Q,S}\).
		
		Interpolation now gives
		\begin{equation*}
		\|\widetilde K_{Q,S}*_{y,s}\|_
				{L_s^{4/3}L_y^{4/3}\to L_s^4L_y^4}^{\,2}
\lesssim_\delta
				(\lambda N_2)^\delta\lambda
				\left[
				\frac{(\lambda N_2)^{d/2}}{|k_0|}
				\left(\frac SQ\right)^{d/2-2}
				+(\lambda N_2)^{d/2-1}
				\left(\frac SQ\right)^{d/2-3}
				\right]
	\end{equation*}
	\begin{equation*}
\lesssim_\delta
				(\lambda N_2)^\delta\lambda
				\left(
				(\lambda N_2)^{d-4}
				+\frac{(\lambda N_2)^{d-2}}{|k_0|}
				\right).
	\end{equation*}
		The last inequality uses \(d\geq6\) and
		\(1\leq S/Q\lesssim\lambda N_2\).
		The two bounds for \(K_0\) similarly imply 
		\begin{equation*}
			\|K_0*_{y,s}\|_
			{L_s^{4/3}L_y^{4/3}\to L_s^4L_y^4}^{\,2}
			\lesssim
			\lambda\left(
			(\lambda N_2)^{d-4}
			+\frac{(\lambda N_2)^{d-2}}{|k_0|}
			\right).
		\end{equation*}
		
		We next estimate \(U_a\) below the level at which
		the contribution of \(K_\rho\) can be absorbed.
		The shell condition implies that the values of
		\(|k|^2\) lie in an interval of length
		\(O(\lambda^2N_1)\). Partition this interval into
		\(J\lesssim\lambda\) consecutive intervals of length
		comparable to \(\lambda N_1\), and write
		\(a=\sum_{j=1}^J a_j\) accordingly.
		Within each interval,
		\begin{equation*}
			\bigl||k|-|k'|\bigr|
			=
			\frac{\bigl||k|^2-|k'|^2\bigr|}{|k|+|k'|}
			\lesssim1.
		\end{equation*}
		Each \(a_j\) is therefore supported in a shell of fixed
		radial thickness and in \(B(k_0,C\lambda N_2)\).
		Theorem~1.7 of \cite{kinoshita2023decoupling} gives
		\begin{equation*}
			\|U_{a_j}\|_{L_s^{p_0}(\mathbb T;L_y^{p_0}(\mathbb T^d))}
			\lesssim_\delta
			(\lambda N_2)^\delta\|a_j\|_{\ell^2},
			\qquad
			p_0:=\frac{2(d+1)}{d-1}.
		\end{equation*}
		To sum these pieces at exponent $p_0$, choose an even function
		$\vartheta\in C_c^\infty((-2,2))$ equal to one on $[-1,1]$.
		Let $e_j$ be the centers of the consecutive intervals of length
		$\lambda N_1$ used above, and define the temporal Fourier multipliers
		\begin{equation*}
			\widehat{P_jH}(y,\omega)
			:=
			\vartheta\left(\frac{\omega-e_j}{\lambda N_1}\right)
			\widehat H(y,\omega),
			\qquad \omega\in\mathbb Z.
		\end{equation*}
		Here the Fourier transform is taken only in $s$.
		Then $P_jU_{a_j}=U_{a_j}$, and bounded overlap gives
		\begin{equation*}
			\left\|\sum_jP_jH_j\right\|_{L_s^2L_y^2}
			\lesssim
			\left\|\left(\sum_j|H_j|^2\right)^{1/2}
			\right\|_{L_s^2L_y^2}, \quad \forall H_j.
		\end{equation*}
		Moreover, Fourier inversion gives
		\begin{equation*}
			P_jH(y,s)
			=
			\int_{\mathbb R}\widehat\vartheta(u)
			e^{-2\pi iue_j/(\lambda N_1)}
			H\left(y,s+\frac{u}{\lambda N_1}\right)\,du.
		\end{equation*}
		Then applying Cauchy--Schwarz in $j$, one can derive
		\begin{equation*}
			\left\|\sum_jP_jH_j\right\|_{L_s^\infty L_y^\infty}
			\lesssim J^{1/2}
			\left\|\left(\sum_j|H_j|^2\right)^{1/2}
			\right\|_{L_s^\infty L_y^\infty}.
		\end{equation*}
		Interpolating these two bounds and applying Minkowski's inequality,
		we obtain
		\begin{equation*}
			\|U_a\|_{L_s^{p_0}L_y^{p_0}}
			\lesssim J^{1/2-1/p_0}
			\left(\sum_j\|U_{a_j}\|_{L_s^{p_0}L_y^{p_0}}^2\right)^{1/2}
			\lesssim_\delta
			\lambda^{1/2-1/p_0}(\lambda N_2)^\delta,
		\end{equation*}
		where we used $J\lesssim\lambda$ and
		$\sum_j\|a_j\|_{\ell^2}^2=\|a\|_{\ell^2}^2=1$.
		
		Choose \(\alpha_0:=C_d(\lambda N_2)^{d/4}\), with
		\(C_d\) sufficiently large. Using the preceding bound
		with a smaller loss exponent,
		\begin{equation*}
				\int_{\mathbb T}
				\int_{\{y:|U_a(y,s)|\leq\alpha_0\}}
				|U_a(y,s)|^4\,dy\,ds
				\leq
				\alpha_0^{4-p_0}
				\|U_a\|_{L_s^{p_0}L_y^{p_0}}^{p_0}
		\end{equation*}
		\begin{equation*}
			\lesssim_\delta
			\lambda^{2/(d-1)}
			(\lambda N_2)^{
				\frac{d(d-3)}{2(d-1)}+\delta}
			\lesssim_\delta
			\lambda(\lambda N_2)^{d-4+\delta}.
		\end{equation*}
		Here we used
		\(2/(d-1)\leq1\) and
		\(d(d-3)/(2(d-1))\leq d-4\).
		
		For \(\alpha\geq\alpha_0\), the same argument as in
		\eqref{eq:high-level-distribution}, using the operator
		bounds above and
		\(\|K_\rho\|_{L_s^\infty L_y^\infty}
		\lesssim(\lambda N_2)^{d/2}\), gives 
		\begin{equation*}
			\left|
			\{(y,s)\in\mathbb T^d\times\mathbb T:
			|U_a(y,s)|>\alpha\}
			\right|
			\lesssim_\delta
			(\lambda N_2)^\delta\lambda
			\left(
			(\lambda N_2)^{d-4}
			+\frac{(\lambda N_2)^{d-2}}{|k_0|}
			\right)\alpha^{-4}.
		\end{equation*}
		Also,
		\(\|U_a\|_{L_s^\infty L_y^\infty}
		\lesssim(\lambda N_2)^{d/2}\).
		Integrating the distribution estimate and combining
		it with the estimate below \(\alpha_0\), we obtain
		\begin{equation*}
			\|U_a\|_{L_s^4(\mathbb T;L_y^4(\mathbb T^d))}^4
			\lesssim_\delta
			(\lambda N_2)^\delta
			(1+\log(\lambda N_2))\lambda
			\left(
			(\lambda N_2)^{d-4}
			+\frac{(\lambda N_2)^{d-2}}{|k_0|}
			\right).
		\end{equation*}
		Since \(\lambda\leq64N_2\) and
		\(|k_0|\sim\lambda N_1\), choosing \(\delta\)
		sufficiently small absorbs all losses into
		\(N_2^\varepsilon\).
		Periodicity and \(|I|=\lambda^{-2}\leq1\) then give
		\eqref{eq:lambda-high-fourth}.
			
			\vspace{7pt}
			\medskip\noindent
\textbf{The range \(\lambda>64N_2\).}
We retain the time interval of length \(T=\lambda^{-2}\).
By time translation, it suffices to consider \(I=[0,T]\).
In this part of the proof, convolution in time is on
\(\mathbb R\), and all spatial norms are on \(\mathbb T^d\).

Choose \(\zeta\in C_c^\infty(\mathbb R)\), equal to one
on \([-1,1]\) and supported in \([-2,2]\), and set
\(\zeta_T(s):=\zeta(s/T)\). Define
\begin{equation*}
	G(y,s)
	:=
	\sum_{k\in\mathbb Z^d}
	e^{-\pi|k-k_0|^2/(\lambda N_2)^2}
	e^{2\pi i(y\cdot k+s|k|^2)},
	\qquad
	D^{\mathbb R}_{\lambda^2N_2}(s)
	=
	\lambda^2N_2e^{2\pi iE_0s}
	\widehat\chi(-\lambda^2N_2s),
\end{equation*}
where \(\chi\geq0\) is smooth, compactly supported,
and equal to one on the interval given by
\(\bigl||k|^2-E_0\bigr|\leq C\lambda^2N_2\).
In particular,
\begin{equation*}
	\widehat{D^{\mathbb R}_{\lambda^2N_2}}(\tau)
	=
	\chi\left(\frac{\tau-E_0}{\lambda^2N_2}\right),
	\qquad
	|D^{\mathbb R}_{\lambda^2N_2}(s)|
	\lesssim_A
	\lambda^2N_2(1+\lambda^2N_2|s|)^{-A}.
\end{equation*}
We now define
\begin{equation*}
	K(y,s)
	:=
	\int_{\mathbb R}
	D^{\mathbb R}_{\lambda^2N_2}(s-\varphi)
	G(y,\varphi)\,d\varphi.
\end{equation*}
Its Fourier expansion is
\begin{equation*}
	K(y,s)
	=
	\sum_{k\in\mathbb Z^d}
	e^{-\pi|k-k_0|^2/(\lambda N_2)^2}
	\chi\left(\frac{|k|^2-E_0}{\lambda^2N_2}\right)
	e^{2\pi i(y\cdot k+s|k|^2)}.
\end{equation*}
Since \(\zeta_T(s-s')=1\) for \(s,s'\in I\), every
\(f\) supported in \(\mathbb T^d\times I\) satisfies
\begin{equation}
	\label{eq:lambda-time-cutoff-pairing}
	\langle(\zeta_TK)*_{y,s}f,f\rangle_{L_s^2L_y^2}
	=
	\langle K*_{y,s}f,f\rangle_{L_s^2L_y^2}.
\end{equation}
We will estimate convolution by \(\zeta_TK\) from
\(L_s^{4/3}L_y^{4/3}\) to \(L_s^4L_y^4\), and then
use this identity to bound \(U_a\) on \(I\).

\medskip\noindent
\textbf{Decomposition of the integral defining \(K\).}
Choose the integer \(J\geq0\) so that
\(1\leq2^{-J}\lambda N_2<2\), and let \(S\) range over
\(2^{-j}\lambda N_2\), \(1\leq j\leq J\).
Using the function \(\eta\) from the proof for
\(\lambda=1\), define
\begin{equation*}
	h_0(\varphi):=\eta((\lambda N_2)^2\varphi),
	\qquad
	h_S(\varphi)
	:=
	\eta(\lambda N_2S\varphi)
	-\eta(2\lambda N_2S\varphi),
\end{equation*}
and
\begin{equation*}
	h_{\mathrm{tail}}(\varphi)
	:=
	1-\eta(2^{-J}(\lambda N_2)^2\varphi).
\end{equation*}
Then we have
\begin{equation*}
	h_0+\sum_Sh_S+h_{\mathrm{tail}}=1.
\end{equation*}
Set
\begin{equation*}
		K_0(y,s)
		:=
		\zeta_T(s)\int_{\mathbb R}
		D^{\mathbb R}_{\lambda^2N_2}(s-\varphi)
		h_0(\varphi)G(y,\varphi)\,d\varphi,
	\end{equation*}
	\begin{equation*}
		K_S(y,s)
		:=
		\zeta_T(s)\int_{\mathbb R}
		D^{\mathbb R}_{\lambda^2N_2}(s-\varphi)
		h_S(\varphi)G(y,\varphi)\,d\varphi,
	\end{equation*}
	\begin{equation*}
		K_{\mathrm{tail}}(y,s)
		:=
		\zeta_T(s)\int_{\mathbb R}
		D^{\mathbb R}_{\lambda^2N_2}(s-\varphi)
		h_{\mathrm{tail}}(\varphi)G(y,\varphi)\,d\varphi.
	\end{equation*}
Thus
\begin{equation}
	\label{eq:lambda-short-time-decomposition}
	\zeta_TK=K_0+\sum_SK_S+K_{\mathrm{tail}}.
\end{equation}

\medskip\noindent
\textbf{The \(L_s^2L_y^2\) operator bounds.}
Before multiplication by \(\zeta_T\), the multiplier
corresponding to \(K_0\) or \(K_S\) is
\begin{equation*}
	e^{-\pi|k-k_0|^2/(\lambda N_2)^2}
	\chi\left(\frac{\tau-E_0}{\lambda^2N_2}\right)
	\widehat h_0(\tau-|k|^2)
	\quad\text{or}\quad
	e^{-\pi|k-k_0|^2/(\lambda N_2)^2}
	\chi\left(\frac{\tau-E_0}{\lambda^2N_2}\right)
	\widehat h_S(\tau-|k|^2).
\end{equation*}
Since
\(\|\widehat{\zeta_T}\|_{L^1(\mathbb R)}
=\|\widehat\zeta\|_{L^1(\mathbb R)}\),
\begin{equation*}
	\sup_{k,\tau}|\widehat K_0(k,\tau)|
	\lesssim\|h_0\|_{L^1(\mathbb R)}
	\lesssim(\lambda N_2)^{-2},
	\qquad
	\sup_{k,\tau}|\widehat K_S(k,\tau)|
	\lesssim\|h_S\|_{L^1(\mathbb R)}
	\lesssim(\lambda N_2S)^{-1}.
\end{equation*}
Space-time Plancherel therefore gives
\begin{equation}
	\label{eq:lambda-short-time-L2}
	\|K_0*_{y,s}\|_{L_s^2L_y^2\to L_s^2L_y^2}
	\lesssim(\lambda N_2)^{-2},
	\qquad
	\|K_S*_{y,s}\|_{L_s^2L_y^2\to L_s^2L_y^2}
	\lesssim(\lambda N_2S)^{-1}.
\end{equation}

\medskip\noindent
\textbf{The pointwise bounds and interpolation.}
The Poisson summation calculation used for
\eqref{eq:rational-gaussian}, with \(q=1\), gives
\begin{equation*}
	|G(y,\varphi)|
	\lesssim
	(\lambda N_2S)^{d/2}
	\sum_{m\in\mathbb Z^d}
	e^{-cS^2|y+2\varphi k_0-m|^2},
	\qquad \varphi\in\operatorname{supp}h_S.
\end{equation*}
On \(\operatorname{supp}h_0\), the same formula holds
with \(S=\lambda N_2\).
Applying Lemma~\ref{lem:gaussian-line}, we further obtain
\begin{equation*}
	\int_{|\varphi|\leq2/(\lambda N_2S)}
	\sum_{m\in\mathbb Z^d}
	e^{-cS^2|y+2\varphi k_0-m|^2}\,d\varphi
	\lesssim
	\frac{1+|k_0|/(\lambda N_2S)}{|k_0|S}.
\end{equation*}
Here \(S\geq1\), as required by the lemma.
Using \(\|D^{\mathbb R}_{\lambda^2N_2}\|_{L^\infty}
\lesssim\lambda^2N_2\), we obtain
\begin{equation*}
	\|K_0\|_{L_s^\infty L_y^\infty}
	\lesssim
	\lambda^2N_2
	\left(
	(\lambda N_2)^{d-3}
	+\frac{(\lambda N_2)^{d-1}}{|k_0|}
	\right),
\end{equation*}
and
\begin{equation}
	\label{eq:lambda-short-time-Linfty}
	\|K_S\|_{L_s^\infty L_y^\infty}
	\lesssim
	\lambda^2N_2(\lambda N_2S)^{d/2-1}
	\left(\frac{\lambda N_2}{|k_0|}+\frac1S\right).
\end{equation}
Interpolating these bounds with
\eqref{eq:lambda-short-time-L2}, we find
\begin{equation*}
	\|K_0*_{y,s}\|_
	{L_s^{4/3}L_y^{4/3}\to L_s^4L_y^4}^{\,2}
	\lesssim
	\lambda\left(
	(\lambda N_2)^{d-4}
	+\frac{(\lambda N_2)^{d-2}}{|k_0|}
	\right),
\end{equation*}
and
\begin{equation*}
		\|K_S*_{y,s}\|_
		{L_s^{4/3}L_y^{4/3}\to L_s^4L_y^4}^{\,2}
		\lesssim
		\lambda\left[
		\frac{(\lambda N_2)^{d/2}}{|k_0|}
		S^{d/2-2}
		+(\lambda N_2)^{d/2-1}S^{d/2-3}
		\right]
	\end{equation*}
	\begin{equation}\label{eq:lambda-short-time-piece}
\lesssim
		\lambda\left(
		(\lambda N_2)^{d-4}
		+\frac{(\lambda N_2)^{d-2}}{|k_0|}
		\right).
	\end{equation}
The last inequality uses \(d\geq6\) and
\(1\leq S\leq\lambda N_2\).

\medskip\noindent
\textbf{Summation over \(S\).}
Summing \eqref{eq:lambda-short-time-piece} directly
would introduce \(1+\log(\lambda N_2)\).
To avoid a loss depending on \(\lambda\), we use the
decay of \(D^{\mathbb R}_{\lambda^2N_2}\) for the
smaller values of \(S\).

The number of scales satisfying
\(S\geq\lambda/(16N_2)\) is at most 
\begin{equation*}
	C\left(
	1+\log\frac{\lambda N_2}{\lambda/(16N_2)}
	\right)
	\lesssim1+\log N_2.
\end{equation*}
For \(S<\lambda/(16N_2)\), the support condition for
\(h_S\) gives
\begin{equation*}
	|\varphi|
	\geq\frac1{2\lambda N_2S}>8T.
\end{equation*}
Since \(\zeta_T(s)\neq0\) implies \(|s|\leq2T\),
we then have
\begin{equation*}
	|s-\varphi|\geq\frac34|\varphi|,
	\qquad
	\lambda^2N_2|s-\varphi|
	\gtrsim\frac{\lambda}{S}.
\end{equation*}
Consequently, for every \(A>0\),
\begin{equation*}
	|D^{\mathbb R}_{\lambda^2N_2}(s-\varphi)|
	\lesssim_A
	\lambda^2N_2\left(\frac S\lambda\right)^A.
\end{equation*}
Thus \eqref{eq:lambda-short-time-Linfty} gains the
factor \((S/\lambda)^A\), and interpolation gives
\begin{equation*}
	\|K_S*_{y,s}\|_
	{L_s^{4/3}L_y^{4/3}\to L_s^4L_y^4}
	\lesssim_A
	\left(\frac S\lambda\right)^{A/2}
	\left[
	\lambda\left(
	(\lambda N_2)^{d-4}
	+\frac{(\lambda N_2)^{d-2}}{|k_0|}
	\right)
	\right]^{1/2}.
\end{equation*}
The factors \((S/\lambda)^{A/2}\) have a bounded
geometric sum over these scales. Combining the two
ranges of \(S\), we obtain
\begin{equation*}
		\sum_S\|K_S*_{y,s}\|_
		{L_s^{4/3}L_y^{4/3}\to L_s^4L_y^4}
\lesssim
		(1+\log N_2)
		\left[
		\lambda\left(
		(\lambda N_2)^{d-4}
		+\frac{(\lambda N_2)^{d-2}}{|k_0|}
		\right)
		\right]^{1/2}.
	\end{equation*}
The same bound holds for the central contribution by the preceding estimate for \(K_0\).

\medskip\noindent
\textbf{The estimate for \(K_{\mathrm{tail}}\).}
By its definition,
\begin{equation*}
	h_{\mathrm{tail}}(\varphi)\neq0
	\quad\Longrightarrow\quad
	|\varphi|\geq
	\frac1{2^{-J}(\lambda N_2)^2}
	>\frac1{2\lambda N_2}.
\end{equation*}
On the other hand,
\(|s|\leq2T<1/(32\lambda N_2)\) on the support of
\(\zeta_T\). Hence \(|s-\varphi|\gtrsim|\varphi|\).
Using
\(|G(y,\varphi)|\lesssim(\lambda N_2)^d\) and the
decay of \(D^{\mathbb R}_{\lambda^2N_2}\), we get
\begin{equation*}
		|K_{\mathrm{tail}}(y,s)|
		\lesssim_A
		(\lambda N_2)^d
		\int_{|\varphi|\geq1/(2\lambda N_2)}
		\lambda^2N_2
		(1+\lambda^2N_2|\varphi|)^{-A-1}\,d\varphi
        \lesssim
		(\lambda N_2)^d\lambda^{-A}.
	\end{equation*}
Because this kernel is supported in \(|s|\leq2T\)
and \(|\mathbb T^d|=1\), Young's inequality gives
\begin{equation*}
	\|K_{\mathrm{tail}}*_{y,s}\|_
	{L_s^2L_y^2\to L_s^2L_y^2}
	\lesssim_A
	T(\lambda N_2)^d\lambda^{-A}.
\end{equation*}
Interpolating with the pointwise kernel bound,
\begin{equation*}
	\|K_{\mathrm{tail}}*_{y,s}\|_
	{L_s^{4/3}L_y^{4/3}\to L_s^4L_y^4}
	\lesssim_A
	T^{1/2}(\lambda N_2)^d\lambda^{-A}
	\lesssim
	\bigl[\lambda(\lambda N_2)^{d-4}\bigr]^{1/2}.
\end{equation*}
For the last inequality, take \(A>d+3\) and use
\(N_2<\lambda/64\); indeed, the ratio of the middle
expression to the last is 
\begin{equation*}
	\lambda^{(d+1)/2-A}N_2^{(d+4)/2}
	\lesssim\lambda^{d+5/2-A}\lesssim1.
\end{equation*}
Together with
\eqref{eq:lambda-short-time-decomposition}, this proves
\begin{equation}
	\label{eq:lambda-short-time-operator}
	\|(\zeta_TK)*_{y,s}\|_
	{L_s^{4/3}L_y^{4/3}\to L_s^4L_y^4}
	\lesssim
	(1+\log N_2)
	\left[
	\lambda\left(
	(\lambda N_2)^{d-4}
	+\frac{(\lambda N_2)^{d-2}}{|k_0|}
	\right)
	\right]^{1/2}.
\end{equation}

\medskip\noindent
\textbf{Conclusion of the $L^4$ estimate.}
The Fourier coefficients in the expansion of \(K\)
are nonnegative and bounded from below on
\(\operatorname{supp}a\). Therefore, for \(f\) supported
in \(\mathbb T^d\times I\), Cauchy--Schwarz gives
\begin{equation*}
    |\langle U_a,f\rangle_{L_s^2L_y^2}|^2
    \lesssim \|a\|_{\ell^2}^2
    \sum_{k\in\mathbb Z^d}
    e^{-\pi|k-k_0|^2/(\lambda N_2)^2}
    \chi\left(\frac{|k|^2-E_0}{\lambda^2N_2}\right)
    |\widehat f(k,|k|^2)|^2.
\end{equation*}
By the Fourier expansion of \(K\),
\begin{equation*}
    \sum_{k\in\mathbb Z^d}
    e^{-\pi|k-k_0|^2/(\lambda N_2)^2}
    \chi\left(\frac{|k|^2-E_0}{\lambda^2N_2}\right)
    |\widehat f(k,|k|^2)|^2
    =\langle K*_{y,s}f,f\rangle_{L_s^2L_y^2},
\end{equation*}
and
\begin{equation*}
    \langle K*_{y,s}f,f\rangle_{L_s^2L_y^2}
    =\langle(\zeta_TK)*_{y,s}f,f\rangle_{L_s^2L_y^2}.
\end{equation*}
The last equality is
\eqref{eq:lambda-time-cutoff-pairing}.
H\"older's inequality and
\eqref{eq:lambda-short-time-operator}, followed by
duality on \(I\times\mathbb T^d\), imply
\begin{equation*}
	\|U_a\|_{L_s^4(I;L_y^4(\mathbb T^d))}^4
	\lesssim
	(1+\log N_2)^2
	\lambda\left(
	(\lambda N_2)^{d-4}
	+\frac{(\lambda N_2)^{d-2}}{|k_0|}
	\right)\|a\|_{\ell^2}^4.
\end{equation*}
Since \(|k_0|\sim\lambda N_1\), absorbing
\((1+\log N_2)^2\) into \(N_2^\varepsilon\) proves
\eqref{eq:lambda-high-fourth} in this range.

		\end{proof}

	Next, we recover the estimate on the low-frequency strips.
		\begin{lemma}
			\label{lem:lambda-low-strip}
			Let \(d\geq5\), let \(v\) be a unit vector, and suppose
			that \(h\geq1\) and
			\begin{equation*}
				\operatorname{supp}b
				\subset
				\left\{
				l\in\mathbb Z^d:
				\bigl||l|-\lambda N_2\bigr|\leq C\lambda,
				\quad |l\cdot v-A|\leq h
				\right\}.
			\end{equation*}
			For every interval \(I\) of length \(T=\lambda^{-2}\),
			\begin{equation}
				\label{eq:lambda-low-strip}
				\|U_b\|_{L_s^4(I;L_y^4(\mathbb T^d))}^4
				\lesssim_\varepsilon
				N_2^\varepsilon
				\lambda(\lambda N_2)^{d-4}h
				\|b\|_{\ell^2}^4.
			\end{equation}
		\end{lemma}
		
		\begin{proof}
		We first adapt Lemma~\ref{lem:real-energy-interval} to
		$\lambda^{-1}$-separated frequencies.
		Let $r\ge1$, let $c$ be supported on a
		$\lambda^{-1}$-separated finite subset of $\mathbb R^D$, and define
		\begin{equation*}
			F(y',s):=
			\sum_{\xi\in\operatorname{supp}c}
			c_\xi e^{2\pi i(y'\cdot\xi+s|\xi|^2)}.
		\end{equation*}
		Partition the frequency support into caps $\theta$ of
		diameter $O(1)$ and set
		\begin{equation*}
			F_\theta(y',s):=
			\sum_{\xi\in\operatorname{supp}c\cap\theta}
			c_\xi e^{2\pi i(y'\cdot\xi+s|\xi|^2)}.
		\end{equation*}
		Separation implies
		$\#(\operatorname{supp}c\cap\theta)\lesssim\lambda^D$.
		Expanding the spatial mean gives
		\begin{equation*}
			(\mathcal M_{y'}|F_\theta|^4)(s)
			=
			\sum_\zeta
			\left|
			\sum_{\substack{\xi,\xi'\in\operatorname{supp}c\cap\theta\\
					\xi+\xi'=\zeta}}
			c_\xi c_{\xi'}
			e^{2\pi is(|\xi|^2+|\xi'|^2)}
			\right|^2.
		\end{equation*}
		For each $\zeta$, specifying $\xi$ determines $\xi'$.
		Thus Cauchy--Schwarz in the inner sum yields
		\begin{equation*}
			(\mathcal M_{y'}|F_\theta|^4)(s)
			\le
			\#(\operatorname{supp}c\cap\theta)
			\sum_{\xi,\xi'\in\operatorname{supp}c\cap\theta}
			|c_\xi|^2|c_{\xi'}|^2
			\lesssim
			\lambda^D
			\left(\sum_{\xi\in\operatorname{supp}c\cap\theta}
			|c_\xi|^2\right)^2.
		\end{equation*}
		
		We now apply the spatial averaging argument used for
		\eqref{eq:real-shell-mean} and \eqref{eq:real-ball-mean}.
		For a unit interval $J$ with center $s_J$, set
		\begin{equation*}
			w_J(s):=(1+|s-s_J|)^{-100(D+1)}.
		\end{equation*}
		Consequently, if
		$\operatorname{supp}c\subseteq
		\{\xi\in\mathbb R^D:\bigl||\xi|-r\bigr|\le C\}$, one can apply shell decoupling and Minkowski's inequality, which yield
		\begin{equation*}
			\left(
			\int_J(\mathcal M_{y'}|F|^4)(s)\,ds
			\right)^{1/4}
			\lesssim_\delta
			r^{(D-3)/4+\delta}
			\left[
			\sum_\theta
			\left(
			\int_{\mathbb R}w_J(s)
			(\mathcal M_{y'}|F_\theta|^4)(s)\,ds
			\right)^{1/2}
			\right]^{1/2}.
		\end{equation*}
		Since $\int_{\mathbb R}w_J(s)\,ds\lesssim1$,
		the preceding cap estimate implies
		\begin{equation*}
			\left[
			\sum_\theta
			\left(
			\int_{\mathbb R}w_J(s)
			(\mathcal M_{y'}|F_\theta|^4)(s)\,ds
			\right)^{1/2}
			\right]^{1/2}
			\lesssim
			\lambda^{D/4}
			\left(
			\sum_\theta
			\sum_{\xi\in\operatorname{supp}c\cap\theta}|c_\xi|^2
			\right)^{1/2}
			=
			\lambda^{D/4}\|c\|_{\ell^2}.
		\end{equation*}
		Similarly, for $\operatorname{supp}c\subseteq
		\{\xi\in\mathbb R^D:|\xi|\le Cr\}$, we can also obtain the corresponding estimate.
		
		Thus, one can conclude that
		\begin{equation*}
			\left(
			\int_J(\mathcal M_{y'}|F|^4)(s)\,ds
			\right)^{1/4}
			\lesssim_\delta
			\begin{cases}
				\lambda^{D/4}r^{(D-3)/4+\delta}\|c\|_{\ell^2},
				& \operatorname{supp}c\subseteq
				\{\xi:\bigl||\xi|-r\bigr|\le C\},\\
				\lambda^{D/4}r^{(D-2)/4+\delta}\|c\|_{\ell^2},
				& \operatorname{supp}c\subseteq B(0,Cr).
			\end{cases}
		\end{equation*}
		The bounds are uniform over unit intervals $J$,
		with constants allowed to depend on $D$ and the fixed $C$.
			
			Applying the energy decomposition and temporal
			summation from Lemma~\ref{lem:real-energy-interval},
			we conclude that, for \(D\geq4\), a
			\(\lambda^{-1}\)-separated set satisfying
			\begin{equation*}
				|\xi|\leq CN_2,
				\qquad
				|\xi|^2\text{ belongs to an interval of length }CN_2
			\end{equation*}
			obeys
			\begin{equation}
				\label{eq:lambda-dense-energy}
				\int_{\mathbb R}e^{-c s^2}
				(\mathcal M_{y'}|F|^4)(s)\,ds
				\lesssim_\varepsilon
				\lambda^D N_2^{D-3+\varepsilon}
				\|c\|_{\ell^2}^4.
			\end{equation}
			Indeed, the low-energy case contributes at most
			\(\lambda^DN_2^{(D-2)/2+\varepsilon}\).
			In the other case, the shell summation contributes
			\begin{equation*}
				\lambda^{D/4}
				\left(\frac{N_2}{r}\right)^{1/4}
				r^{(D-3)/4+\delta}
				\lesssim_\delta
				\lambda^{D/4}N_2^{(D-3)/4+\delta}.
			\end{equation*}
			Both estimates use \(D\geq4\). No power of \(\lambda\)
			is absorbed into the loss \(N_2^\varepsilon\).
			
			We now prove \eqref{eq:lambda-low-strip} for \(h=1\).
			After translating time, we may take
			\(I=[-T/2,T/2]\), whose characteristic function is
			bounded by a constant times \(e^{-\pi s^2/T^2}\).
			Write
			\begin{equation*}
				l=\xi_l+(A+\gamma_l)v,
				\qquad
				\xi_l\in v^\perp,
				\qquad |\gamma_l|\leq1,
			\end{equation*}
			and define
			\begin{equation*}
				F(y',s)
				:=
				\sum_l|b(l)|
				e^{2\pi i(y'\cdot\xi_l+s|\xi_l|^2)}.
			\end{equation*}
			Unlike the calculation in
			\eqref{eq:projection-fourth-moment}, the time integral
			does not impose an exact energy relation.
			For a quadruple satisfying \(l_1+l_2=l_3+l_4\), put
			\begin{equation*}
				\Omega_l
				:=|l_1|^2+|l_2|^2-|l_3|^2-|l_4|^2,
				\qquad
				\Omega_\xi
				:=|\xi_{l_1}|^2+|\xi_{l_2}|^2
				-|\xi_{l_3}|^2-|\xi_{l_4}|^2.
			\end{equation*}
			The terms involving \(A\) cancel as before, giving
			\begin{equation*}
				|\Omega_l-\Omega_\xi|
				=
				|\gamma_{l_1}^2+\gamma_{l_2}^2
				-\gamma_{l_3}^2-\gamma_{l_4}^2|
				\leq4.
			\end{equation*}
			Since \(T\leq1\),
			\begin{equation*}
				\Omega_l^2\geq\frac12\Omega_\xi^2-16,
				\qquad
				e^{-\pi T^2\Omega_l^2}
				\leq e^{16\pi}e^{-\pi T^2\Omega_\xi^2/2}.
			\end{equation*}
			Expanding the Gaussian-weighted $L^4$-norm, taking
			absolute values of the coefficients, and allowing all
			quadruples with
			\(\xi_{l_1}+\xi_{l_2}=\xi_{l_3}+\xi_{l_4}\), we obtain
			\begin{equation*}
		\|U_b\|_{L_s^4(I;L_y^4(\mathbb T^d))}^4
					\lesssim
					\int_{\mathbb R}e^{-\pi s^2/T^2}
					\int_{\mathbb T^d}|U_b(y,s)|^4\,dy\,ds
	\end{equation*}
	\begin{equation*}
                    \lesssim
					\int_{\mathbb R}e^{-2\pi s^2/T^2}
					(\mathcal M_{y'}|F|^4)(s)\,ds.
	\end{equation*}
			The second inequality follows from the Gaussian Fourier
			transforms
			\begin{equation*}
				\int_{\mathbb R}e^{-\pi s^2/T^2}e^{2\pi is\Omega}\,ds
				=Te^{-\pi T^2\Omega^2},
				\qquad
				\int_{\mathbb R}e^{-2\pi s^2/T^2}e^{2\pi is\Omega}\,ds
				=\frac{T}{\sqrt2}e^{-\pi T^2\Omega^2/2}.
			\end{equation*}
			
			The separation argument used after
			\eqref{eq:projection-fourth-moment} still applies:
			projected points at distance less than one come from
			integer points at distance less than \(\sqrt5\).
			Thus we may partition the sum into \(O_d(1)\) parts
			whose projected frequencies are \(1\)-separated.
			It suffices to estimate one such part.
			
			For nonempty support, \(|A|\lesssim\lambda N_2\), and 
			\begin{equation*}
				|\xi_l|\lesssim\lambda N_2,
				\qquad
				|\xi_l|^2
				=|l|^2-A^2-2A\gamma_l-\gamma_l^2
			\end{equation*}
			belongs to an interval of length \(O(\lambda^2N_2)\).
			Divide the projected frequencies by \(\lambda\) and
			set \(\tau:=\lambda^2s\), \(z:=\lambda y'\).
			The new frequencies are \(\lambda^{-1}\)-separated,
			lie in a ball of radius \(O(N_2)\), and have 
			energies in an interval of length \(O(N_2)\).
			The spatial mean is invariant under this spatial
			rescaling, while \(ds=\lambda^{-2}d\tau\).
			Applying \eqref{eq:lambda-dense-energy} with \(D=d-1\),
			\begin{equation*}
				\int_{\mathbb R}e^{-2\pi s^2/T^2}
				(\mathcal M_{y'}|F|^4)(s)\,ds
				\lesssim_\varepsilon
				\lambda^{-2}\lambda^{d-1}
				N_2^{d-4+\varepsilon}\|b\|_{\ell^2}^4
				=
				N_2^\varepsilon
				\lambda(\lambda N_2)^{d-4}\|b\|_{\ell^2}^4.
			\end{equation*}
			Summing the bounded number of parts proves the
			unit-width estimate.
			
			For $h\ge1$, partition the range of $l\cdot v$ into consecutive
			unit intervals and write $b=\sum_{j=1}^Jb_j$, with $J\lesssim h$.
			The spatial orthogonality argument proving
			\eqref{eq:strip-recombination} applies at each fixed time;
			integrating over $I$ therefore gives
			\begin{equation*}
				\|U_b\|_{L_s^4(I;L_y^4(\mathbb T^d))}^4
				\lesssim
				J\left(
				\sum_j\|U_{b_j}\|_{L_s^4(I;L_y^4(\mathbb T^d))}^2
				\right)^2.
			\end{equation*}
			Applying the unit-width estimate and using the orthogonality, we obtain
			\begin{equation*}
				\|U_b\|_{L_s^4(I;L_y^4(\mathbb T^d))}^4
				\lesssim_\varepsilon
				hN_2^\varepsilon\lambda(\lambda N_2)^{d-4}
				\left(\sum_j\|b_j\|_{\ell^2}^2\right)^2
				=
				hN_2^\varepsilon\lambda(\lambda N_2)^{d-4}
				\|b\|_{\ell^2}^4.
			\end{equation*}
			This proves the lemma.

		\end{proof}
		
		\begin{proof}[Proof of Theorem~\ref{thm:lambda-shell-strichartz}]
			We apply the localization and orthogonality arguments
			from Section~\ref{sec:framework},
			keeping the time scale \(T=\lambda^{-2}\).

			Partition the high-frequency coefficients into
			disjoint cubes of side comparable to \(\lambda N_2\)
			and energy intervals of length comparable to
			\(\lambda^2N_2\).
			The low-frequency support has spatial diameter
			\(O(\lambda N_2)\) and energy width
			\(O(\lambda^2N_2)\). The resulting products therefore have bounded overlap
			in spatial and temporal Fourier variables, exactly as
			in \eqref{eq:framework-high-sum}.
			
			For one high-frequency piece, choose \(k_0\) in its
			support and put \(v:=k_0/|k_0|\). Since
			\begin{equation*}
				2|k_0|\,v\cdot(k-k_0)
				=
				|k|^2-|k_0|^2-|k-k_0|^2,
			\end{equation*}
			this piece lies in a strip of width
			\begin{equation*}
				h:=1+\frac{\lambda N_2^2}{N_1},
			\end{equation*}
			up to a fixed constant. Here we used
			\(\lambda^2N_2\leq(\lambda N_2)^2\) and
			\(|k_0|\sim\lambda N_1\).
			
			Partition the low-frequency coefficients into strips
			of this width in the coordinate \(v\cdot l\).
			For the fixed high-frequency piece, the product
			supports have bounded overlap in the corresponding
			coordinate \(v\cdot(k+l)\), as in
			\eqref{eq:framework-low-sum}.
			The two $L^4$-norm estimates have the common
			coefficient
			\begin{equation*}
				N_2^\varepsilon
				\lambda(\lambda N_2)^{d-4}
				\left(1+\frac{\lambda N_2^2}{N_1}\right).
			\end{equation*}
			H\"older's inequality, followed by the same two
			orthogonality summations as for \(\lambda=1\), gives
			\begin{equation}
				\label{eq:lambda-integer-bilinear}
					\|U_{a_1}U_{a_2}\|_
					{L_s^2(I;L_y^2(\mathbb T^d))}^2
					\lesssim_\varepsilon
					N_2^\varepsilon
					\lambda(\lambda N_2)^{d-4}
					\left(1+\frac{\lambda N_2^2}{N_1}\right)
					\|a_1\|_{\ell^2}^2\|a_2\|_{\ell^2}^2.
			\end{equation}
			
			Finally, \eqref{eq:lambda-evolution-scaling} gives
			\begin{equation*}
					\|e^{it\Delta}\phi_1\,e^{it\Delta}\phi_2\|_
					{L_t^2([0,1];L_x^2(\mathbb T_\lambda^d))}^2
					=
					2\pi\lambda^{2-d}
					\int_{-1/(2\pi\lambda^2)}^0
					\int_{\mathbb T^d}
					|U_{a_1}(y,s)U_{a_2}(y,s)|^2\,dy\,ds.
			\end{equation*}
			The time interval is contained in one of length \(T\).
			Applying \eqref{eq:lambda-integer-bilinear} and using
			\begin{equation*}
				\lambda^{2-d}\lambda(\lambda N_2)^{d-4}
				\left(1+\frac{\lambda N_2^2}{N_1}\right)
				=
				\frac{N_2^{d-4}}{\lambda}
				+\frac{N_2^{d-2}}{N_1}
			\end{equation*}
			proves the theorem, after taking square roots and
			relabeling \(\varepsilon\).

		\end{proof}
		
		\begin{rem}
			The two terms in \eqref{eq:lambda-shell-strichartz}
			are comparable when \(N_1=\lambda N_2^2\).
			In particular, setting \(\lambda=1\) recovers the
			coefficient in
			Theorem~\ref{thm:improved-bilinear-strichartz}.
			The absence of a \(\lambda^\varepsilon\) loss is
			essential for the limiting argument below.
		\end{rem}
		
		\vspace{10pt}
		\section{The Euclidean estimate and the large-torus limit}
		\label{sec:euclidean-shell-strichartz}
		
		The second term in \eqref{eq:lambda-shell-strichartz}
		is the coefficient of the corresponding Euclidean
		estimate. We first prove that estimate directly.
		Unlike the periodic result, it holds over the whole
		time line and has no \(N_2^\varepsilon\) loss.
		
		\begin{thm}
			\label{thm:euclidean-shell-strichartz}
			Let \(d\geq3\), \(N_1\geq N_2\geq2\), and suppose that
			\(\phi_1,\phi_2\in L^2(\mathbb R^d)\) satisfy
			\begin{equation*}
				\operatorname{supp}\widehat{\phi_j}
				\subset
				\{\xi\in\mathbb R^d:\bigl||\xi|-N_j\bigr|\leq1\},
				\qquad j=1,2.
			\end{equation*}
			Then
			\begin{equation}
				\label{eq:euclidean-shell-strichartz}
				\|e^{it\Delta}\phi_1\,e^{it\Delta}\phi_2\|_
				{L_t^2(\mathbb R;L_x^2(\mathbb R^d))}
				\lesssim
				\frac{N_2^{\frac{d-2}{2}}}{N_1^{\frac{1}{2}}}
				\|\phi_1\|_{L_x^2(\mathbb R^d)}\|\phi_2\|_{L_x^2(\mathbb R^d)}.
			\end{equation}
		\end{thm}
		
		\begin{proof}
			A fixed change of time reduces the proof to the
			following evolutions
			\begin{equation*}
				u_j(x,t)
				:=
				\int_{\mathbb R^d}
				\widehat{\phi_j}(\xi)
				e^{2\pi i(x\cdot\xi+t|\xi|^2)}\,d\xi.
			\end{equation*}
			We treat separated and comparable frequencies
			differently.
			
			\medskip\noindent
			\textbf{Separated frequencies: \(N_1\geq4N_2\).}
			The space-time Fourier transform of the product is
			\begin{equation*}
				\widehat{u_1u_2}(\zeta,\tau)
				=
				\int_{\mathbb R^d}
				\widehat{\phi_1}(\zeta-\eta)
				\widehat{\phi_2}(\eta)
				\delta\bigl(\tau-|\zeta-\eta|^2-|\eta|^2\bigr)\,d\eta,
			\end{equation*}
			where the delta integral is understood through the
			coarea formula. For $\zeta\in\operatorname{supp}\widehat\phi_1
			+\operatorname{supp}\widehat\phi_2$, we have
			\begin{equation*}
				|\zeta|
				\geq N_1-N_2-2
				\geq N_1/2.
			\end{equation*}
			We estimate the total coarea measure by discarding the
			high-frequency shell restriction and keeping
			\(N_2-1\leq|\eta|\leq N_2+1\).
			Writing \(\eta=\rho\omega\), with
			\(\omega\in\mathbb S^{d-1}\), gives
			\begin{equation*}
		\int_{N_2-1\leq|\eta|\leq N_2+1}
					\delta\bigl(\tau-|\zeta-\eta|^2-|\eta|^2\bigr)
					\,d\eta
=
					\int_{N_2-1}^{N_2+1}\rho^{d-1}
					\int_{\mathbb S^{d-1}}
					\delta\bigl(
					\tau-|\zeta|^2-2\rho^2+2\rho\zeta\cdot\omega
					\bigr)\,d\sigma(\omega)\,d\rho.
	\end{equation*}
			Using \(z=\omega\cdot\zeta/|\zeta|\), the inner integral
			equals
			\begin{equation*}
				|\mathbb S^{d-2}|
				\int_{-1}^1
				\delta\bigl(
				\tau-|\zeta|^2-2\rho^2+2\rho|\zeta|z
				\bigr)
				(1-z^2)^{(d-3)/2}\,dz
				\lesssim\frac1{\rho|\zeta|}.
			\end{equation*}
			Here \(d\geq3\) ensures that the angular weight is
			bounded. Therefore,
			\begin{equation*}
				\int_{N_2-1\leq|\eta|\leq N_2+1}
				\delta\bigl(\tau-|\zeta-\eta|^2-|\eta|^2\bigr)\,d\eta
				\lesssim
				\frac1{N_1}
				\int_{N_2-1}^{N_2+1}\rho^{d-2}\,d\rho
				\lesssim\frac{N_2^{d-2}}{N_1}.
			\end{equation*}
			Cauchy--Schwarz with respect to this coarea measure
			now yields
			\begin{equation*}
				|\widehat{u_1u_2}(\zeta,\tau)|^2
				\lesssim
				\frac{N_2^{d-2}}{N_1}
				\int_{\mathbb R^d}
				|\widehat{\phi_1}(\zeta-\eta)|^2
				|\widehat{\phi_2}(\eta)|^2
				\delta\bigl(\tau-|\zeta-\eta|^2-|\eta|^2\bigr)\,d\eta.
			\end{equation*}
			Integrating in \((\zeta,\tau)\) and applying Plancherel,
			\begin{equation*}
				\|u_1u_2\|_{L_t^2L_x^2}^2
				\lesssim
				\frac{N_2^{d-2}}{N_1}
				\|\phi_1\|_{L_x^2}^2\|\phi_2\|_{L_x^2}^2.
			\end{equation*}
			
			\medskip\noindent
			\textbf{Comparable frequencies: \(N_2\leq N_1<4N_2\).}
			We prove the linear shell estimate
			\begin{equation}
				\label{eq:euclidean-linear-shell-fourth}
				\|u_j\|_{L_t^4(\mathbb R;L_x^4(\mathbb R^d))}
				\lesssim
				N_j^{(d-3)/4}\|\phi_j\|_{L_x^2},
				\qquad j=1,2.
			\end{equation}
			The Tomas--Stein extension theorem, followed when
			necessary by interpolation with \(L^\infty\), gives
			\begin{equation*}
				\left\|
				\int_{\mathbb S^{d-1}}
				h(\omega)e^{2\pi ix\cdot\omega}\,d\sigma(\omega)
				\right\|_{L_x^4(\mathbb R^d)}
				\lesssim
				\|h\|_{L^2(\mathbb S^{d-1})},
				\qquad d\geq3,
			\end{equation*}
			because \(4\geq2(d+1)/(d-1)\).
			
			Let \(\widetilde u_j(x,\tau)\) denote the Fourier
			transform in time only. Polar coordinates give
			\begin{equation*}
				\widetilde u_j(x,\tau)
				=
				\frac12\tau^{(d-2)/2}
				\int_{\mathbb S^{d-1}}
				\widehat{\phi_j}(\sqrt\tau\,\omega)
				e^{2\pi i\sqrt\tau\,x\cdot\omega}\,d\sigma(\omega),
				\qquad \tau>0.
			\end{equation*}
			Its temporal support lies in
			\([(N_j-1)^2,(N_j+1)^2]\), an interval of length
			\(4N_j\). The sphere estimate and spatial scaling imply 
			\begin{equation*}
		\int_{\mathbb R}
					\|\widetilde u_j(\cdot,\tau)\|_{L_x^4}^2\,d\tau
					\lesssim
					\int_{N_j-1}^{N_j+1}
					\rho^{3d/2-3}
					\|\widehat{\phi_j}(\rho\,\cdot)\|_
					{L^2(\mathbb S^{d-1})}^2\,d\rho
            \lesssim
					N_j^{d/2-2}\|\phi_j\|_{L_x^2}^2.
	\end{equation*}
			Bernstein, Plancherel, and Minkowski's
			inequality therefore give
			\begin{equation*}
		\|u_j\|_{L_t^4L_x^4}
					=\|u_j\|_{L_x^4L_t^4}
                   \lesssim
					N_j^{1/4}\|u_j\|_{L_x^4L_t^2}
					=
					N_j^{1/4}\|\widetilde u_j\|_{L_x^4L_\tau^2}
	\end{equation*}
	\begin{equation*}
                  \lesssim
					N_j^{1/4}
					\left(
					\int_{\mathbb R}
					\|\widetilde u_j(\cdot,\tau)\|_{L_x^4}^2
					\,d\tau
					\right)^{1/2}
                  \lesssim
					N_j^{(d-3)/4}\|\phi_j\|_{L_x^2}.
	\end{equation*}
			This proves \eqref{eq:euclidean-linear-shell-fourth}.
			Since \(N_1\sim N_2\), H\"older's inequality gives
			\begin{equation*}
				\|u_1u_2\|_{L_t^2L_x^2}
				\lesssim
				(N_1N_2)^{(d-3)/4}
				\|\phi_1\|_{L_x^2}\|\phi_2\|_{L_x^2}
				\lesssim
				\left(\frac{N_2^{d-2}}{N_1}\right)^{1/2}
				\|\phi_1\|_{L_x^2}\|\phi_2\|_{L_x^2}.
			\end{equation*}
			Together with the separated-frequency estimate, this
			proves the theorem.

		\end{proof}
		
		\medskip\noindent
		\textbf{Passage from expanding tori to Euclidean space.}
		We finally explain precisely how
		\eqref{eq:lambda-shell-strichartz} is consistent with
		\eqref{eq:euclidean-shell-strichartz}.
		Fix \(N_1,N_2\), and first take Euclidean data with
		smooth Fourier transforms compactly supported in the
		interiors of the prescribed shells. Define
		\begin{equation*}
			\phi_{j,\lambda}(x)
			:=
			\lambda^{-d}
			\sum_{k\in\mathbb Z^d}
			\widehat{\phi_j}(k/\lambda)
			e^{2\pi ik\cdot x/\lambda}.
		\end{equation*}
		These functions satisfy the same shell support
		conditions on \(\mathbb T_\lambda^d\), and Parseval gives
		\begin{equation*}
			\|\phi_{j,\lambda}\|_
			{L_x^2(\mathbb T_\lambda^d)}^2
			=
			\lambda^{-d}
			\sum_{k\in\mathbb Z^d}
			|\widehat{\phi_j}(k/\lambda)|^2
			\longrightarrow
			\|\phi_j\|_{L_x^2(\mathbb R^d)}^2.
		\end{equation*}
		Likewise,
		\begin{equation*}
			e^{it\Delta}\phi_{j,\lambda}(x)
			=
			\lambda^{-d}
			\sum_{k\in\mathbb Z^d}
			\widehat{\phi_j}(k/\lambda)
			e^{2\pi ix\cdot k/\lambda-it|k/\lambda|^2}
			\longrightarrow
			e^{it\Delta}\phi_j(x)
		\end{equation*}
		uniformly on compact subsets of
		\(\mathbb R^d\times[0,1]\), by convergence of Riemann
		sums.
		
		Represent \(\mathbb T_\lambda^d\) by the centered cube
		\([-\lambda/2,\lambda/2)^d\).
		Every fixed compact spatial set is contained in this
		cube for sufficiently large \(\lambda\). Applying
		\eqref{eq:lambda-shell-strichartz}, passing to the
		limit on such a compact set, and then increasing the
		sets to \(\mathbb R^d\), we obtain
		\begin{equation*}
			\|e^{it\Delta}\phi_1\,e^{it\Delta}\phi_2\|_
			{L_t^2([0,1];L_x^2(\mathbb R^d))}
			\lesssim_\varepsilon
			N_2^\varepsilon
			\left(\frac{N_2^{d-2}}{N_1}\right)^{1/2}
			\|\phi_1\|_{L_x^2(\mathbb R^d)}\|\phi_2\|_{L_x^2(\mathbb R^d)}.
		\end{equation*}
		The extension to general shell-supported \(L^2\) data
		follows by density.
		
		Thus the periodic contribution
		\(N_2^{d-4}/\lambda\) vanishes as \(\lambda\to\infty\),
		while the remaining frequency coefficient agrees with
		the Euclidean estimate:
		\begin{equation*}
			\lim_{\lambda\to\infty}
			\left(
			\frac{N_2^{d-4}}{\lambda}
			+\frac{N_2^{d-2}}{N_1}
			\right)^{1/2}
			=
			\left(\frac{N_2^{d-2}}{N_1}\right)^{1/2}.
		\end{equation*}
		The limiting argument alone retains \(N_2^\varepsilon\)
		and the bounded time interval. The direct proof of
		Theorem~\ref{thm:euclidean-shell-strichartz} removes
		that loss and establishes the estimate on the whole
		time line.

		\vspace{10pt}
		\section{A simpler proof of the bilinear decoupling in \cite{fan2018bilinear}}
		\label{sec:fsww-proof}
		
		In this appendix, we give a simpler proof of the bilinear decoupling estimate of Fan--Staffilani--Wang--Wilson \cite{fan2018bilinear}. The main ingredients are the familiar cap--plate decomposition, H\"older's inequality, and linear decoupling. We first consider the case $\lambda=1$ and then discuss the general $\lambda$-version.
		
		We will repeatedly use the following simple observation. If $\tau$ is a $\left(v,\delta\right)$-plate whose shorter side is parallel to the $e_d$-direction and $f$ is supported in $\tau$, then Lemma \ref{lifting} and the Bourgain--Demeter $\ell^2$-decoupling theorem in dimension $d-1$ give
		\begin{equation}\label{FSWW-linear}
			\|Ef\|_{L^4(\omega_{B_{N_1^2}})}
			\lesssim_\epsilon
			N_2^{\frac{d-3}{4}+\epsilon}
			\left(
			\sum_{\substack{\theta\subseteq\tau\\|\theta|=1/N_1}}
			\|Ef_\theta\|_{L^4(\omega_{B_{N_1^2}})}^2
			\right)^{1/2}.
		\end{equation}
		
		\begin{prop}\label{FSWW-simple}
			Let $d\ge3$, $N_1\ge N_2\gg1$, and suppose that
			\begin{equation*}
				\operatorname{supp}f_1\subseteq\left\{\xi\in\R^d:|\xi|\sim 1\right\},
				\qquad
				\operatorname{supp}f_2\subseteq
				\left\{\xi\in\R^d:|\xi|\sim\frac{N_2}{N_1}\right\}.
			\end{equation*}
			Then, for any $\epsilon>0$,
			\begin{equation}\label{FSWW-simple-estimate}
				\|Ef_1Ef_2\|_{L^2(\omega_{B_{N_1^2}})}
				\lesssim_\epsilon
				N_2^\epsilon
				\left(
				N_2^{d-3}+\frac{N_2^{d-1}}{N_1}
				\right)^{1/2}
				\prod_{j=1}^2
				\left(
				\sum_{|\theta|=1/N_1}
				\|Ef_{j,\theta}\|_{L^4(\omega_{B_{N_1^2}})}^2
				\right)^{1/2}.
			\end{equation}
		\end{prop}
		
		\begin{proof}
			By the $L^2$-orthogonality, it suffices to localize $f_1$ and $f_2$ to $v$-caps centered near $e_d$ and $0$, respectively.
			
			We first consider $N_1\ge N_2^2$. Since $v^2\le \delta$, Remark~\ref{84} allows us to decompose directly into $\left(v,\delta\right)$-plates:
			\begin{equation*}
				\|Ef_1Ef_2\|_{L^2(\omega_{B_{N_1^2}})}^2
				\lesssim
				\sum_{\tau_1,\tau_2}
				\|Ef_{1,\tau_1}Ef_{2,\tau_2}\|_{L^2(\omega_{B_{N_1^2}})}^2.
			\end{equation*}
			For every fixed plate pair, H\"older's inequality and \eqref{FSWW-linear} give
			\begin{equation*}
				\|Ef_{1,\tau_1}Ef_{2,\tau_2}\|_{L^2(\omega_{B_{N_1^2}})}
				\lesssim_\epsilon
				N_2^{\frac{d-3}{2}+\epsilon}
				\prod_{j=1}^2
				\left(
				\sum_{\substack{\theta\subseteq\tau_j\\|\theta|=1/N_1}}
				\|Ef_{j,\theta}\|_{L^4(\omega_{B_{N_1^2}})}^2
				\right)^{1/2}.
			\end{equation*}
			Squaring and summing over $(\tau_1,\tau_2)$ yields the first term in \eqref{FSWW-simple-estimate}.
			
			We next assume $N_1\le N_2^2$. In this case, the lossless cap--plate decomposition gives $(v,v^2)$-plates. For a fixed pair $(\tau_1,\tau_2)$, decompose each $\tau_j$ further into $\left(v,\delta\right)$-plates $\widetilde{\tau}_j$. The number of such subplates is
			\begin{equation*}
				H\sim\frac{v^2}{\delta}=\frac{N_2^2}{N_1}.
			\end{equation*}
			Applying Cauchy--Schwarz in one factor and $L^2$-orthogonality in the other gives
			\begin{equation*}
				\|Ef_{1,\tau_1}Ef_{2,\tau_2}\|_{L^2(\omega_{B_{N_1^2}})}^2
				\lesssim
				H
				\sum_{\substack{\widetilde{\tau}_1\subseteq\tau_1\\
						\widetilde{\tau}_2\subseteq\tau_2}}
				\|Ef_{1,\widetilde{\tau}_1}Ef_{2,\widetilde{\tau}_2}\|_{L^2(\omega_{B_{N_1^2}})}^2.
			\end{equation*}
			Applying H\"older's inequality and \eqref{FSWW-linear} to every finer plate pair, and then summing over all the fine and coarse plates, gives
			\begin{equation*}
				\|Ef_1Ef_2\|_{L^2(\omega_{B_{N_1^2}})}
				\lesssim_\epsilon
				H^{1/2}N_2^{\frac{d-3}{2}+\epsilon}
				\prod_{j=1}^2
				\left(
				\sum_{|\theta|=1/N_1}
				\|Ef_{j,\theta}\|_{L^4(\omega_{B_{N_1^2}})}^2
				\right)^{1/2}.
			\end{equation*}
			Since $H^{1/2}N_2^{(d-3)/2}=N_2^{(d-1)/2}N_1^{-1/2}$, this gives the second term in \eqref{FSWW-simple-estimate}.

		\end{proof}
		
	We now record the general $\lambda$-version. Set
	\begin{equation*}
		\Omega_\lambda:=[0,N_1^2]\times[0,(\lambda N_1)^2]^d,
		\qquad \lambda\ge1,
	\end{equation*}
	where the first coordinate is time.
	Suppose that $f$ is supported in a
	$(v,(\lambda N_1)^{-1})$-plate $\tau$ whose short direction
	is parallel to $e_d$.
	First partition its $\xi'$-projection into cubes of side
	length $N_1^{-1}$, and denote the corresponding pieces by $\sigma$.
	Applying $(d-1)$-dimensional Bourgain--Demeter decoupling
	and Lemma~\ref{lifting}, followed by weighted localization
	to $\Omega_\lambda$, gives
	\begin{equation*}
		\|Ef\|_{L^4(\omega_{\Omega_\lambda})}
		\lesssim_\epsilon
		N_2^{\frac{d-3}{4}+\epsilon}
		\left(
		\sum_\sigma
		\|Ef_\sigma\|_{L^4(\omega_{\Omega_\lambda})}^2
		\right)^{1/2}.
	\end{equation*}
	Within each $\sigma$, subdivide each of the $d-1$ tangential
	coordinates into $O(\lambda)$ consecutive intervals of length
	$(\lambda N_1)^{-1}$.
	Applying the estimate
	\eqref{A-recombination} successively in these coordinates yields
	\begin{equation*}
		\|Ef_\sigma\|_{L^4(\omega_{\Omega_\lambda})}
		\lesssim
		\lambda^{\frac{d-1}{4}}
		\left(
		\sum_{\theta\subseteq\sigma}
		\|Ef_\theta\|_{L^4(\omega_{\Omega_\lambda})}^2
		\right)^{1/2}.
	\end{equation*}
	Combining these estimates, we obtain
	\begin{equation}\label{FSWW-linear-lambda}
		\|Ef\|_{L^4(\omega_{\Omega_\lambda})}
		\lesssim_\epsilon
		N_2^\epsilon
		\lambda^{\frac{d-1}{4}}
		N_2^{\frac{d-3}{4}}
		\left(
		\sum_{\substack{\theta\subseteq\tau\\|\theta|=1/(\lambda N_1)}}
		\|Ef_\theta\|_{L^4(\omega_{\Omega_\lambda})}^2
		\right)^{1/2}.
	\end{equation}
		
		\begin{prop}\label{FSWW-simple-lambda}
			Under the assumptions of Proposition~\ref{FSWW-simple}, for every $\lambda\ge1$,
			\begin{equation}\label{FSWW-general-lambda}
				\|Ef_1Ef_2\|_{L^2(\omega_{\Omega_\lambda})}
				\lesssim_\epsilon
				N_2^\epsilon
				\lambda^{d/2}
				\left(
				\frac{N_2^{d-3}}{\lambda}
				+
				\frac{N_2^{d-1}}{N_1}
				\right)^{1/2}
				\prod_{j=1}^2
				\left(
				\sum_{|\theta|=1/(\lambda N_1)}
				\|Ef_{j,\theta}\|_{L^4(\omega_{\Omega_\lambda})}^2
				\right)^{1/2}.
			\end{equation}
		\end{prop}
		
		\begin{proof}
			As before, we first localize both factors to $v$-caps. If
			\begin{equation*}
				1\le\lambda\le\frac{N_1}{N_2^2},
			\end{equation*}
			then $v^2\le(\lambda N_1)^{-1}$, and the lossless cap--plate decomposition applies directly to $\left(v,(\lambda N_1)^{-1}\right)$-plates. H\"older's inequality and \eqref{FSWW-linear-lambda} give
			\begin{equation*}
				\|Ef_1Ef_2\|_{L^2(\omega_{\Omega_\lambda})}
				\lesssim_\epsilon
				N_2^\epsilon
				\lambda^{\frac{d-1}{2}}
				N_2^{\frac{d-3}{2}}
				\prod_{j=1}^2
				\left(
				\sum_{|\theta|=1/(\lambda N_1)}
				\|Ef_{j,\theta}\|_{L^4(\omega_{\Omega_\lambda})}^2
				\right)^{1/2}.
			\end{equation*}
			This is exactly the first term on the right-hand side of \eqref{FSWW-general-lambda}.
			
			If
			\begin{equation*}
				\lambda\ge\frac{N_1}{N_2^2},
			\end{equation*}
			we first decompose losslessly into $(v,v^2)$-plates and then further into $\left(v,(\lambda N_1)^{-1}\right)$-plates. The number of fine plates contained in a coarse plate is
			\begin{equation*}
				H_\lambda\sim \frac{v^2}{(\lambda N_1)^{-1}}
				=\frac{\lambda N_2^2}{N_1}.
			\end{equation*}
			Repeating the preceding Cauchy--Schwarz, orthogonality, and H\"older argument gives the additional factor $H_\lambda^{1/2}$. Thus,
			\begin{equation*}
				\|Ef_1Ef_2\|_{L^2(\omega_{\Omega_\lambda})}
				\lesssim_\epsilon
				N_2^\epsilon
				H_\lambda^{1/2}
				\lambda^{\frac{d-1}{2}}
				N_2^{\frac{d-3}{2}}
				\prod_{j=1}^2
				\left(
				\sum_{|\theta|=1/(\lambda N_1)}
				\|Ef_{j,\theta}\|_{L^4(\omega_{\Omega_\lambda})}^2
				\right)^{1/2}.
			\end{equation*}
			Since
			\begin{equation*}
				H_\lambda^{1/2}\lambda^{\frac{d-1}{2}}N_2^{\frac{d-3}{2}}
				=
				\lambda^{d/2}\frac{N_2^{\frac{d-1}{2}}}{N_1^{1/2}},
			\end{equation*}
			we obtain the second term in \eqref{FSWW-general-lambda}. This completes the proof.

		\end{proof}
		
		\begin{rem}
			Compared with \cite{fan2018bilinear}, the above argument substantially simplifies the high-dimensional case $d\ge3$. In particular, after the cap--plate decomposition of Lemma~5.1 in \cite{fan2018bilinear}, we directly apply H\"older's inequality and $(d-1)$-dimensional linear decoupling. Thus, the auxiliary estimates Lemmas~3.2--3.6, Lemma~5.5, Corollary~5.7 of \cite{fan2018bilinear} are not needed.
			
			Moreover, this argument makes the transition threshold transparent. The plates produced by the cap--plate decomposition have
			thickness $v^2=N_2^2/N_1^2$ in the $e_d$-direction, while the desired scale is $(\lambda N_1)^{-1}$. Hence the transition occurs exactly at
			\begin{equation*}
				v^2=\frac{1}{\lambda N_1}
				\qquad\Longleftrightarrow\qquad
				\lambda=\frac{N_1}{N_2^2}.
			\end{equation*}
			Below this threshold no further decomposition is needed, while above it one must further subdivide the $(v,v^2)$-plates. In particular, for $\lambda=1$, the threshold is precisely $N_1=N_2^2$.
		\end{rem}
		
		\clearpage
		\section*{Acknowledgement}
		The author is grateful to Prof. Alex Cohen for helpful suggestions and discussions, as well as for his generous support of the author's travel to the ICM and PHADE, which provided valuable opportunities to communicate with other researchers. The author would also like to particularly thank Yixuan Pang and Ruyi Wang for many helpful discussions.
		
		\section*{Conflict of interest statement}
		The author does not have any possible conflict of interest.
		
		\section*{Data availability statement}
		The manuscript has no associated data.
		\bigskip
		\bigskip

		\bibliographystyle{alpha}
		\bibliography{Shell}

	\end{document}